\documentclass[11pt]{amsart}
\usepackage{latexsym,amsfonts,amssymb,amsthm,amsmath,mathrsfs,color,amscd,graphicx,fullpage,parskip,hyperref,bm,bbm}

\usepackage{comment}
\usepackage{commath,mathtools,setspace}
\usepackage{paralist}
\usepackage{extarrows}

\excludecomment{comment} 

\newcommand{\s}[1]{{\mathcal #1}}

\newcommand{\bb}[1]{{\mathbb #1}}

\newcommand{\ip}[2]{\left\langle #1,#2 \right\rangle}

\newcommand{\floor}[1]{\left\lfloor #1 \right\rfloor}

\DeclareMathOperator{\Lip}{Lip}

\newtheorem{theorem}{Theorem} 
\newtheorem{corollary}[theorem]{Corollary}

\newtheorem{lemma}[theorem]{Lemma}
\newtheorem{proposition}[theorem]{Proposition}

\newtheorem{remark}[theorem]{Remark}

\numberwithin{equation}{section}
\numberwithin{theorem}{section}

\begin{document}

	\title{Parameter sensitivity analysis for mean field games of production}

	\author{P. Jameson Graber}
	\thanks{Jameson Graber gratefully acknowledges support from the National Science Foundation through NSF CAREER Award 2045027 and NSF Grant DMS-1905449.
	Marcus Laurel is also thankful to be partially supported by NSF Grant DMS-1905449.}
	\address{J. Graber: Baylor University, Department of Mathematics;
		Sid Richardson Building\\
		1410 S.~4th Street\\
		Waco, TX 76706\\
		Tel.: +1-254-710-3480 \\
		Fax: +1-254-710-3569 
	}
	\email{Jameson\_Graber@baylor.edu}
	
	\author{Marcus Laurel}
	\address{M. Laurel: Baylor University, Department of Mathematics;
		Sid Richardson Building\\
		1410 S.~4th Street\\
		Waco, TX 76706
	}
	\email{Marcus\_Laurel@baylor.edu}

	\date{\today}   
	
	\begin{abstract}
		We study a mean field game system introduced by Chan and Sircar (\emph{AMO}, 2015) to model production of an exhaustible resource.
		In particular, we study the sensitivity of the solution with respect to a parameter $\varepsilon$, which measures the degree to which producers are interchangeable.
		We prove that on some interval $[0,\varepsilon_0]$, where $\varepsilon_0 > 0$, the solution is infinitely differentiable with respect to $\varepsilon$.
		The result is based on a set of new a priori estimates for forward-backward systems of linear partial differential equations.
	\end{abstract}
	
	
	\maketitle
	
	
\section{Introduction} \label{sec:intro}

Consider the following system of nonlinear partial differential equations:
\begin{equation}\label{u-msystem}
	\begin{cases}
		\begin{array}{lll}
			(i)&\displaystyle u_t+\frac{\sigma^2}{2}u_{xx}-ru+F(t,u_x,m;\varepsilon)^2=0, \text{\hspace{2cm}} &0\leq x< \infty,\,\,0\leq t\leq T\\
			(ii)&\displaystyle m_t-\frac{\sigma^2}{2}m_{xx}-\left[F(t,u_x,m;\varepsilon)m\right]_x=0, &0\leq x< \infty,\,\,0\leq t\leq T\\
			(iii)&\displaystyle m(x,0)=m_0(x),\text{\hspace{.2cm}}u(x,T)=u_T(x), &0\leq x< \infty\\
			(iv)&\displaystyle u(0,t)=m(0,t)=0, &0\leq t\leq T
		\end{array}
	\end{cases}
\end{equation}
where the data consist of a parameter $\varepsilon\geq0$, a time horizon $T > 0$, a smooth probability density $m_0(x)$, a smooth function $u_T(x)$, and two positive constants $\sigma$ and $r$, and where
$F(t,u_x,m;\varepsilon)$ is given by
\begin{equation}\label{def:F}
	F(t,u_x,m;\varepsilon):=\frac{1}{2}\left(\frac{2}{2+\varepsilon\xi(t)}+\frac{\varepsilon\xi(t)}{2+\varepsilon\xi(t)}\int_0^\infty u_x(x,t)m(x,t)\dif x-u_x(x,t)\right)
\end{equation}
for some fixed smooth function $\xi:[0,T]\to[0,1]$ such that $\xi(T)=0$.
We assume that $m_0$ and $u_T$ satisfy zero- and first-order compatibility conditions so that solutions of \eqref{u-msystem} are classical (see \cite{graber2021nonlocal}).

System \eqref{u-msystem} models a particular kind of \emph{mean field game} for the production of exhaustible resources, as proposed by Chan and Sircar in \cite{chan2015bertrand}, cf.~\cite{chan2017fracking,gueant2011mean}.
In this game, players control the rate at which they will sell from their current stock, and they must leave the game once that stock goes to zero.
The optimal strategies are determined by the market price.
In equilibrium, the market price is determined by the average of all the players' strategies.
When the demand depends linearly on the production rate (or on the price offered), one can determine the equilibrium by solving System \eqref{u-msystem}; in particular, if $(u,m)$ is the solution, then $F(t,u_x,m;\varepsilon)$ gives the equilibrium rate of production and $m(x,t)$ gives the density of players whose remaining stock is $x$ at time $t$.
The parameter $\varepsilon$ measures the degree to which firms compete.
If $\varepsilon = 0$, every firm is a monopolist, while if $\varepsilon$ is very large, all firms are nearly interchangeable in the eyes of consumers.

Existence and uniqueness of solutions to System \eqref{u-msystem} has been established under the assumptions we make in this article, thanks to the results found in \cite{graber2018existence,graber2020commodities,graber2021nonlocal}, cf.~\cite{graber2018variational,graber2016linear}.
In \cite{chan2015bertrand}, Chan and Sircar proposed a numerical method to solving \eqref{u-msystem} that consists of assuming the following Taylor expansion:
\begin{equation}
	\begin{split}
		u(x,t) &= u^{(0)}(x,t) + \varepsilon u^{(1)}(x,t) + \frac{\varepsilon^2}{2}u^{(2)}(x,t) + \cdots,\\
		m(x,t) &= m^{(0)}(x,t) + \varepsilon m^{(1)}(x,t) + \frac{\varepsilon^2}{2}m^{(2)}(x,t) + \cdots.
	\end{split}
\end{equation}
Formally, $u^{(k)}$ and $m^{(k)}$ can be derived by differentiating System \eqref{u-msystem} $k$ times with respect to $\varepsilon$, letting $\varepsilon = 0$, and then solving.
Notice that setting $\varepsilon = 0$ decouples the system.
Thus $u^{(k)}$ and $m^{(k)}$ are computed by solving two equations separately,  avoiding the computational difficulties arising from the forward/backward-in-time coupling.

To establish rigorously the accuracy of this numerical method, one must answer the question: is the solution to System \eqref{u-msystem} differentiable with respect to $\varepsilon$?
In the present article, we provide an affirmative answer to this question under generic assumptions on the data.
In particular, we establish that there exists $\varepsilon_0>0$ small enough such that for every $k\in\bb{N}$,  the solution to System \eqref{u-msystem} is $k$ times differentiable with respect to $\varepsilon$ for  $0 \leq \varepsilon \leq \varepsilon_0$.
See Theorem \ref{thm:ExistenceUniquenessEtc_un/mn} below.
More precisely, we establish that the difference between the solution and its $k$th order Taylor expansion is bounded by $C_k\varepsilon^{k+1}$ for some constant $C_k$.
On the other hand, the constants $C_k$ that appear in our analysis blow up very fast as $k \to \infty$, and for this reason we conjecture that the infinite series expansion does not in fact converge (see Remark \ref{rem:constants}).

While the numerical method of Chan and Sircar gives a practical application, our work is also motivated by the theory of forward-backward linear systems of equations, which turn out to be of fundamental importance in mean field game theory.
In their seminal work \cite{cardaliaguet2019master}, Cardaliaguet et al.~show that one can established well-posedness of the \emph{master equation} by the following steps: (1) solve the mean field game system, (2) linearize the system by formally differentiating with respect to the measure variable $m$, and, most crucially, (3) proving \emph{estimates on solutions to linearized systems}, which have as a corollary that the solution from step 1 is differentiable with respect to the initial measure.
We do not study the master equation in this article.
However, our sensitivity analysis follows exactly the same steps, where instead of linearizing with respect to an infinite-dimensional measure variable, we differentiate with respect to $\varepsilon$.
At an abstract level, the system of linear equations we study has the form
\begin{equation}\label{abstract_system}
	\begin{cases}
		\begin{array}{lll}
			(i)&\displaystyle w_t+\frac{\sigma^2}{2}w_{xx}-rw+\Psi+F^{(0)}(\varepsilon)\big(G(t,w_x,\mu;\varepsilon)-w_x\big)=0, \text{\hspace{.5cm}} &0\leq x< \infty,\,\,0\leq t\leq T\vspace{4pt}\\
			(ii)&\displaystyle \mu_t-\frac{\sigma^2}{2}\mu_{xx}-\big(F^{(0)}\mu\big)_x=\Big[\Phi+\frac{1}{2}\big(G(t,w_x,\mu;\varepsilon)-w_x\big)m^{(0)}\Big]_x, &0\leq x< \infty,\,\,0\leq t\leq T\vspace{4pt}\\
			(iii)&\displaystyle \mu(x,0)=0,\text{\hspace{.2cm}}w(x,T)=0, &0\leq x< \infty\vspace{4pt}\\
			(iv)&\displaystyle w(0,t)=\mu(0,t)=0, &0\leq t\leq T,
		\end{array}
	\end{cases}
\end{equation}
where $F^{(0)},  u^{(0)}, m^{(0)}, \Psi,$ and $\Phi$ are given functions, and 
\begin{equation}\label{def:G}
	G(w_x,\mu;\varepsilon):=\beta\big(\varepsilon\xi(t)\big)\int_0^\infty \Big(u_x^{(0)}\mu+m^{(0)}w_x\Big)\dif x.
\end{equation}
with $\beta$ a known function such that $\beta(\varepsilon)\to0$ as $\varepsilon\to0$. Our \emph{main mathematical contribution} in this article is a set of a priori estimates on solutions to System \eqref{abstract_system}.
(See Section \ref{sec:a_priori}.)
The estimates share three types in common with \cite[Section 3.3]{cardaliaguet2019master}:
(1) ``energy estimates," derived by expanding $\od{}{t} \ip{w}{\mu}$ and using the duality between equations (i) and (ii);
(2) standard Schauder estimates on parabolic equations;
and (3) estimates on $\mu$ in the space dual to a certain H\"older space.
Since the coupling depends on integral terms involving the unknowns, the energy estimates are considerably more technical than for standard mean field games.
In particular, they require a fourth type of estimate, namely (4) estimates on $\int_0^T \del{\int_0^\infty \mu(x,t)\dif x}^2 \dif t$; see Section \ref{sec:a priori F-P}.
Additionally, because the boundary conditions in our model are of Dirichlet type, the estimates on $\mu$ in the dual of a H\"older space require additional care compared to similar estimates on a torus; see Section \ref{subsec:Duality}.
All of our results hold in one space dimension, on which the problem was originally posed by Chan and Sircar (cf.~\cite{gueant2011mean}).
In higher space dimensions, the Dirichlet boundary conditions pose an additional difficulty for the purposes of regularity, since there would be a corner in the domain, although the interior estimates are expected to hold as in the one-dimensional case.
This is a technical aspect we do not address in the present work.

Mean field games were introduced circa 2006 by Lasry and Lions \cite{lasry07} and by Caines, Huang, and Malham\'e \cite{huang2006large}.
Since then, both theory and applications have been well-studied.
For a general exposition, we refer the reader to the texts \cite{bensoussan2013mean,carmona2017probabilistic,carmona2017probabilisticII}.
Economic models are a common application of mean field game theory.
See e.g.~the overviews in \cite{achdou2014partial,gueant2008application,gomes2015economics}, and for the particular example of exhaustible resource production see \cite{chan2015bertrand,chan2017fracking,gueant2011mean}.
When the equilibrium strategy depends on the distribution of controls, as is often the case for economics applications, we give the name \emph{mean field game of the controls} \cite{cardaliaguet2017mfgcontrols} (or \emph{extended mean field games} \cite{gomes2014extended}) to the resulting mathematical model.
The theory of partial differential equations for mean field games of controls has been developed in \cite{cardaliaguet2017mfgcontrols,kobeissi2021classical,kobeissi2020mean,gomes2016extended,gomes2014extended,graber2021weak}.
The present work is, to the best of our knowledge, the first study of parameter sensitivity for a mean field game of controls.
Some of the estimates presented here will be useful in a forthcoming study of the master equation for a mean field game of controls with absorbing boundary conditions \cite{graber2021master}.

The remainder of this article is organized as follows.
In the rest of this introduction, we define some notation and present the main result.
Section \ref{sec:a_priori} is the core of the paper and provides a priori estimates on systems of equations with the abstract form \eqref{abstract_system}.
In Section \ref{sec:L-S_Theorem} we prove existence and uniqueness of solutions to \eqref{abstract_system}.
Finally, in Section \ref{sec:Induction}, we prove the main result.

\subsection{Notation}

Throughout this manuscript, $L^p(D)$ ($1 \leq p \leq \infty$) will denote the usual Lebesgue space on a domain $D$ with standard norm, denoted either by $\enVert{\cdot}_{L^p(D)}$ or simply $\enVert{\cdot}_p$.
We will often consider the space-time domain $D = (0,\infty) \times [0,T]$, on which we consider the space $L_t^p L_x^q$ of functions $f = f(x,t)$ such that the map $t \mapsto \enVert{f(\cdot,t)}_q$ is in $L^p(0,T)$.
The norm is given by
\begin{equation}
	\enVert{f}_{L_t^p L_x^q} = \enVert{t \mapsto \enVert{f(\cdot,t)}_q}_p.
\end{equation}
For $\alpha \in (0,1)$ the space $\s{C}^{\alpha} = \s{C}^{\alpha}\del{\intco{0,\infty}}$ denotes the space of H\"older continuous functions $u$ such that the following norm is finite:
\begin{equation}
	\enVert{u}_{\s{C}^{\alpha}} := \enVert{u}_\infty + \sup \cbr{\frac{\abs{u(x) - u(y)}}{\abs{x-y}^\alpha} : x,y \geq 0,  \ x \neq y}.
\end{equation}
If $k$ is an integer, $\s{C}^{k + \alpha} = \s{C}^{k + \alpha}\del{\intco{0,\infty}}$ denotes the space of $k$ times differentiable functions $u$ such that $\od[k]{u}{x} \in \s{C}^\alpha$.
It is a Banach space with the norm
\begin{equation}
	\enVert{u}_{\s{C}^{k + \alpha}} = \enVert{\dod[k]{u}{x}}_{\s{C}^\alpha} +  \enVert{u}_\infty,
\end{equation}
which by standard interpolation results is equivalent to the norm $\enVert{\od[k]{u}{x}}_{\s{C}^\alpha} +  \sum_{j=0}^{k-1} \enVert{\od[j]{u}{x}}_\infty$.

For $\alpha, \beta \in (0,1)$ the space $\s{C}^{\alpha,\beta} = \s{C}^{\alpha,\beta}\del{\intco{0,\infty},[0,T]}$ denotes the space of all H\"older continuous functions $u$ such that the following norm is finite:
\begin{equation}
	\enVert{u}_{\s{C}^{\alpha,\beta}} := \enVert{u}_\infty + \sup \cbr{\frac{\abs{u(x,t) - u(y,s)}}{\abs{x-y}^\alpha + \abs{t-s}^\beta} : x,y \geq 0, \ t,s \in [0,T], \ x \neq y, \ t \neq s}.
\end{equation}
If $j,k$ are integers, we can also define $\s{C}^{j + \alpha,k + \beta}$ analogously.
In particular, the space $\s{C}^{2+\alpha,1+\alpha/2}$ denotes the standard H\"older space for parabolic equations (cf.~\cite{ladyzhenskaia1968linear});
its norm can be written
\begin{equation}
	\enVert{u}_{\s{C}^{2+\alpha,1+\alpha/2}} = \enVert{u}_\infty + \enVert{\dpd{u}{t}}_{\s{C}^{\alpha,\alpha/2}} + \enVert{\dpd[2]{u}{x}}_{\s{C}^{\alpha,\alpha/2}}.
\end{equation}

Given any function $f(x,t;\varepsilon)$, we denote the $k$-th partial derivative of $f(x,t;\varepsilon)$ with respect to $\varepsilon$ by
\begin{equation}\label{notation:dd epsilon}
	\displaystyle f^{(k)}=f^{(k)}(x,t;\varepsilon):=\frac{\partial^k}{\partial\varepsilon^k}\left[f(x,t;\varepsilon)\right],
\end{equation}
with the convention that $f^{(0)}=f$.

In studying System \eqref{abstract_system},
we will frequently suppress notation and write $G(u_x,m;\varepsilon)$, $G(\varepsilon)$, or even $G$ to denote $G(t,u_x,m;\varepsilon)$, provided that no ambiguity arises.
Analogous statements hold for other functionals that depend on multiple arguments.

\subsection{Statement of the main result} \label{sec:setup}

Throughout this paper, we assume that conditions hold on the data so that System \eqref{u-msystem} has a unique solution $(u,m)$ satisfying $u,m \in \s{C}^{2+\alpha,1+\alpha/2}$ for some $\alpha \in (0,1)$.
Sufficient conditions are provided in \cite{graber2021nonlocal}, cf.~\cite{graber2018existence}.
We will denote this solution by $(u^{(0)},m^{(0)})$.
In addition to smoothness, we will also need to assume the initial density $m_0$ has finite first moment $\int_0^\infty x m_0(x)\dif x < \infty$.
\\\\

\begin{comment}
Our goal is to study System \eqref{u-msystem} upon differentiating with respect to $\varepsilon$.
We will use $u^{(k)}$ and $m^{(k)}$ to denote the formal $k$-th derivatives with respect to $\varepsilon$ of $u$ and $m$ in accordance with the notation established in \eqref{notation:dd epsilon}.
Next, observe that $u_T(x)$ does not vary with $\varepsilon$, so $u^{(k)}_T(x)=0$.
When $t=T$, we have that $F^{(k)}(x,T,\varepsilon)=0$, since $\xi(T)=0$ and $u_x(0,T)=u'_T(x)$.
Similarly, $M(x)$ does not depend on $\varepsilon$ as an initial condition, so $m^{(k)}(x,0)=M^{(k)}(x)=0$.
Dirichlet boundary conditions on the left also remain for $u^{(k)}$ and $m^{(k)}$.
\end{comment}
Formally, if we differentiate \eqref{u-msystem} $k$ times with respect to $\varepsilon$, we obtain the system
\begin{equation}\label{un-mnsystem}
	\begin{cases}
		\begin{array}{lll}
			(i)&\displaystyle u_t^{(k)}+\frac{\sigma^2}{2}u_{xx}^{(k)}-ru^{(k)}+J_k\big(x,t, u_x^{(k)}, m^{(k)};\varepsilon\big)=0, &0\leq x< \infty,\text{\hspace{.1cm}} 0\leq t\leq T\vspace{4pt}\\
			(ii)&\displaystyle m_t^{(k)}-\frac{\sigma^2}{2}m_{xx}^{(k)}-(K_k)_x\big(x,t, u_x^{(k)}, m^{(k)};\varepsilon\big)=0, &0\leq x< \infty,\text{\hspace{.1cm}}0\leq t\leq T\vspace{4pt}\\
			(iii)&\displaystyle
			m^{(k)}(x,0)=0, \hspace{.2cm} u^{(k)}(x,T)=0, &0\leq x< \infty\vspace{4pt}\\
			(iv)&\displaystyle
			u^{(k)}(0,t)=m^{(k)}(0,t)=0, &0\leq t\leq T,
		\end{array}
	\end{cases}
\end{equation}
where
\begin{equation}\label{def:Jk}
	J_k\big(x,t,u_x^{(k)},m^{(k)};\varepsilon\big):=\sum_{j=0}^{k}\binom{k}{j}F^{(j)}\big(x,t,u_x^{(j)}, m^{(j)};\varepsilon\big)F^{(k-j)}\big(x,t,u_x^{(k-j)}, m^{(k-j)};\varepsilon\big)
\end{equation}
and
\begin{equation}\label{def:Kk}
	K_k\big(x,t, u_x^{(k)}, m^{(k)};\varepsilon\big):=\sum_{j=0}^k\binom{k}{j}F^{(j)}\big(x,t,u_x^{(j)}, m^{(j)};\varepsilon\big)m^{(k-j)}(x,t).
\end{equation}
For clarity, we note that the full expression for $F^{(k)}$ is given by
\begin{multline}\label{def:Fk}
	F^{(k)}\big(x,t, u_x^{(k)}, m^{(k)};\varepsilon\big)=
	\frac{1}{2}\xi(t)^k\alpha^{(k)}(\varepsilon\xi(t))-\frac{1}{2}u_x^{(k)}(x,t)
	\\[4pt]
	+\frac{1}{2}\sum_{i=0}^k\sum_{j=0}^{k-i}\binom{k}{i}\binom{k-i}{j}\xi(t)^i\beta^{(i)}\big(\varepsilon\xi(t)\big)\int_0^\infty u_x^{(j)}(x,t)m^{(k-i-j)}(x,t)\dif x,
\end{multline}
where
\begin{equation}\label{def:a_and_b}
	\alpha(\varepsilon):=\frac{2}{2+\varepsilon}\,\,\text{ and }\,\,\beta(\varepsilon):=\frac{\varepsilon}{2+\varepsilon},
\end{equation}
so that for $k\in\bb{N}$ with $k\geq1$,
\begin{equation}
	\alpha^{(k)}(\varepsilon)=\frac{2(-1)^k k!}{(2+\varepsilon)^{k+1}}\,\,\text{ and }\,\,\beta^{(k)}=-\alpha^{(k)}.
\end{equation}

The reader may wonder why, in writing $F^{(k)}$, $J_k$, and $K_k$, we have suppressed the arguments $u^{(j)}$ and $m^{(j)}$ for $j < k$.
This is because we will be arguing inductively as follows: to prove that $u^{(k)}$ and $m^{(k)}$ exist, we may assume that $u^{(j)}$ and $m^{(j)}$ are known functions for $j < k$.

We now state our main result.

\begin{theorem}\label{thm:ExistenceUniquenessEtc_un/mn}
	There exists $\varepsilon_0 > 0$ small enough  (see Equation \eqref{Assumption:epsilon_small}) such that for each $k\in\bb{N}$ and all $\varepsilon \in [0,\varepsilon_0]$, System \eqref{un-mnsystem} has a unique classical solution $(u^{(k)}, m^{(k)})$, which satisfies the identity
	\begin{equation}
		(u^{(k)}, m^{(k)})=\Big(\pd[k]{u}{\varepsilon}, \pd[k]{m}{\varepsilon}\Big)
	\end{equation}
 	where $(u,m)$ is the solution to \eqref{u-msystem}.
	That is, the formal differentiation carried out on $u$ and $m$ to obtain System \eqref{un-mnsystem} is justified.
	Moreover, there exists a constant $C_k > 0$ such that for $\varepsilon\in[0,\varepsilon_0]$,
	\begin{equation} \label{eq:Taylor estimate}
		\begin{array}{c}
			\enVert[1]{u(\cdot,\cdot;\varepsilon) - \sum_{j=0}^k \frac{\varepsilon^j}{j!}u^{(j)}(\cdot,\cdot;0)}_{\s{C}^{2+\alpha,1+\alpha/2}}\leq C_k\varepsilon^{k+1},
			\vspace{12pt}\\
			\enVert[1]{m(\cdot,\cdot;\varepsilon) - \sum_{j=0}^k \frac{\varepsilon^j}{j!}m^{(j)}(\cdot,\cdot;0)}_{\s{C}^{2+\alpha,1+\alpha/2}}\leq C_k\varepsilon^{k+1},
		\end{array}
	\end{equation}
	where $\alpha$ is sufficiently small.
\end{theorem}

\begin{comment}
Although the specific systems given in \eqref{u-msystem} and \eqref{un-mnsystem} are ultimately what interest us, the various a priori estimates we require are easier to demonstrate in a more abstract framework. 
Moreover, there is an underlying general structure that System \eqref{un-mnsystem} enjoys, which will resurface in our later proof of Theorem \ref{thm:ExistenceUniquenessEtc_un/mn} showing that the formal derivatives in $\varepsilon$ of $u$ and $m$ are the genuine derivatives in $\varepsilon$.
\end{comment}

\begin{comment}
When it is unimportant to stress the dependence of $w_x$ and $\mu$ on $G$, we may simply write $G(\varepsilon)$ or even simpler $G$, provided the dependence of $\varepsilon$ on $G$ is not overtly relevant.
\end{comment}

\section{A priori estimates}\label{sec:a_priori}

In this section we present our main mathematical contribution by proving new a priori estimates for a coupled forward-backward system of linear partial differential equations \eqref{abstract_system}.
It turns out that a useful first step for such systems is to prove ``energy estimates," which are derived by expanding $\od{}{t} \ip{w}{\mu}$ and using the duality between equations (i) and (ii).
However, because of the integral term $G$ appearing in \eqref{abstract_system}, the energy estimates are not useful without certain a priori estimates on the quantity $\int_0^T \del{\mu(x,t)\dif x}^2 \dif t$, where $\mu$ is the solution to  the Fokker-Planck type equation \eqref{abstract_system}(ii).
We derive these estimates first in Section \ref{sec:a priori F-P}, which may have independent interest to the reader interested in Fokker-Planck equations with a source.
Section \ref{sec:energy} then provides the desired energy estimates, which we apply throughout the rest of the section.
In section \ref{subsec:Duality} we arrive at further estimates for the Fokker-Planck equation in the dual to a H\"older space.
Finally, in Section \ref{sec:full regularity} we give a priori estimates that establish full parabolic regularity.

\subsection{A priori estimates on Fokker-Planck equations with a source}

\label{sec:a priori F-P}

In this subsection, we collect some estimates on the Fokker-Planck equation \eqref{abstract_system}(ii), in particular with respect to the $L^1$ norm and first moment in the space variable.
These results have very little to do with the particular structure of the coupled system \eqref{abstract_system} and can be stated abstractly for a Fokker-Planck equation with a source.
The proofs of these results are given in Appendix \ref{ap:a priori}.

\begin{proposition}\label{prop:LtInftyLx1_moment_mu}
	Let $\mu$ solve
	\begin{equation}\label{F-P_equation_mu}
		\begin{cases}
			\begin{array}{lll}
				(i)&\displaystyle \mu_t-\frac{\sigma^2}{2}\mu_{xx}-(b\mu)_x=\nu_x, &0\leq x< \infty,\,\,0\leq t\leq T\\
				(ii)&\displaystyle \mu(x,0)=\mu_0(x), &0\leq x< \infty\\
				(iii)&\displaystyle \mu(0,t)=0, &0\leq t\leq T,
			\end{array}
		\end{cases}
	\end{equation}
	where $\mu_0 \in L^1(\intco{0,\infty})$, $b\in L^\infty([0,\infty)\times[0,T])$ and $\nu\in L_t^\infty\big([0,T]; L_x^1\big([0,\infty)\big)\big)$ are known functions.
	Additionally, assume $x\mu_0 \in L^1(\intco{0,\infty})$ and $x\nu\in L_t^\infty\big([0,T]; L_x^1\big([0,\infty)\big)\big)$.
	Then
	\begin{enumerate}
		\item[(a)] $\mu\in L_t^\infty\big([0,T]; L_x^1\big([0,\infty)\big)\big)$ with
		$$
		\enVert{\mu}_{L_t^\infty(L_x^1)}\leq C\del{\enVert{\mu_0}_{L^1} +  \enVert{\nu}_{L_t^\infty(L_x^1)}},
		$$
		where the constant $C\in(0,\infty)$ depends only on $\sigma$, $\enVert{b}_\infty$, and $T$;
		\item[(b)] and
		$$
		\displaystyle \sup_{0\leq \tau\leq T}\int_0^\infty x|\mu(x,\tau)|\dif x\leq C\big(\enVert{(1+x)\mu_0}_{L^1} + \enVert{\nu}_{L_t^\infty(L_x^1)}+\enVert{x\nu}_{L_t^\infty(L_x^1)}\big),
		$$
		where the constant $C\in(0,\infty)$ depends only on $\sigma$ and $\enVert{b}_{\infty}$, and $T$.
	\end{enumerate}
\end{proposition}

\begin{proof}
	See Appendix \ref{ap:a priori F-P}.
\end{proof}

The following proposition could be stated for an abstract Fokker-Planck equation like \eqref{F-P_equation_mu}; nevertheless, it is given in the form below to make it more obvious how it may be applied later on.
\begin{proposition}\label{prop:L2timeL1space_mu}
	Define the constant
	\begin{equation} \label{eq:C0}
		C_0:=384c'\ln(2)\big(1+\enVert[1]{u_x^{(0)}}_\infty^2\big)^3,
	\end{equation}
	where $c'$ is as in \eqref{c'_constant}, and suppose $\mu$ satisfies $(ii)$ of System \eqref{abstract_system} with $\Phi\in L_t^\infty(L_x^1)$. If $\lambda> C_0$, then
	\begin{equation}\label{L2timeLxLambdaBig}
		\int_0^Te^{-\lambda t}\Bigg(\int_0^\infty|\mu|\dif x\Bigg)^2\dif t\leq \frac{C_1\enVert[1]{\Phi}_{L_t^\infty(L_x^1)}^2}{C_0}+C_2\int_0^T e^{-\lambda t}\int_0^\infty w_x^2m^{(0)}\dif x\dif t,
	\end{equation}
	where the constants $C_1,C_2\in(0,\infty)$ depend only on $\enVert[1]{u_x^{(0)}}_{\infty}$.
	
	In particular, if $T<\infty$, then there is no restriction on $\lambda$, and for $\lambda\in(0,C_0]$,
	\begin{equation}\label{L2timeLxLambdaSmall}
		\int_0^Te^{-\lambda t}\Bigg(\int_0^\infty|\mu|\dif x\Bigg)^2\dif t\leq\frac{C_1\enVert[1]{\Phi}_{L_t^\infty(L_x^1)}^2e^{2C_0T}}{C_0}+C_2e^{2C_0T}\int_0^T e^{-\lambda t}\int_0^\infty w_x^2m^{(0)}\dif x\dif t,
	\end{equation}
	with $C_1$ and $C_2$ as in \eqref{L2timeLxLambdaBig}.
\end{proposition}

\begin{proof}
	See Appendix \ref{ap:a priori F-P}.
\end{proof}

\subsection{Energy estimates}

\label{sec:energy}

\begin{comment}
Proposition \ref{prop:L2timeL1space_mu} gives a necessary estimate that leads to the following energy type estimate for System \eqref{abstract_system}.
As we will later see, the following estimate is crucial to eventually demonstrating H\"{o}lder regularity for both $w$ and $\mu$.
\end{comment}

By \emph{energy estimates} we mean specifically an estimate on the quantity $\int_0^T e^{-rt} \int_0^\infty \abs{w_x}^2 m^{(0)} \dif x \dif t$.
Before stating our result, we first define an upper bound on the parameter $\varepsilon$ that we will need in the proof.
With $\beta$ as in \eqref{def:a_and_b} and $C_0$ as in \eqref{eq:C0}, let $\varepsilon_0>0$ be such that
	\begin{equation}\label{Assumption:epsilon_small}
		\beta(\varepsilon_0)<\frac{1+3\enVert[1]{u_x^{(0)}}_\infty^2}{96e^{2C_0T}\Big(\enVert[1]{F^{(0)}}_\infty^2+\dfrac{3}{4}\enVert[1]{u_x^{(0)}}_\infty^2\Big)}\quad\text{ and }\quad \beta(\varepsilon_0)<\frac{1}{10}.
	\end{equation}
Note that such a $\varepsilon_0$ is possible to obtain, because $\lim_{\varepsilon\to0^+}\beta(\varepsilon)=0$. The necessity of this assumption is a consequence of the method of proof for the energy estimate. Specifically, the upper bounds are motivated by \eqref{beta_delta_small}, \eqref{energy_estimate:penultimateStep}, and the value of $C_2$, the constant appearing in Proposition~\ref{prop:L2timeL1space_mu}---see the proof of Proposition~\ref{prop:L2timeL1space_mu} for an explicit value.

\begin{proposition}\label{prop:EnergyEstimate_w}
	Suppose $(w,\mu)$ satisfies \eqref{abstract_system} with $\Psi\in L^\infty$ and $\Phi\in L_t^\infty(L_x^1)$. Let $\varepsilon_0$ satisfy \eqref{Assumption:epsilon_small}. Then for $\varepsilon\in[0,\varepsilon_0]$ the following energy estimate is valid
	\begin{equation}
		\int_0^Te^{-rt}\int_0^\infty |w_x|^2m^{(0)}\dif x\dif t\leq C_1\enVert[1]{\Phi}_{L_t^\infty(L_x^1)}\enVert[1]{w_x}_{\infty}+C_2\Big(\enVert[1]{\Phi}_{L_t^\infty(L_x^1)}^2+\enVert[1]{\Psi}_{\infty}\Big),
	\end{equation}
	where $C_1\in(0,\infty)$ is a constant that depends only on $r$; and $C_2\in(0,\infty)$ is a constant that depends only on $\enVert[1]{u_x^{(0)}}_{\infty}$, $\sigma$, $r$, and $T$.
	As an immediate corollary,
\begin{equation}
	\int_0^T\int_0^\infty |w_x|^2m^{(0)}\dif x\dif t\leq C_1'\enVert[1]{\Phi}_{L_t^\infty(L_x^1)}\enVert[1]{w_x}_{\infty}+C_2'\Big(\enVert[1]{\Phi}_{L_t^\infty(L_x^1)}^2+\enVert[1]{\Psi}_{\infty}\Big),
\end{equation}
with $C_1'=e^{rT}C_1$ and $C_2'=e^{rT}C_2$.
\end{proposition}
\begin{proof}
	We begin by observing
	\begin{align}\label{ddt_Int_wmu1}
		\frac{\dif}{\dif t}\int_0^\infty e^{-rt}w\mu\dif x&=\int_0^\infty e^{-rt}\Big(\big(w_t-rw\big)\mu-w\mu_t\Big)\dif x
		\\[4pt]
		&=\int_0^\infty e^{-rt}\Bigg\{\mu\Big(\frac{-\sigma^2}{2}w_{xx}-F^{(0)}(G-w_x)-\Psi\Big)\nonumber
		\\[4pt]
		&\quad+w\Big(\frac{\sigma^2}{2}\mu_{xx}+\big[F^{(0)}\mu\big]_x+\Big[\Phi+\frac{1}{2}(G-w_x)m^{(0)}\Big]_x\Big)\Bigg\}\dif x\nonumber
		\\[4pt]
		&=-e^{-rt}\int_0^\infty\Big(F^{(0)}G\mu+\Psi\mu+\Phi w_x+\frac{1}{2}w_x(G-w_x)m^{(0)}\Big)\dif x.\nonumber
	\end{align}
	The first equality follows via differentiation under the integral sign, the second equality follows from substituting in the equations for $w_t-rw$ and $\mu_t$ given by System \eqref{abstract_system}, and the last equality follows via integration by parts.
	Next, unpacking $G$, \eqref{ddt_Int_wmu1} implies
	\begin{align}\label{ddt_Int_wmu2}
		-\frac{\dif}{\dif t}\int_0^\infty e^{-rt}w\mu\dif x&=e^{-rt}\beta(\varepsilon\xi)\Bigg(\int_0^\infty F^{(0)}\mu\dif x\Bigg)\Bigg(\int_0^\infty\big(w_xm^{(0)}+u_x^{(0)}\mu\big)\dif x\Bigg)
		\\[4pt]
		&\quad+e^{-rt}\int_0^\infty\big(\Psi\mu+\Phi w_x\big)\dif x-\frac{1}{2}e^{-rt}\int_0^\infty w_x^2m^{(0)}\dif x\nonumber
		\\[4pt]
		&\quad+\frac{1}{2}e^{-rt}\beta(\varepsilon\xi)\Bigg(\int_0^\infty w_xm^{(0)}\dif x\Bigg)\Bigg(\int_0^\infty\big(w_xm^{(0)}+u_x^{(0)}\mu\big)\dif x\Bigg).\nonumber
	\end{align}
	We then have the following three estimates. First,
	\begin{align}
		\Bigg(\int_0^\infty F^{(0)}\mu&\dif x\Bigg)\Bigg(\int_0^\infty\big(w_xm^{(0)}+u_x^{(0)}\mu\big)\dif x\Bigg)
		\\[4pt]
		&=\Bigg(\int_0^\infty F^{(0)}\mu\dif x\Bigg)\Bigg(\int_0^\infty w_xm^{(0)}\dif x\Bigg)+\Bigg(\int_0^\infty F^{(0)}\mu\dif x\Bigg)\Bigg(\int_0^\infty u_x^{(0)}\mu\dif x\Bigg)\nonumber
		\\[4pt]
		&\leq\Bigg(\int_0^\infty F^{(0)}\mu\dif x\Bigg)^2+\frac{1}{2}\Bigg(\int_0^\infty w_xm^{(0)}\dif x\Bigg)^2+\frac{1}{2}\Bigg(\int_0^\infty u_x^{(0)}\mu\dif x\Bigg)^2\nonumber
		\\[4pt]
		&\leq\Big(\enVert[1]{F^{(0)}}_{\infty}^2+\frac{1}{2}\enVert[1]{u_x^{(0)}}_{\infty}^2\Big)\Bigg(\int_0^\infty |\mu|\dif x\Bigg)^2+\frac{1}{2}\int_0^\infty w_x^2m^{(0)}\dif x\nonumber,
	\end{align}
	where the first inequality follows by two applications of Young's inequality, and the last inequality via the Cauchy-Schwarz inequality.
	Second, let $\delta>0$ to be chosen later. Then
	\begin{align}
		\int_0^\infty\big(\Psi\mu+\Phi w_x\big)\dif x&\leq\enVert[1]{\Psi}_{\infty}\int_0^\infty|\mu|\dif x+\enVert[1]{\Phi}_{L_x^1}\enVert{w_x}_{\infty}
		\\[4pt]
		&\leq\frac{1}{2\delta}+\frac{\delta}{2}\enVert[1]{\Psi}_{\infty}^2\Bigg(\int_0^\infty|\mu|\dif x\Bigg)^2+\enVert[1]{\Phi}_{L_x^1}\enVert{w_x}_{\infty}.\nonumber
	\end{align}
	The last inequality above follows from a generalized Young's inequality.
	Third and last,
	\begin{align}
		\Bigg(\int_0^\infty w_xm^{(0)}&\dif x\Bigg)\Bigg(\int_0^\infty\big(w_xm^{(0)}+u_x^{(0)}\mu\big)\dif x\Bigg)
		\\[4pt]
		&=\Bigg(\int_0^\infty w_xm^{(0)}\dif x\Bigg)^2+\Bigg(\int_0^\infty w_xm^{(0)}\dif x\Bigg)\Bigg(\int_0^\infty u_x^{(0)}\mu\dif x\Bigg)\nonumber
		\\[4pt]
		&\leq\int_0^\infty w_x^2m^{(0)}\dif x+\frac{1}{2}\Bigg(\int_0^\infty w_xm^{(0)}\dif x\Bigg)^2+\frac{1}{2}\Bigg(\int_0^\infty u_x^{(0)}\mu\dif x\Bigg)^2\nonumber
		\\[4pt]
		&\leq\frac{3}{2}\int_0^\infty w_x^2m^{(0)}\dif x+\frac{1}{2}\enVert[1]{u_x^{(0)}}_{\infty}^2\Bigg(\int_0^\infty|\mu|\dif x\Bigg)^2.\nonumber
	\end{align}
	The first inequality follows via the Cauchy-Schwarz inequality and Young's inequality.
	The last inequality follows via another application of the Cauchy-Schwarz inequality.
	Together with \eqref{ddt_Int_wmu2}, the above three estimates imply
	\begin{align}\label{ddt_Int_wmu3}
			-\frac{\dif}{\dif t}\int_0^\infty e^{-rt}w\mu\dif x&\leq e^{-rt}\Bigg(\beta(\varepsilon\xi)\Big(\enVert[1]{F^{(0)}}_{\infty}^2+\frac{3}{4}\enVert[1]{u_x^{(0)}}_{\infty}^2\Big)+\frac{\delta}{2}\enVert[1]{\Psi}_{\infty}\Bigg)\Bigg(\int_0^\infty|\mu|\dif x\Bigg)^2
			\\[4pt]
			&\quad+\frac{1}{2}e^{-rt}\Big(\frac{5}{2}\beta(\varepsilon\xi)-1\Big)\int_0^\infty w_x^2m^{(0)}\dif x+e^{-rt}\Big(\frac{1}{2\delta}+\enVert[1]{\Phi}_{L_x^1}\enVert[1]{w_x}_{\infty}\Big).\nonumber
	\end{align}
	To make future calculations less cluttered define
	\begin{equation}
		f(t):=2\beta(\varepsilon\xi)\Big(\enVert[1]{F^{(0)}}_{\infty}^2+\frac{3}{4}\enVert[1]{u_x^{(0)}}_{\infty}^2\Big)+\delta\enVert[1]{\Psi}_{\infty}.
	\end{equation}
	Now, because $w(x,T)=\mu(x,0)=0$, integrating the left-hand side of \eqref{ddt_Int_wmu3} in time and rearranging the inequality, we obtain
	\begin{align}
		\int_0^T\Big(1-\frac{5}{2}\beta(\varepsilon\xi)\Big)e^{-rt}\int_0^\infty w_x^2m^{(0)}\dif x\dif t&\leq \int_0^Tf(t)e^{-rt}\Bigg(\int_0^\infty|\mu|\dif x\Bigg)^2\dif t
		\\[4pt]
		&\quad+\int_0^Te^{-rt}\Big(\frac{1}{\delta}+2\enVert[1]{\Phi}_{L_x^1}\enVert[1]{w_x}_{\infty}\Big)\dif t\nonumber,
	\end{align}
	Note that $\beta(\varepsilon\xi)$, and subsequently $f(t)$, converges to 0 uniformly in $t$ as $\varepsilon$ and $\delta$ go to 0.
	Let $\delta_1>0$ to be determined shortly.
	We can take $\varepsilon$ and $\delta$ small enough such that
	\begin{equation}\label{beta_delta_small}
		\beta(\varepsilon\xi)<\frac{\delta_1}{4\Big(\enVert[1]{F^{(0)}}_{\infty}^2+\displaystyle\frac{3}{4}\enVert[1]{u_x^{(0)}}_{\infty}^2\Big)}\,\,\text{ and }\,\,\delta<\frac{\delta_1}{2\enVert[1]{\Psi}_{\infty}}.
	\end{equation}
	That is, $\enVert[1]{f}_{L_t^\infty}<\delta_1$ for $\varepsilon$ and $\delta$ small enough. Consequently, 
	by Proposition~\ref{prop:L2timeL1space_mu},
	\begin{align}\label{energy_estimate:penultimateStep}
		\int_0^T\Big(1-\frac{5}{2}\beta(\varepsilon\xi)-\delta_1C_2\Big)e^{-rt}\int_0^\infty& w_x^2m^{(0)}\dif x\dif t
		\\[4pt]
		&\leq \delta_1C_1\enVert[1]{\Phi}_{L_t^\infty(L_x^1)}^2+\int_0^Te^{-rt}\Big(\frac{1}{\delta}+2\enVert[1]{\Phi}_{L_x^1}\enVert[1]{w_x}_{\infty}\Big)\dif t\nonumber,
	\end{align}
	for some constants $C_1,C_2\in(0,\infty)$ which depend only on $\enVert[1]{u_x^{(0)}}_{\infty}$, $\sigma$, and $T$.
	We now choose $\delta_1<1/(4C_2)$.
	If necessary, we may take $\varepsilon$ smaller so that $\beta(\varepsilon\xi)<1/10$ as well. This, together with \eqref{beta_delta_small}, demonstrates that by choosing $\varepsilon_0$ as in Equation \eqref{Assumption:epsilon_small}, we find for all $\varepsilon\in[0,\varepsilon_0]$ that
	\begin{equation}
		1-\frac{5}{2}\beta(\varepsilon\xi)-\delta_1C_2<\frac{1}{2}.
	\end{equation}
	Subsequently,
	\begin{equation}
		\int_0^Te^{-rt}\int_0^\infty w_x^2m^{(0)}\dif x\dif t\leq \frac{C_1\enVert[1]{\Phi}_{L_t^\infty(L_x^1)}^2}{2C_2}+2\int_0^Te^{-rt}\Big(\frac{1}{\delta}+2\enVert[1]{\Phi}_{L_x^1}\enVert[1]{w_x}_{\infty}\Big)\dif t.
	\end{equation}
	To clean things up, the restriction on $\delta_1$ implies $\delta<1/\big(8C_2\enVert{\Psi}_{\infty}\big)$, and we can  restrict $\delta$ from below so that $\delta>1/\big(16C_2\enVert{\Psi}_{\infty}\big)$.
	As such, we find that
	\begin{equation}
		\int_0^Te^{-rt}\int_0^\infty w_x^2m^{(0)}\dif x\dif t\leq\frac{C_1\enVert[1]{\Phi}_{L_t^\infty(L_x^1)}^2}{2C_2}+32C_2\enVert[1]{\Psi}_{\infty}r^{-1}+4r^{-1}\enVert{\Phi}_{L_t^\infty(L_x^1)}\enVert{w_x}_{\infty},
	\end{equation}
	after estimating $\int_0^Te^{-rt}\dif t\leq r^{-1}$.
	This yields the desired result.
\end{proof}

As a consequence of the above energy estimate, we find that the maximum of $w$ is controlled partially by the square root of the maximum of $w_x$.

\begin{corollary}\label{cor:max_w}
	Suppose $(w,\mu)$ satisfies System \eqref{abstract_system} with $\Psi\in L^\infty$ and $\Phi\in L_t^\infty(L_x^1)$.
	Let $\varepsilon_0$ satisfy \eqref{Assumption:epsilon_small}. Then for $\varepsilon\in[0,\varepsilon_0]$,
	\begin{equation}
		|w(x,t)|\leq C_1\kappa\big(1+\enVert[1]{w_x}_{\infty}\big)^{1/2}+C_2\enVert{\Psi}_{\infty}\,\,\text{ for all }\,\, x,t\in[0,\infty)\times[0,T],
	\end{equation}
	where $C_1\in(0,\infty)$ is a constant depending only on $\enVert[1]{u_x^{(0)}}_{\infty}$, $\sigma$, $r$, and $T$; $C_2\in(0,\infty)$ is a constant depending only on $r$, and $T$; and $\kappa\in[0,\infty)$ is a constant depending solely on $\enVert[1]{\Phi}_{L_t^\infty(L_x^1)}$ and $\enVert[1]{\Psi}_{\infty}$ in such a way that $\kappa=0$ whenever $\enVert[1]{\Phi}_{L_t^\infty(L_x^1)}=\enVert[1]{\Psi}_{\infty}=0$.
\end{corollary}
\begin{proof}
	Begin by defining
	\begin{equation}
		f(t):=\enVert[1]{F^{(0)}}_{\infty}\beta(\varepsilon\xi)\int_0^\infty\envert{w_xm^{(0)}+u_x^{(0)}\mu}\dif x.
	\end{equation}
	System \eqref{abstract_system} implies
	\begin{equation}\label{max_wEq1}
		\envert[2]{w_t+\frac{\sigma^2}{2}w_{xx}-rw-F^{(0)}w_x}\leq f(t)+\enVert[1]{\Psi}_{\infty}.
	\end{equation}
	Set
	\begin{equation}
		v(x,t):=e^{-rt}w(x,t)-\int_t^Te^{-rs}\big(f(s)+\enVert[1]{\Psi}_\infty\big)\dif s,
	\end{equation}
	and hence \eqref{max_wEq1} implies
	\begin{equation}
		v_t+\frac{\sigma^2}{2}v_{xx}-F^{(0)}v_x\geq 0.
	\end{equation}
	Using the standard maximum principle on $v$, we find that
	\begin{equation}
		e^{-rt}w(x,t)\leq\int_0^Te^{-rs}\big(f(s)+\enVert[1]{\Psi}_{\infty}\big)\dif s.
	\end{equation}
	A similar argument, where we define
	\begin{equation}
		v(x,t):=e^{-rt}w(x,t)+\int_t^T\big(f(s)+\enVert[1]{\Psi}_{\infty}\big)\dif s
	\end{equation}
	instead, shows that
	\begin{equation}
		e^{-rt}w(x,t)\geq-\int_0^Te^{-rs}\big(f(s)+\enVert[1]{\Psi}_{\infty}\big)\dif s.
	\end{equation}
	As such,
	\begin{equation}\label{max_wEq2}
		\envert[2]{e^{-rt}w(x,t)}\leq\int_0^Te^{-rs}\big(f(s)+\enVert[1]{\Psi}_{\infty}\big)\dif s\,\,\text{ for all }(x,t)\in[0,\infty)\times[0,T].
	\end{equation}
	Using several careful applications of the Cauchy-Schwarz inequality we estimate
	\begin{align}
		\int_0^Te^{-rs}f(s)\dif s&\leq\enVert[1]{F^{(0)}}_{\infty}\Bigg(\int_0^Te^{-rs}\int_0^\infty\envert[2]{w_xm^{(0)}}\dif x\dif s+\int_0^Te^{-rs}\int_0^\infty\envert[2]{u_x^{(0)}\mu}\dif x\dif s\Bigg)
		\\[4pt]
		&\leq\enVert[1]{F^{(0)}}_{\infty}\int_0^Te^{-rs}\Bigg(\int_0^\infty w_x^2m^{(0)}\dif x\Bigg)^{1/2}\Bigg(\int_0^\infty m^{(0)}\dif x\Bigg)^{1/2}\dif s\nonumber
		\\[4pt]
		&\quad+\enVert[1]{F^{(0)}}_{\infty}\Bigg(\int_0^Te^{-rs}\dif s\Bigg)^{1/2}\Bigg(\int_0^Te^{-rs}\Bigg(\int_0^\infty\envert[2]{u_x^{(0)}\mu}\dif x\Bigg)^2\dif s\Bigg)^{1/2}\nonumber
		\\[4pt]
		&\leq\enVert[1]{F^{(0)}}_{\infty}\int_0^Te^{-rs}\Bigg(\int_0^\infty w_x^2m^{(0)}\dif x\Bigg)^{1/2}\dif s\nonumber
		\\[4pt]
		&\quad+\enVert[1]{F^{(0)}}_{\infty}\enVert[1]{u_x^{(0)}}_{\infty}r^{-1/2}\Bigg(\int_0^Te^{-rs}\Bigg(\int_0^\infty|\mu|\dif x\Bigg)^2\dif s\Bigg)^{1/2}\nonumber
		\\[4pt]
		&\leq\enVert[1]{F^{(0)}}_{\infty}\Bigg(\int_0^Te^{-rs}\dif s\Bigg)^{1/2}\Bigg(\int_0^Te^{-rs}\int_0^\infty w_x^2m^{(0)}\dif x\dif s\Bigg)^{1/2}\nonumber
		\\[4pt]
		&\quad+\enVert[1]{F^{(0)}}_{\infty}\enVert[1]{u_x^{(0)}}_{\infty}r^{-1/2}\Bigg(C_1\Vert \Phi\Vert_{L_t^\infty(L_x^1)}^2+C_2\int_0^Te^{-rs}\int_0^\infty w_x^2m^{(0)}\dif x\dif s\Bigg)^{1/2}\nonumber
		\\[4pt]
		&\leq C\Bigg(\enVert{\Phi}_{L_t^\infty(L_x^1)}^2+\int_0^Te^{-rs}\int_0^\infty w_x^2m^{(0)}\dif x\dif s\Bigg)^{1/2}.\nonumber
	\end{align}
	Note that the fourth inequality follows from Proposition~\ref{prop:L2timeL1space_mu}, so that $C_1,C_2\in(0,\infty)$ are constants that depend only on $\enVert[1]{u_x^{(0)}}_{\infty}$, $\sigma$, and $T$.
	Also, the constant $C$ can be given by
	\begin{equation}
		C:=\enVert[1]{F^{(0)}}_{\infty}r^{-1/2}\Big(1+\enVert[1]{u_x^{(0)}}_{\infty}(C_1+C_2)^{1/2}\Big).
	\end{equation}
	Therefore, combining this estimate with \eqref{max_wEq2} we obtain
	\begin{equation}
		|w(x,t)|\leq Ce^{rT}\Bigg(\enVert{\Phi}_{L_t^\infty(L_x^1)}^2+\int_0^Te^{-rs}\int_0^\infty w_x^2m^{(0)}\dif x\dif s\Bigg)^{1/2}+r^{-1}e^{rT}\enVert[1]{\Psi}_{\infty}
	\end{equation}
	for all $(x,t)\in[0,\infty)\times[0,T]$.
	Now, using the energy estimate, established in Proposition \ref{prop:EnergyEstimate_w}, there exist constants $C_3,C_4\in(0,\infty)$, depending only on $\enVert[1]{u_x^{(0)}}_{\infty}$, $\sigma$, $r$, and $T$ such that for $\kappa':=\enVert{\Phi}_{L_t^\infty(L_x^1)}^2+\enVert{\Psi}_{\infty}$,
	\begin{align}
		|w(x,t)|&\leq Ce^{rT}\Big(\enVert{\Phi}_{L_t^\infty(L_x^1)}^2+C_3\kappa'+C_4\enVert{\Phi}_{L_t^\infty(L_x^1)}\enVert[1]{w_x}_{\infty}\Big)^{1/2}+r^{-1}e^{rT}\enVert[1]{\Psi}_{\infty}
		\\[4pt]
		&\leq C\Big(\enVert{\Phi}_{L_t^\infty(L_x^1)}^2+C_3\kappa'+C_4\enVert{\Phi}_{L_t^\infty(L_x^1)}\Big)^{1/2}\big(1+\enVert[1]{w_x}_{\infty}\big)^{1/2}+r^{-1}e^{rT}\enVert[1]{\Psi}_{\infty}\nonumber
	\end{align}
	for all $(x,t)\in[0,\infty)\times[0,T]$, which is the desired result.
	\end{proof}

\subsection{Estimates in the dual of $\s{C}^{1+\alpha}$}\label{subsec:Duality}

In this subsection we want to provide a priori estimates on $\mu$, the solution to (ii) of System \eqref{abstract_system}, in the space
\begin{equation}
	\s{C}^{\alpha/4}([0,T];\,(\s{C}^{1+\alpha})^\ast)
\end{equation}
for a given $\alpha \in (0,1)$.
Such estimates are required to deduce the time-regularity of integral terms involving $\mu$.
We first use duality methods to obtain estimates on $\mu$ in the space
\begin{equation}
	\s{C}^{\alpha/4}([0,T];\,(\s{C}^{1+\alpha}_\diamond)^\ast),
\end{equation}
where $\s{C}^{1+\alpha}_\diamond$ is the space of all $\phi \in \s{C}^{1+\alpha}$ such that $\phi(0) = 0$.
These estimates rely on corresponding estimates for the primal problem, which are found in Appendix \ref{sec:abstract holder}.

\begin{proposition}\label{prop:HolderEstimates_mu}
	Let $(w,\mu)$ satisfy \eqref{abstract_system}.
	Let $\varepsilon_0$ satisfy \eqref{Assumption:epsilon_small}. Then for $\varepsilon\in[0,\varepsilon_0]$ there exists a constant $\kappa\in[0,\infty)$ and a constant $C\in(0,\infty)$, depending exclusively on $\enVert[1]{u_x^{(0)}}_{\infty}$, $\sigma$, $\lambda$, $\alpha$, and $T$, such that 
	\begin{equation}
		\enVert[1]{\mu}_{\s{C}^{\alpha/4}([0,T];\,(\s{C}^{1+\alpha}_\diamond)^\ast)}\leq C\kappa\big(1+\enVert[1]{w_x}_{\infty}\big)^{1/2}.
	\end{equation}
	Moreover, the constant $\kappa$ depends solely on $\enVert[1]{\Phi}_{L_t^\infty(L_x^1)}$ and $\enVert[1]{\Psi}_{\infty}$ in such a way that $\kappa=0$ whenever $\enVert[1]{\Phi}_{L_t^\infty(L_x^1)}=\enVert[1]{\Psi}_{\infty}=0$.
\end{proposition}
\begin{proof}
	Fix some $t_1\in[0,T]$.
	Consider the dual PDE to $\mu$,
	\begin{equation}
		\begin{cases}
			\begin{array}{ll}
				-\psi_t-\frac{\sigma^2}{2}\psi_{xx}+F^{(0)}\psi_x+\lambda\psi=0, \hspace{2cm} &(x,t)\in[0,\infty)\times[0,t_1]
				\\[4pt]
				\psi(x,t_1)=\varphi(x)\in\s{C}_{\diamond}^{1+\alpha}([0,\infty)), & x\in[0,\infty)
				\\[4pt]
				\psi(0,t)=0, & t\in[0,t_1],
			\end{array}
		\end{cases}
	\end{equation}
	for some $\lambda>0$.
	By Lemma \ref{lem:uniform t Holder ux}, we have
	\begin{equation} \label{eq:psi estimate}
		\enVert{\psi}_{\s{C}^{1+\alpha,\alpha/2}}, \enVert{\psi}_{\s{C}^{\alpha/4}\del{[0,T];\s{C}^{1+\alpha/2}(\overline{\s{D}})}}  \leq C_0\enVert{\phi}_{\s{C}^{1+\alpha}}
	\end{equation}
	for some constant $C_0$ depending only on $\lambda$, $\alpha$, and $\enVert{F^{(0)}}_\infty$.
	
	We write $\widetilde{\psi}(x,t):=e^{-\lambda t}\psi(x,t)$, so that the above equation implies
	\begin{equation}
		-\widetilde{\psi}_t-\frac{\sigma^2}{2}\widetilde{\psi}_{xx}+F^{(0)}\widetilde{\psi}_x=0.
	\end{equation}
	We observe that
	\begin{equation}
		\begin{split}
		\frac{\dif}{\dif t}\int_0^\infty\widetilde{\psi}(x,t)\mu(x,t)\dif x&=\int_0^\infty\big(\widetilde{\psi}_t\mu+\widetilde{\psi}\mu_t\big)\dif t
		\\[4pt]
		&=\int_0^\infty\big(-\frac{\sigma^2}{2}\widetilde{\psi}_{xx}+F^{(0)}\widetilde{\psi}_x\big)\mu\dif x
		\\[4pt]
		&\quad+\int_0^\infty\widetilde{\psi}\Bigg(\frac{\sigma^2}{2}\mu_{xx}+\big[F^{(0)}\mu\big]_x+\Big[\Phi+\frac{1}{2}(G-w_x)m^{(0)}\Big]_x\Bigg)\dif x 
		\\[4pt]
		&=-\int_0^\infty\widetilde{\psi}_x\Big(\Phi+\frac{1}{2}(G-w_x)m^{(0)}\Big)\dif x.
	\end{split}
	\end{equation}
	The first equality follows via differentiation under the integral sign, the second equality follows from substituting in the equations for $\widetilde{\psi}_t$ and $\mu_t$, and the last equality follows via integration by parts.
	Consequently, with $t_2\in[0,t_1]$ arbitrary, we have
	\begin{align}\label{muDualEstimate1}
		\int_0^\infty\widetilde{\psi}(x,t_1)\mu(x,t_1)\dif x&-\int_0^\infty\widetilde{\psi}(x,t_2)\mu(x,t_2)\dif x
		\\[4pt]
		&=-\int_{t_2}^{t_1}\int_0^\infty \psi_xe^{-\lambda t}\Big(\Phi+\frac{1}{2}(G-w_x)m^{(0)}\Big)\dif x\dif t\nonumber
		\\[4pt]
		&\leq\enVert{\psi_x}_{\infty}\int_{t_2}^{t_1}\int_0^\infty e^{-\lambda t}\envert[2]{\Phi+\frac{1}{2}(G-w_x)m^{(0)}}\dif x\dif t\nonumber
		\\[4pt]
		&\leq C_0\enVert[1]{\varphi}_{\s{C}^{1+\alpha}}\int_{t_2}^{t_1}e^{-\lambda t}\big(\eta_1(t)+\eta_2(t)\big)\dif t,\nonumber
	\end{align}
	where we have used \eqref{eq:psi estimate}, and where
	\begin{equation}
		\eta_1(t):=\int_0^\infty\envert[1]{\Phi}\dif x\,\,\text{ and }\,\, \eta_2(t):=\frac{1}{2}\int_0^\infty|G-w_x|m^{(0)}\dif x.
	\end{equation}
	By assumption and the fact that $e^{-\lambda t}\in\s{C}^{\alpha/2}([0,T])$, we find
	\begin{equation}\label{eta_1Estimate}
		\int_{t_2}^{t_1}e^{-\lambda t}\eta_1(t)\dif t\leq\frac{1}{\lambda}\enVert[1]{\Phi}_{L_t^\infty(L_x^1)}\envert{e^{-\lambda t_1}-e^{-\lambda t_2}}\leq C_1\enVert[1]{\Phi}_{L_t^\infty(L_x^1)}|t_1-t_2|^{\alpha/2},
	\end{equation}
	where the constant $C_1$ depends on $\lambda$ and $\alpha$.
	
	Turning our attention now to $\eta_2$, we first use the Cauchy-Schwarz inequality to write
	\begin{equation}\label{eta_2C-S}
		\int_{t_2}^{t_1}e^{-\lambda t}\eta_2(t)\dif t\leq\Bigg(\int_{t_2}^{t_1}e^{-\lambda t}\eta_2(t)^2\dif t\Bigg)^{1/2}\Bigg(\int_{t_2}^{t_1}e^{-\lambda t}\dif t\Bigg)^{1/2}.
	\end{equation}
	Then $\eta_2(t)^2$ can be estimated in the following way.
	\begin{align}
		\eta_2(t)^2
		&\leq\frac{1}{4}\Bigg(\int_0^\infty|u_x^{(0)}\mu|\dif x+2\int_0^\infty|w_x|m^{(0)}\dif x\Bigg)^2\nonumber
		\\[4pt]
		&\leq\frac{1}{2}\Bigg(\int_0^\infty|u_x^{(0)}\mu|\dif x\Bigg)^2+2\Bigg(\int_0^\infty|w_x|m^{(0)}\dif x\Bigg)^2\nonumber
		\\[4pt]
		&\leq\frac{1}{2}\enVert[1]{u_x^{(0)}}_{\infty}\Bigg(\int_0^\infty |\mu|\dif x\Bigg)^2+2\int_0^\infty w_x^2m^{(0)}\dif x.\nonumber
	\end{align}
	The first inequality follows from a combination of using the triangle inequality, unpacking the definition of $G$, and recalling that $\int_0^\infty m^{(0)}\dif x\leq 1$.
	The second inequality uses the fact that $\big(\sum_{k=1}^nx_k\big)^2\leq n\sum_{k=1}^nx_k^2$.
	The third inequality is a consequence of the Cauchy-Schwarz inequality.
	Using this initial estimate on $\eta_2$, we find that
	\begin{align}\label{eta_2^2Estimate}
		\int_{t_2}^{t_1}e^{-\lambda t}\eta_2(t)^2\dif t&\leq\frac{1}{2}\enVert[1]{u_x^{(0)}}_{\infty}\int_{t_2}^{t_1}e^{-\lambda t}\Bigg(\int_0^\infty |\mu|\dif x\Bigg)^2\dif t+2\int_{t_2}^{t_1}e^{-\lambda t}\int_0^\infty w_x^2m^{(0)}\dif x\dif t
		\\[4pt]
		&\leq C_2\kappa'+C_3\enVert[1]{\Phi}_{L_t^\infty(L_x^1)}\enVert[1]{w_x}_{\infty}\nonumber
	\end{align}
	where the constants $C_2,C_3\in(0,\infty)$ depend only on $\enVert[1]{u_x^{(0)}}_{\infty}$, $\sigma$, $\lambda$, and $T$, and we define
	\begin{equation}\label{def:kappa'}
		\kappa':=\enVert{\Psi}_\infty+\enVert{\Phi}_{L_t^\infty(L_x^1)}^2.
	\end{equation}
	The second inequality in \eqref{eta_2^2Estimate} is a consequence of Proposition~\ref{prop:L2timeL1space_mu}, followed by an application of Proposition~\ref{prop:EnergyEstimate_w}.
	Returning to \eqref{eta_2C-S}, it follows that
	\begin{align}\label{eta_2Estimate}
		\int_{t_2}^{t_1}e^{-\lambda t}\eta_2(t)\dif t&\leq \Big(C_2\kappa'+C_3\enVert[1]{\Phi}_{L_t^\infty(L_x^1)}\enVert[1]{w_x}_{\infty}\Big)^{1/2}\Bigg(\int_{t_1}^{t_2}e^{-\lambda t}\dif t\Bigg)^{1/2}
		\\[4pt]
		&\leq C_1\Big(C_2\kappa'+C_3\enVert[1]{\Phi}_{L_t^\infty(L_x^1)}\enVert[1]{w_x}_{\infty}\Big)^{1/2}|t_1-t_2|^{1/2}\nonumber
	\end{align}
	with $C_1$ as in \eqref{eta_1Estimate}.
	Hence \eqref{muDualEstimate1} and the two estimates in \eqref{eta_1Estimate} and \eqref{eta_2Estimate} show that
	\begin{align}\label{muDualEstimate2}
		\int_0^\infty\widetilde{\psi}(x,t_1)&\mu(x,t_1)\dif x-\int_0^\infty\widetilde{\psi}(x,t_2)\mu(x,t_2)\dif x
		\\[4pt]
		&\leq C\Big(\big(\kappa'+\enVert[1]{\Phi}_{L_t^\infty(L_x^1)}\big)^{1/2}+\enVert[1]{\Phi}_{L_t^\infty(L_x^1)}\Big)\big(1+\enVert[1]{w_x}_{\infty}\big)^{1/2}|t_1-t_2|^{\alpha/2}\enVert[1]{\varphi}_{\s{C}_\diamond^{1+\alpha}},\nonumber
	\end{align}
	where
	\begin{equation}\label{def:muDual_C}
		C_\alpha:=C_0C_1\big((C_2+C_3)^{1/2}+1\big).
	\end{equation}

	Now, call
	\begin{equation}
		\kappa:=\big(\kappa'+\enVert[1]{\Phi}_{L_t^\infty(L_x^1)}\big)^{1/2}+\enVert[1]{\Phi}_{L_t^\infty(L_x^1)},
	\end{equation}
	and set $t_2=0$, so that $\mu(x,t_2)=0$. Hence \eqref{muDualEstimate2} implies
	\begin{equation}
		\int_0^\infty\mu(x,t_1)\varphi(x)\dif x\leq C\kappa e^{\lambda T}\big(1+\enVert[1]{w_x}_{\infty}\big)^{1/2}T^{\alpha/2}\enVert[1]{\varphi}_{\s{C}_\diamond^{1+\alpha}}.
	\end{equation}
	In particular, since $t_1$ and $\varphi$ were arbitrary, we have
	\begin{equation}\label{LinftyTimeHolderDual_mu}
		\enVert[1]{\mu}_{L_t^\infty(C_\diamond^{1+\alpha})^\ast}\leq C_\alpha\kappa \big(1+\enVert[1]{w_x}_{\infty}\big)^{1/2},
	\end{equation}
	with $C_\alpha' :=Ce^{\lambda T}T^{\alpha/2}$.\\
	Again let $t_1,t_2\in[0,T]$ be arbitrary and without loss of generality take $t_1\geq t_2$.
	Then
	\begin{align}\label{Step2_InitialEstimate}
		\int_0^\infty \widetilde{\psi}(x,t_1)\big(\mu(x,t_1)-\mu(x,t_2)\big)\dif x&=\int_0^\infty\widetilde{\psi}(x,t_1)\mu(x,t_1)\dif x-\int_0^\infty\widetilde{\psi}(x,t_2)\mu(x,t_2)\dif x
		\\[4pt]
		&\quad+\int_0^\infty\big(\widetilde{\psi}(x,t_2)-\widetilde{\psi}(x,t_1)\big)\mu(x,t_2)\dif x.\nonumber
	\end{align}
	For the last integral, we estimate
	\begin{align}
		\int_0^\infty\big(\widetilde{\psi}(x,t_2)-\widetilde{\psi}(x,t_1)\big)\mu(x,t_2)\dif x&\leq \enVert[1]{\widetilde{\psi}}_{\s{C}^{\alpha/4}([0,t_1];\s{C}_\diamond^{1+\alpha/2})}\enVert[1]{\mu}_{L_t^\infty(\s{C}^{1+\alpha/2})^\ast}|t_1-t_2|^{\alpha/4}
		\\[4pt]
		&\leq C_{\alpha/2}'\kappa\enVert[1]{\varphi}_{\s{C}_\diamond^{1+\alpha}}(1+\enVert[1]{w_x}_{\infty})^{1/2}|t_1-t_2|^{\alpha/4},\nonumber
	\end{align}
	where the last inequality follows from \eqref{LinftyTimeHolderDual_mu} and \eqref{eq:psi estimate}.
	Combining this with what we found in \eqref{muDualEstimate2}, the equation in \eqref{Step2_InitialEstimate} implies
	\begin{equation}
		\int_0^\infty e^{-\lambda t_1}\varphi(x)\big(\mu(x,t_1)-\mu(x,t_2)\big)\dif x\leq (C'_{\alpha/2}+C_\alpha T^{\alpha/4})\kappa\enVert[1]{\varphi}_{\s{C}_\diamond^{1+\alpha}}(1+\enVert[1]{w_x}_{\infty})^{1/2}|t_1-t_2|^{\alpha/4}.
	\end{equation}
	Hence, as $t_1$, $t_2$, and $\varphi$ were arbitrary,
	\begin{equation}
		\enVert[1]{\mu}_{\s{C}^{\alpha/4}([0,T];(\s{C}^{1+\alpha})^\ast)}\leq (C'_{\alpha/2}+C_\alpha T^{\alpha/4})\kappa(1+\enVert[1]{w_x}_{\infty})^{1/2}.
	\end{equation}
	This is the desired result.
\end{proof}

Notice that any $\phi \in \s{C}^{1+\alpha}$ can be written as the sum of a constant and an element of $\s{C}^{1+\alpha}_\diamond$, since $\phi - \phi(0) \in \s{C}^{1+\alpha}_\diamond$.
Therefore the following proposition will complete our estimates of $\mu$ in the dual of $\s{C}^{1+\alpha}$.
\begin{proposition}\label{prop:HolderEstimates_mu_2}
	Let $(w,\mu)$ satisfy \eqref{abstract_system}. 
	Define
	\begin{equation}
		\eta(t) := \int_0^\infty \mu(x,t)\dif x.
	\end{equation}
	Let $\varepsilon_0$ satisfy \eqref{Assumption:epsilon_small}. Then for $\varepsilon\in[0,\varepsilon_0]$ and $\alpha \in \intoc{0,1/6}$, there exists a constant $\kappa\in[0,\infty)$, and a constant $C\in(0,\infty)$, depending solely on $\enVert[1]{u_x^{(0)}}_{\infty}$, $\sigma$, $\lambda$, $\alpha$ and $T$ such that for all $\varepsilon\in[0,\varepsilon_0]$
	\begin{equation} \label{eq:eta holder}
		\enVert{\eta}_{\s{C}^{\alpha}([0,T])} \leq C\kappa\big(1 + \enVert[1]{w_x}_{\infty}\big)^{5/6}.
	\end{equation}
	Moreover, the constant $\kappa$ depends exclusively on $\enVert[1]{\Phi}_{L_t^\infty(L_x^1)}$ and $\enVert[1]{\Psi}_{\infty}$ in such a way that $\kappa=0$ whenever $\enVert[1]{\Phi}_{L_t^\infty(L_x^1)}=\enVert[1]{\Psi}_{\infty}=0$.
\end{proposition}
\begin{proof}
	We begin by applying \eqref{eq:mu duhamel} with the identifications
	$\mu:=\mu$, $b:=F^{(0)}$, and\\ $\nu:=\Phi+\frac{1}{2}(G-w_x)m^{(0)}$ to get
	\begin{equation} \label{eq:mn duhamel}
		\mu(x,t) = 
		I_1(x,t) + I_2(x,t),
	\end{equation}
	where
	\begin{equation}
		\begin{split}
			I_1(x,t) &:= \int_{0}^{t}\int_0^\infty \del{\dpd{S}{x}(x-y,t-s) + \dpd{S}{x}(x+y,t-s)}F^{(0)}(y,s)\mu(y,s)\dif y \dif s,\\
			I_2(x,t) &:= \int_{0}^{t}\int_0^\infty \del{\dpd{S}{x}(x-y,t-s) + \dpd{S}{x}(x+y,t-s)}\nu(y,s)\dif y \dif s.
		\end{split}
	\end{equation}
	Our first step will be to prove a bound on the functional
	\begin{equation}
		f(t):=\int_0^\infty|\mu(x,t)|\dif x.
	\end{equation}
	By \eqref{Phi+(G-wx)m0 estimate} and H\"older's inequality, we have
	\begin{equation} \label{eq:K + Fm0 estimate 1}
		\int_0^\infty \envert{\nu(x,t)} \dif x\leq \enVert[1]{\Phi}_{L_x^1} + \enVert[1]{u_x^{(0)}}_\infty f(t) +2\del{\int_0^\infty \envert{w_x}^2 m^{(0)} \dif x}^{1/2}.
	\end{equation}
	We apply \eqref{eq:K + Fm0 estimate 1} and \eqref{Integral_S_x} to \eqref{eq:mn duhamel}
	and get
	\begin{equation} \label{eq:f estimate 0}
		f(t)\leq \int_{0}^{t}(t-s)^{-1/2}\del{C_1 f(s)+g(s)}\,\dif s,
	\end{equation}
	where
	\begin{equation}
		C_1 = c'\del{\enVert[1]{u_x^{(0)}}_\infty + \enVert[1]{F^{(0)}}_\infty}
		\,\,\text{ and }\,\,
		g(s) = \enVert[1]{\Phi}_{L_x^1} + 2\Bigg(\int_0^\infty \abs{w_x}^2 m^{(0)} \dif x\Bigg)^{1/2},
	\end{equation}
	and $c'$ is as in \eqref{c'_constant}. Multiply both sides of \eqref{eq:f estimate 0} by $e^{-\lambda t}$ for some $\lambda > 0$ to be chosen.
	Define $f_\lambda(t) = e^{-\lambda t}f(t)$ and $g_\lambda(t) = e^{-\lambda t}g(t)$.
	We see
	\begin{equation} \label{eq:f estimate 1}
		f_\lambda(t)\leq \int_{0}^{t}(t-s)^{-1/2}e^{-\lambda(t-s)}\del{C_1 f_\lambda(s)+g_\lambda(s)}\dif s.
	\end{equation}
	We wish to estimate $\int_{0}^{t}(t-s)^{-1/2}e^{-\lambda(t-s)}g_\lambda(s)\dif s$.	
	Since
	\begin{equation}
		\int_{0}^{t}(t-s)^{-1/2}e^{-\lambda(t-s)}\dif s
		= \int_{0}^{t}s^{-1/2}e^{-\lambda s}\dif s
		\leq \lambda^{-1/2}\int_{0}^{\infty}s^{-1/2}e^{- s}\dif s = \lambda^{-1/2}\pi^{1/2},
	\end{equation}
	we have
	\begin{multline}
		\int_{0}^{t}(t-s)^{-1/2}e^{-\lambda(t-s)}g_\lambda(s)\dif s\\
		\leq \enVert[1]{\Phi}_{L_t^\infty(L_x^1)}\lambda^{-1/2}\pi^{1/2}
		+ 2\int_{0}^{t}(t-s)^{-1/2}e^{-\lambda(t-s)}\Bigg(\int_0^\infty w_x^2 m^{(0)}\dif x\Bigg)^{1/2}\dif s.
	\end{multline}
	On the other hand, using Proposition \ref{prop:EnergyEstimate_w}, H\"older's inequality, and recalling that $\int_0^\infty m^{(0)}\dif x \leq 1$, we get
	\begin{equation}
		\begin{split}
			\int_{0}^{t}&(t-s)^{-1/2}e^{-\lambda(t-s)}\del{\int_0^\infty w_x^2 m^{(0)}\dif x}^{1/2}\dif s\\
			&\leq \enVert{w_x}_{L^\infty_{x,t}}^{1/2}\int_{0}^{t}(t-s)^{-1/2}e^{-\lambda(t-s)}\del{\int_0^\infty \envert{w_x}^{2} m^{(0)}\dif x}^{1/4}\dif s\\
			&\leq \enVert{w_x}_{L^\infty_{x,t}}^{1/2}\del{\int_{0}^{t}(t-s)^{-2/3}e^{-\frac{4}{3}\lambda(t-s)}\dif s}^{3/4}
			\del{\int_{0}^{t}\int_0^\infty \envert{w_x}^{2} m^{(0)}\dif x\dif s}^{1/4}\\
			&\leq \enVert{w_x}_{L^\infty_{x,t}}^{1/2}\lambda^{-1/4}\Gamma(1/3)\del{C_2\enVert{\Phi}_{L_t^\infty(L_x^1)}\enVert{w_x}_{\infty} + C_3\kappa'}^{1/4},
		\end{split}
	\end{equation}
	where $C_2$ and $C_3$ are constants depending only on $\sigma$, $\lambda$, and $T$, and $\kappa'$ is as in \eqref{def:kappa'}. We combine these estimates to get
	\begin{multline} \label{eq:f estimate 2}
		f_\lambda(t)\leq C_1\int_{0}^{t}(t-s)^{-1/2}e^{-\lambda(t-s)} f_\lambda(s)\dif s\\
		+ \enVert[1]{\Phi}_{L_t^\infty(L_x^1)}\lambda^{-1/2}\pi^{1/2}
		+ 2\enVert{w_x}_{\infty}^{1/2}\lambda^{-1/4}\Gamma(1/3)\del{C_2\enVert{\Phi}_{L_t^\infty(L_x^1)}\enVert{w_x}_{\infty} + C_3\kappa'}^{1/4}.
	\end{multline}
	Taking the supremum, we get
	\begin{multline} \label{eq:f estimate 3}
		\sup_{t\in[0,T]}f_\lambda(t)\leq C_1\lambda^{-1/2}\pi^{1/2}\sup_{s\in[0,T]}f_\lambda(s)\\
		+ \enVert[1]{\Phi}_{L_t^\infty(L_x^1)}\lambda^{-1/2}\pi^{1/2}
		+ 2\enVert{w_x}_{\infty}^{1/2}\lambda^{-1/4}\Gamma(1/3)\del{C_2\enVert{\Phi}_{L_t^\infty(L_x^1)}\enVert{w_x}_{\infty} + C_3\kappa'}^{1/4}.
	\end{multline}
	Then setting $\lambda = (2C_1 \pi^{1/2})^2$, we deduce
	\begin{equation}
		\sup_{t\in[0,T]}f_\lambda(t)\leq 2\enVert[1]{\Phi}_{L_t^\infty(L_x^1)}\lambda^{-1/2}\pi^{1/2}
		+ 4\enVert{w_x}_{\infty}^{1/2}\lambda^{-1/4}\Gamma(1/3)\del{C_2\enVert{\Phi}_{L_t^\infty(L_x^1)}\enVert{w_x}_{\infty} + C_3\kappa'}^{1/4},
	\end{equation}
	which yields
	\begin{equation} \label{eq:abs mn bound}
		\sup_{t\in[0,T]} \int_0^\infty \abs{\mu(x,t)}\dif x \leq A\kappa''\big(\enVert{w_x}_{\infty}+1\big)^{3/4} + B\enVert{\Phi}_{L_t^\infty(L_x^1)}
	\end{equation}
	for some constants $A,B\in(0,\infty)$, depending only on $\enVert[1]{u_x^{(0)}}_{\infty}$, $\sigma$, $\lambda$, and $T$, and
	\begin{equation}
		\kappa'':=\kappa'+\enVert{\Phi}_{L_t^\infty(L_x^1)}.
	\end{equation}

	We now turn to H\"{o}lder estimates.
	By integrating \eqref{eq:mn duhamel} with respect to $x$, we discover that
	\begin{equation}
		\eta(t) = \eta_1(t) + \eta_2(t),
	\end{equation}
	where
	\begin{equation}
		\begin{split}
			\eta_1(t) &= -2\int_0^t \int_0^\infty S(y,t-s)F^{(0)}(y,s)\mu(y,s)\dif y \dif s,\\
			\eta_2(t) &= -2\int_0^t \int_0^\infty S(y,t-s)\nu(y,s)\dif y \dif s.
		\end{split}
	\end{equation}
	We want H\"older estimates on each term.
	Let $0 \leq t_0 \leq t_1 \leq T$.
	For $\eta_1$, write
	\begin{multline}
		\eta_1(t_1) - \eta_1(t_0)
		= - 2\int_{t_0}^{t_1} \int_0^\infty S(y,t_1-s)F^{(0)}(y,s)\mu(y,s)\dif y \dif s\\
		- 2\int_0^{t_0} \int_0^\infty \int_{t_0}^{t_1} \dpd{S}{t}(y,t-s)F^{(0)}(y,s)\mu(y,s)\dif t\dif y\dif s.
	\end{multline}
	Note that
	\begin{equation}
		\dpd{S}{t}(x,t) = (2\sigma^2 \pi)^{-1/2}t^{-3/2}\del{\frac{x^2}{2\sigma^2 t} - \frac{1}{2}}\exp\cbr{- 	\frac{x^2}{2\sigma^2 t}}.
	\end{equation}
	Since $xe^{-x} \leq e^{-1} \leq 1/2$ for all $x$, it follows that
	\begin{equation}
		\abs{\dpd{S}{t}(x,t)} \leq (2\sigma^2 \pi)^{-1/2}t^{-3/2} \quad \forall x.
	\end{equation}
	Let $c' = 2(2\sigma^2 \pi)^{-1/2}$.
	We get
	\begin{multline}
		\abs{\eta_1(t_1) - \eta_1(t_0)}
		\leq c'\enVert[1]{F^{(0)}}_\infty \int_{t_0}^{t_1} (t_1 - s)^{-1/2} \int_0^\infty \envert{\mu(y,s)}\dif y \dif s\\
		+ c'\enVert[1]{F^{(0)}}_\infty \int_0^{t_0} \int_0^\infty \int_{t_0}^{t_1} (t-s)^{-3/2} \envert{\mu(y,s)}\dif t\dif y\dif s.
	\end{multline}
	Using \eqref{eq:abs mn bound} and computing the remaining integrals, we obtain
	\begin{equation} \label{eq:eta1 estimate}
		\abs{\eta_1(t_1) - \eta_1(t_0)} \leq C\enVert[1]{F^{(0)}}_\infty\del{A\kappa'\big(\enVert{w_x}_{\infty}+1\big)^{3/4} + B\enVert{\Phi}_{L_t^\infty(L_x^1)}}(t_1-t_0)^{1/2}
	\end{equation}
	for some constant $C\in(0,\infty)$.
	
	We proceed similarly for $\eta_2$.
	Write
	\begin{multline}
		\eta_2(t_1) - \eta_2(t_0)
		= 2\int_{t_0}^{t_1} \int_0^\infty S(y,t_1-s)\nu(y,s)\dif y \dif s\\
		+ 2\int_0^{t_0} \int_0^\infty \int_{t_0}^{t_1} \dpd{S}{t}(y,t-s)\nu(y,s)\dif t\dif y\dif s.
	\end{multline}
	By \eqref{Phi+(G-wx)m0 estimate} and \eqref{eq:abs mn bound}, we have
	\begin{equation}
		\int_0^\infty \abs{\nu(x,t)} \dif x \leq \tilde A \kappa'' \big(\enVert{w_x}_{\infty}+1\big)^{3/4} + \tilde B \enVert{\Phi}_{L_t^\infty(L_x^1)}
		+ 2\del{\int_0^\infty \envert{w_x(x,t)}^2 m^{(0)}(x,t) \dif x}^{1/2},
	\end{equation}
	where $\tilde A$ and $\tilde B$ are again constants depending only on $\enVert[1]{u_x^{(0)}}_{\infty}$, $\sigma$, $\lambda$, and $T$.
	Using the same reasoning as above, we get
	\begin{multline}
		\abs{\eta_2(t_1) - \eta_2(t_0)}
		\leq C\del{\tilde A\kappa' \big(\enVert{w_x}_{\infty}+1\big)^{3/4} + \tilde B \enVert{\Phi}_{L_t^\infty(L_x^1)}}(t_1-t_0)^{1/2}\\
		+ 2c'\int_{t_0}^{t_1} (t_1-s)^{-1/2}\del{\int_0^\infty \envert{w_x(x,s)}^2 m^{(0)}(x,s) \dif x}^{1/2} \dif s\\
		+ 2c'\int_0^{t_0}  \int_{t_0}^{t_1} (t-s)^{-3/2}\del{\int_0^\infty \envert{w_x(x,s)}^2 m^{(0)}(x,s) \dif x}^{1/2}\dif t \dif s.
	\end{multline}
	Using H\"older's inequality and Proposition \ref{prop:EnergyEstimate_w}, we get
	\begin{equation}
		\begin{split}
			\int_0^{t_0}  &\int_{t_0}^{t_1} (t-s)^{-3/2}\del{\int_0^\infty \abs{w_x(x,s)}^2 m^{(0)}(x,s) \dif x}^{1/2}\dif t \dif s\\
			&\leq \enVert{w_x}_{\infty}^{2/3} \int_0^{t_0}  \int_{t_0}^{t_1} (t-s)^{-3/2}\del{\int_0^\infty \abs{w_x(x,s)}^2 m^{(0)}(x,s) \dif x}^{1/6}\dif t \dif s\\
			&\leq \enVert{w_x}_{\infty}^{2/3} \del{\int_0^{t_0}  \int_{t_0}^{t_1} (t-s)^{-9/5}\dif t \dif s}^{5/6}
			\del{\int_0^{t_0}  \int_{t_0}^{t_1} \int_0^\infty \abs{w_x(x,s)}^2 m^{(0)}(x,s) \dif x\dif t \dif s}^{1/6}\\
			&\leq C\enVert{w_x}_{\infty}^{2/3}(t_1-t_0)^{1/6}\del{C_2\enVert{\Phi}_{L_t^\infty(L_x^1)}\enVert{w_x}_{\infty} + C_3\kappa'}^{1/6}.
		\end{split}
	\end{equation}
	By a similar computation, we get
	\begin{multline}
		\int_{t_0}^{t_1} (t_1-s)^{-1/2}\del{\int_0^\infty \abs{w_x(x,s)}^2 m^{(0)}(x,s) \dif x}^{1/2} \dif s\\
		\leq C\enVert{w_x}_{\infty}^{2/3}(t_1-t_0)^{1/3}\del{C_2\enVert{\Phi}_{L_t^\infty(L_x^1)}\enVert{w_x}_{\infty} + C_3\kappa'}^{1/6}.
	\end{multline}
	We deduce
	\begin{equation} \label{eq:eta2 estimate}
		\begin{split}
			\abs{\eta_2(t_1) - \eta_2(t_0)}
		&\leq CT^{1/3}\del{\tilde A\kappa''\big(\enVert{w_x}_{\infty}+1)^{3/4} + \tilde B \enVert{\Phi}_{L_t^\infty(L_x^1)}}(t_1-t_0)^{1/6}
		\\[4pt]
		&\quad+ C(1+T^{1/6})\enVert{w_x}_{\infty}^{2/3}(t_1-t_0)^{1/6}\del{C_2\enVert{\Phi}_{L_t^\infty(L_x^1)}\enVert{w_x}_{\infty} + C_3\kappa'}^{1/6}.
		\end{split}
	\end{equation}
	Combining \eqref{eq:eta1 estimate} and \eqref{eq:eta2 estimate}, we can find some constant $\widetilde C \in(0,\infty)$, depending only on $\enVert[1]{F^{(0)}}_{\infty}$, $\enVert[1]{u_x^{(0)}}_{\infty}$, $\sigma$, $\lambda$, and $T$, such that
	\begin{equation}
		\envert{\eta(t_1) - \eta(t_0)} \leq \widetilde C\kappa\del{\enVert{w_x}_{\infty} + 1}^{5/6}(t_1 - t_0)^{1/6},
	\end{equation}
	where
	\begin{equation}
		\kappa:=\big(\kappa'+\enVert{\Phi}_{L_t^\infty(L_x^1)}\big)^{1/6}+2\big(\kappa''+\enVert{\Phi}_{L_t^\infty(L_x^1)}\big).
	\end{equation}
	This yields \eqref{eq:eta holder}, as desired.
\end{proof}

With Propositions \ref{prop:HolderEstimates_mu} and \ref{prop:HolderEstimates_mu_2} established, we now have the complete duality estimate on $\mu$.

\begin{corollary}\label{cor:DualEstimate_mu}
	Let $(w,\mu)$ satisfy \eqref{abstract_system}. Let $\varepsilon_0$ satisfy \eqref{Assumption:epsilon_small}. Then for $\varepsilon\in[0,\varepsilon_0]$, one has $\mu\in\s{C}^{\alpha/4}\big([0,T];(C^{1+\alpha})^\ast\big)$ with
	\begin{equation}
		\enVert[1]{\mu}_{\s{C}^{\alpha/4}\big([0,T];(\s{C}^{1+\alpha})^\ast\big)}\leq C\kappa\big(1+\enVert[1]{w_x}_{\infty}\big)^{5/6},
	\end{equation}
	for some constant $C\in(0,\infty)$, depending exclusively on $\enVert[1]{u_x^{(0)}}_{\infty}$, $\sigma$, $\lambda$, $\alpha$ and $T$; and $\kappa\in[0,\infty)$, depending solely on $\enVert[1]{\Phi}_{L_t^\infty(L_x^1)}$ and $\enVert[1]{\Psi}_{\infty}$ in such a way that $\kappa=0$ whenever $\enVert[1]{\Phi}_{L_t^\infty(L_x^1)}=\enVert[1]{\Psi}_{\infty}=0$.
\end{corollary}

\subsection{Full $\s{C}^{2+\alpha,1+\alpha/2}$ regularity} \label{sec:full regularity}

We are now in a position to establish full parabolic regularity in classical spaces for solutions to System \eqref{abstract_system}.

\begin{theorem}\label{thm:HolderRegularity_w/mu}
	Let $\alpha \in (0,1/6]$ be such that $(u^{(0)},m^{(0)}) \in \s{C}^{2+2\alpha,1+\alpha}$.
	Let $(w,\mu)$ satisfy \eqref{abstract_system} with $\Psi\in\s{C}^{\alpha,\alpha/2}$ and $\Phi\in\s{C}^{1+\alpha,\alpha/2}$. Let $\varepsilon_0$ satisfy \eqref{Assumption:epsilon_small}.
	Then for $\varepsilon\in[0,\varepsilon_0]$, one has
	\begin{equation}
		w\in\s{C}^{2+\alpha, 1+\alpha/2}\big([0,\infty)\times[0,T]\big)\,\,\text{ and }\,\, \mu\in\s{C}^{2+\alpha, 1+\alpha/2}\big([0,\infty)\times[0,T]\big).
	\end{equation}
	Furthermore, there exist constants $C\in(0,\infty)$ and $\kappa\in[0,\infty)$, depending solely on the data such that
	\begin{equation}\label{HolderRegularityEstimates_w/mu}
		\enVert{w}_{\s{C}^{2+\alpha, 1+\alpha/2}}\leq C\kappa\,\,\text{ and }\,\,\enVert{\mu}_{\s{C}^{2+\alpha, 1+\alpha/2}}\leq C\kappa
	\end{equation}
	with $\kappa$ depending exclusively on $\enVert{\Psi}_{\s{C}^{\alpha,\alpha/2}}$ and $\enVert{\Phi}_{\s{C}^{1+\alpha,\alpha/2}}$ in such a way that $\kappa=0$ whenever $\Phi=\Psi=0$.
	In addition, we have 
	$\mu \in \s{C}^{1/2}\del{[0,T];L^1((0,\infty)}$ with an estimate
	\begin{equation} \label{eq:holder t L1x}
		\enVert{\mu}_{\s{C}^{1/2}\del{[0,T];L^1((0,\infty)}} \leq C\kappa.
		\end{equation}
\end{theorem}
\begin{proof}
	Begin by defining
	\begin{equation}
		\begin{array}{lll}
			\displaystyle \eta_1(t):=\int_0^\infty w_xm^{(0)}\dif x &\text{and}&\displaystyle \eta_2(t):=\int_0^\infty u_x^{(0)}\mu\dif x,
		\end{array}
	\end{equation}
	so that after expanding terms (i) of System \eqref{abstract_system} reads
	\begin{equation}
		w_t+\frac{\sigma^2}{2}w_{xx}-rw-F^{(0)}w_x+\beta(\varepsilon\xi)\big(\eta_1(t)+\eta_2(t)\big)+\Psi=0.
	\end{equation}
	Note that $\beta(\varepsilon\xi)\in\s{C}^{\alpha/2}([0,T])$, and by assumption, $\Psi\in\s{C}^{\alpha,\alpha/2}([0,\infty),[0,T])$.
	As such, using classical estimates \cite[Theorem IV.5.1]{ladyzhenskaia1968linear}, we find
	\begin{equation}\label{HolderEstimate1_w}
		\enVert[1]{w}_{\s{C}^{2+\alpha,1+\alpha/2}}\leq C_0\big(\enVert[1]{\Psi}_{\s{C}^{\alpha,\alpha/2}}+\enVert{\eta_1}_{\s{C}^{\alpha/2}}+\enVert{\eta_2}_{\s{C}^{\alpha/2}}\big),
	\end{equation}
	where the constant $C_0\in(0,\infty)$ depends on the data.
	
	Focus on $\eta_2$ first. We estimate
	\begin{align}\label{Eta2Estimate_wHolder}
		|\eta_2(t_1)-\eta_2(t_2)|&\leq\Big|\int_0^\infty\big(u_x^{(0)}(x,t_1)-u_x^{(0)}(x,t_2)\big)\mu(x,t_1)\dif x\Big|
		\\[4pt]
		&\quad+\Big|\int_0^\infty u_x^{(0)}(x,t_2)\big(\mu(x,t_1)-\mu(x,t_2)\big)\dif x\Big|\nonumber
		\\[4pt]
		&\leq\enVert[1]{\mu(\cdot, t_1)}_{(\s{C}^{1+2\alpha})^\ast}\enVert[1]{u_x^{(0)}(\cdot,t_1)-u_x^{(0)}(\cdot,t_2)}_{\s{C}^{1+2\alpha}}\nonumber
		\\[4pt]
		&\quad+\enVert[1]{u_x^{(0)}(\cdot, t_2)}_{\s{C}^{1+2\alpha}}\enVert[1]{\mu(\cdot,t_1)-\mu(\cdot,t_2)}_{(\s{C}^{1+2\alpha})^\ast}\nonumber
		\\[4pt]
		&\leq 2\enVert[1]{\mu}_{\s{C}^{\alpha/2}([0,T];(\s{C}^{1+2\alpha})^\ast)}\enVert[1]{u^{(0)}}_{\s{C}^{2+2\alpha,1+\alpha}}|t_1-t_2|^{\alpha/4}.\nonumber
	\end{align}
	Hence by Corollary \ref{cor:DualEstimate_mu},
	\begin{align}
		\enVert[1]{\eta_2}_{\s{C}^{\alpha/2}([0,T])}&\leq2\enVert[1]{\mu}_{\s{C}^{\alpha/2}([0,T];(\s{C}^{1+2\alpha})^\ast)}\enVert[1]{u^{(0)}}_{\s{C}^{2+2\alpha,1+\alpha}}
		\\[4pt]
		&\leq C_1\kappa'\big(1+\enVert[1]{w_x}_{\infty}\big)^{5/6}\nonumber
	\end{align}
	for some constant $C_1\in(0,\infty)$, depending solely on $\enVert[1]{u^{(0)}}_{\s{C}^{2+2\alpha,1+\alpha}}$ and the data, and $\kappa'\in[0,\infty)$, depending exclusively on $\enVert[1]{\Phi}_{L_t^\infty(L_x^1)}$ and $\enVert[1]{\Psi}_{\infty}$ in such a way that forces $k'=0$ whenever $\enVert[1]{\Phi}_{L_t^\infty(L_x^1)}=\enVert[1]{\Psi}_{\infty}=0$.
	On the other hand, since $m^{(0)}$ satisfies a (homogeneous) Fokker-Planck equation, by classical arguments (cf.~\cite[Lemma 3.1]{graber2021master}) combined with the estimates of Proposition \ref{prop:HolderEstimates_mu_2} (cf.~\cite[Lemma 3.7]{graber2021master}), we have an estimate on the norm $\enVert[1]{m^{(0)}}_{\s{C}^{\alpha/2}([0,T];(\s{C}^{\alpha})^\ast)}$.
	Now a similar calculation as in \eqref{Eta2Estimate_wHolder} shows
	\begin{equation}
		\enVert[1]{\eta_1(t)}_{C^{\alpha/2}}\leq C_2\enVert[1]{w_x}_{\s{C}^{\alpha,\alpha/2}},
	\end{equation}
	where the constant $C_2\in(0,\infty)$ depends solely on $\enVert[1]{m^{(0)}}_{\s{C}^{\alpha/4}([0,T];(\s{C}^{\alpha/2})^\ast)}$.
	Therefore, returning to \eqref{HolderEstimate1_w}, we have
	\begin{align}\label{HolderEstimate2_w}
		\enVert[1]{w}_{\s{C}^{2+\alpha, 1+\alpha/2}}&\leq C_0\Big(\enVert[1]{\Psi}_{\s{C}^{\alpha,\alpha/2}}+C_1\kappa'\big(1+\enVert[1]{w_x}_{\infty}\big)^{5/6}+C_2\enVert[1]{w_x}_{\s{C}^{\alpha,\alpha/2}}\Big)
		\\[4pt]
		&\leq  C_0\big(\enVert{\Psi}_{\s{C}^{\alpha,\alpha/2}}+C_1\kappa'+(C_1\kappa_1+C_2)\enVert[1]{w_x}_{\s{C}^{\alpha,\alpha/2}}\big)\nonumber
		\\[4pt]
		&\leq \widetilde C\big(\kappa_1+\enVert[1]{w_x}_{\s{C}^{\alpha,\alpha/2}}\big)\nonumber
	\end{align}
	where $\widetilde C:=C_0\big(1+C_1(\kappa_1+1)+C_2\big)$, and $\kappa_1:=\enVert{\Psi}_{\s{C}^{\alpha,\alpha/2}}+\kappa'$. By interpolation of H\"{o}lder spaces, there exists a constant $C'>0$ such that
	\begin{equation}\label{Interpolation}
		\enVert[1]{w_x}_{\s{C}^{\alpha,\alpha/2}}\leq \frac{1}{2\widetilde C}\enVert[1]{w}_{\s{C}^{2+\alpha,1+\alpha/2}}+C'\enVert[1]{w}_{\infty},
	\end{equation}
	where $\widetilde C$ is as in the last line of  \eqref{HolderEstimate2_w}.
	As such, \eqref{Interpolation} along with \eqref{HolderEstimate2_w} show that
	\begin{align}\label{un_Holder_regularity_C}
		\enVert[1]{w}_{\s{C}^{2+\alpha,1+\alpha/2}}&\leq 2\widetilde C\big(\kappa_1+C'\enVert[1]{w}_{\infty}\big)
		\\[4pt]
		&\leq 2\widetilde C\Big(\kappa_1+C'\big(C_3\kappa''(1+\enVert[1]{w_x}_{\infty})^{1/2}+C_4\enVert{\Psi}_{\infty}\big)\Big)\nonumber
		\\[4pt]
		&\leq C''\kappa_2\Big(1+\enVert[1]{w_x}_{\infty}^{1/2}\Big).\nonumber
	\end{align}
	The second inequality follows from Corollary \ref{cor:max_w}, so that the constants $C_3,C_4\in(0,\infty)$ depend only on the data, and $\kappa''\in[0,\infty)$ depends solely on $\enVert[1]{\Phi}_{L_t^\infty(L_x^1)}$ and $\enVert[1]{\Psi}_{\infty}$ in such a way that $\kappa''=0$ whenever $\enVert[1]{\Phi}_{L_t^\infty(L_x^1)}=\enVert[1]{\Psi}_{\infty}=0$.
	The constants $C''\in(0,\infty)$ and $\kappa_2\in[0,\infty)$ can be given by 
	\begin{equation}
		C'':=2\widetilde C(1+C'C_3+C_4)\,\,\text{ and }\,\,\kappa_2:=\kappa_1+\kappa''+\enVert{\Psi}_\infty.
	\end{equation}
	Subsequently, by Young's inequality,
	\begin{align}
		\enVert[1]{w}_{\s{C}^{2+\alpha,1+\alpha/2}}&\leq C''\kappa_2+\frac{1}{2}\enVert[1]{w}_{\s{C}^{2+\alpha,1+\alpha/2}}+\frac{1}{2}(C''\kappa_2)^2.
	\end{align}
	Thus,
	\begin{equation}\label{eq:wHolder}
		\enVert[1]{w}_{\s{C}^{2+\alpha,1+\alpha/2}}\leq 2C''\kappa_2+(C''\kappa_2)^2\leq C\kappa,
	\end{equation}
	with $C:=(C'')^2+2C''$ and $\kappa:=\kappa_2^2+\kappa_2$. Hence $w$ has the desired regularity with appropriate bounds in terms of the data as stated in \eqref{HolderRegularityEstimates_w/mu}.

	Turning our attention to $\mu$, we begin by observing that (ii) of System \eqref{abstract_system} implies
	\begin{equation}
		\mu_t-\frac{\sigma^2}{2}\mu_{xx}-\frac{1}{2}u_{xx}^{(0)}\mu-F^{(0)}\mu_x=\Phi_x+\frac{1}{2}\Big(Gm_x^{(0)}-w_{xx}m^{(0)}-w_xm_x^{(0)}\Big).
	\end{equation}
	We already have estimates on the coefficients in $\s{C}^{\alpha,\alpha/2}$.
	As for the right-hand side, we estimate $Gm_x^{(0)}$ in $\s{C}^{\alpha,\alpha/2}$ by using
	\begin{equation}
		\begin{split}
			\enVert[1]{G}_{\s{C}^{\alpha,\alpha/2}}&\leq \enVert{\int_0^\infty\big(u_x^{(0)}\mu+w_xm^{(0)}\big)\dif x}_{\s{C}^{\alpha,\alpha/2}}
		\\[4pt]
		&\leq \enVert[1]{u^{(0)}}_{\s{C}^{2+2\alpha,1+\alpha}}\enVert{\mu}_{\s{C}^{\alpha/2}([0,T];(\s{C}^{1+2\alpha})^\ast)}+\enVert{w}_{\s{C}^{2+\alpha,1+\alpha/2}}\enVert{m^{(0)}}_{\s{C}^{\alpha/2}([0,T];(\s{C}^{\alpha})^\ast)}
		\\[4pt]
		&\leq C_1\kappa_1,
		\end{split}
	\end{equation}
	for some $C_1\in(0,\infty)$, depending solely on the data, and $\kappa_1\in[0,\infty)$ depending on $\enVert{\Psi}_{\s{C}^{\alpha,\alpha/2}}$ and $\enVert{\Phi}_{\s{C}^{\alpha,\alpha/2}}$ in such a way that $\kappa_1=0$ whenever $\Phi=\Psi=0$.
	Finally, since $\Phi\in\s{C}^{1+\alpha,\alpha/2}$, we obtain via classical estimates \cite[Theorem IV.5.1]{ladyzhenskaia1968linear}
	\begin{align}
		\enVert[1]{\mu}_{\s{C}^{2+\alpha,1+\alpha/2}}&\leq C_0\Big(\enVert{\Phi_x}_{\s{C}^{\alpha,\alpha/2}}+\enVert[1]{Gm^{(0)}}_{\s{C}^{\alpha,\alpha/2}}+\enVert[1]{w_{xx}m^{(0)}}_{\s{C}^{\alpha,\alpha/2}}+\enVert[1]{w_xm^{(0)}}_{\s{C}^{\alpha,\alpha/2}}\Big)
		\\[4pt]
		&\leq C_0\Big(\enVert{\Phi_x}_{\s{C}^{\alpha,\alpha/2}}+\enVert[1]{G}_{\s{C}^{\alpha,\alpha/2}}\enVert[1]{m^{(0)}}_{\s{C}^{\alpha,\alpha/2}}+\enVert{w}_{\s{C}^{2+\alpha,1+\alpha/2}}\enVert[1]{m^{(0)}}_{\s{C}^{2+\alpha,1+\alpha/2}}\Big)\nonumber
		\\[4pt]
		&\leq C_0\Big(\enVert{\Phi_x}_{\s{C}^{\alpha,\alpha/2}}+C_1\kappa_1+C_2\kappa_2\Big)\nonumber
	\end{align}
	for some constant $C_0\in(0,\infty)$ depending only on the data, some constant $C_2\in(0,\infty)$ depending only on the data and $\enVert[1]{m^{(0)}}_{\s{C}^{2+\alpha,1+\alpha/2}}$, and $\kappa_2\in[0,\infty)$ depending on $\enVert{\Psi}_{\s{C}^{\alpha,\alpha/2}}$ and $\enVert{\Phi_x}_{\s{C}^{\alpha,\alpha/2}}$ in such a way that $\kappa_2=0$ whenever $\Phi=\Psi=0$.
	Therefore,
	\begin{equation}
		\enVert[1]{\mu}_{\s{C}^{2+\alpha,1+\alpha/2}}\leq C\kappa,
	\end{equation}
	with $C:=C_0(1+C_1+C_2)$ and $\kappa:=\enVert{\Phi_x}_{\s{C}^{\alpha,\alpha/2}}+\kappa_1+\kappa_2$.
	This demonstrates the desired regularity for $\mu$ with the required bound given in \eqref{HolderRegularityEstimates_w/mu}.
	
	Finally, we appeal to Proposition \ref{prop:LtInftyLx1_moment_mu} and the estimates just obtained to find that, for all $K > 0$,
	\begin{equation} \label{eq:holder t L1x1}
		\begin{split}
			\int_0^\infty \abs{\mu(x,t_1) - \mu(x,t_2)}\dif x
			&\leq \int_0^K \abs{\mu(x,t_1) - \mu(x,t_2)}\dif x
			+ \frac{1}{K}\int_K^\infty x\abs{\mu(x,t_1) - \mu(x,t_2)}\dif x\\
			&\leq K\enVert{\mu}_{\s{C}^{2,1}}\abs{t_1-t_2}
			+ \frac{C}{K}\del{\enVert{(1+x)\nu}_{L_t^\infty(L_x^1)}},
		\end{split}
	\end{equation}
	where $\nu = \Phi + (G - w_x)m^{(0)}$, and
	where $C$ is a constant depending on $\enVert{F^{(0}}_\infty$ and on the data.
	Pick $K = \abs{t_1 - t_2}^{-1/2}$ and take the supremum over $t_1 \neq t_2$ to get an estimate on the H\"older seminorm of $t \mapsto \mu(\cdot,t) \in L^1_x$.
	Combine this with the estimates above to obtain \eqref{eq:holder t L1x}.
\end{proof}

\section{Existence and uniqueness for the abstract system}\label{sec:L-S_Theorem}

In this section, we use a fixed point argument to demonstrate existence of solutions to \eqref{abstract_system}.
For completeness, we provide a statement of the Leray-Schauder fixed point theorem, taken from \cite[Theorem 11.6, p. 286]{gilbarg2015elliptic}.

\begin{theorem}[The Leray Schauder Fixed Point Theorem]\label{thm:Leray-Schauder}
	Let $\s{X}$ be a Banach space, and let $\s{T}:\s{X}\times[0,1]\to\s{X}$ be a compact mapping satisfying $\s{T}(x,0)=0$ for all $x\in \s{X}$.
	Suppose there exists a constant $C\in(0,\infty)$ such that
	\begin{equation}
		\enVert{x}_{\s X}\leq C\,\,\text{ for all }\,\, (x,\lambda)\in \s X\times[0,1]\,\text{ satisfying } \s{T}(x,\lambda)=x.
	\end{equation}
	Then the mapping $\s{T}_1:\s{X}\to \s{X}$, given by $\s{T}_1(x)=\s{T}(x,1)$, has a fixed point.
\end{theorem}

We now aim to construct a suitable Banach space along with a compact operator to satisfy the hypotheses of Theorem~\ref{thm:Leray-Schauder}.

Fix $\alpha \in (0,1/6]$ such that $u^{(0)},m^{(0)} \in \s{C}^{2+2\alpha,1+\alpha}$.
Define the Banach space $\s{X}:=\s{C}^{\alpha,\alpha/2}\cap L_t^\infty(L_x^1) \cap \s{C}^{\alpha/2}\del{[0,T];(\s{C}^{1+2\alpha})^*}$ with norm
\begin{equation}
	\enVert{m}_{\s{X}}:= \enVert{m}_{\s{C}^{\alpha,\alpha/2}} + \enVert{m}_{L_t^\infty(L_x^1)}
	+ \enVert{m}_{\s{C}^{\alpha/2}\del{[0,T];(\s{C}^{1+2\alpha})^*}}.
\end{equation}

Let $\Psi\in\s{C}^{\alpha,\alpha/2}$ and $\Phi\in\s{C}^{1+\alpha,\alpha/2}\cap L_t^\infty(L_x^1)$ be given functions with $x\Phi\in L_t^\infty(L_x^1)$ and $\Psi(0,T) = 0$.
For any $m\in\s{X}$ and $\lambda\in[0,1]$, let $\mu=\s{T}(m,\lambda)$ be determined by solving the system
\begin{equation}\label{def:L-S_operator}
	\begin{cases}
		\begin{array}{lll}
			(i)&\displaystyle w_t+\frac{\sigma^2}{2}w_{xx}-ru+\lambda\Psi+\lambda F^{(0)}(\varepsilon)\big(G(w_x,m;\varepsilon)-w_x\big)=0, \text{\hspace{.5cm}} &0\leq x< \infty,\,\,0\leq t\leq T\vspace{4pt}\\
			(ii)&\displaystyle \mu_t-\frac{\sigma^2}{2}\mu_{xx}-\lambda\big(F^{(0)}\mu\big)_x=\lambda\Big[\Phi+\frac{1}{2}\big(G(w_x,m;\varepsilon)-w_x\big)m^{(0)}\Big]_x, &0\leq x< \infty,\,\,0\leq t\leq T\vspace{4pt}\\
			(iii)&\displaystyle \mu(x,0)=0,\text{\hspace{.2cm}}w(x,T)=0, &0\leq x< \infty\vspace{4pt}\\
			(iv)&\displaystyle w(0,t)=\mu(0,t)=0, &0\leq t\leq T,
		\end{array}
	\end{cases}
\end{equation}
where we recall that $G$ is defined in \eqref{def:G}.

The subsequent lemma establishes properties of $\s{T}$ that are necessary for applying Theorem \ref{thm:Leray-Schauder}.

\begin{lemma}\label{lem:L-S_hypotheses}
	The operator $\s{T}$, defined in \eqref{def:L-S_operator}, is well-defined, continuous, and compact.
\end{lemma}
\begin{proof} We break the proof into three parts.
	Throughout the proof, we will refer to the results found in \cite{ladyzhenskaia1968linear}, especially Theorems IV.5.1-3, as ``classical results/estimates."
	\\\\
	1. \textit{Well-Definiteness.}
	To show that the operator $\s{T}$ is well-defined, suppose $m\in\s{X}$ and $\lambda\in[0,1]$.
	By Lemma \ref{lem:existence Holder} there exists a unique solution $w \in \s{C}^{2+\alpha,1+\alpha/2}$ to Equation \eqref{def:L-S_operator}(i) (with boundary conditions given in (iii)-(iv)).
	In particular, $G(w_x,m;\varepsilon)$ (which does not depend on $x$) is estimated in $\s{C}^{\alpha,\alpha/2}$, and hence via classical results, Equation \eqref{def:L-S_operator}(ii) has a unique solution $\mu$, which is estimated in $\s{C}^{2+\alpha,1+\alpha/2}$.
	In addition, as a consequence of Proposition~\ref{prop:LtInftyLx1_moment_mu} {\it{(a)}}, with 
	\begin{equation}
		b:=\lambda F^{(0)} \,\,\text{ and }\,\, \nu:=\lambda\Phi+\frac{1}{2}\lambda\big(G(w_x,m;\varepsilon)-w_x\big)m^{(0)},
	\end{equation}
	we find that $\mu\in L_t^\infty(L_x^1)$.
	Using Equation \eqref{eq:holder t L1x1}, we can deduce that $\mu \in \s{C}^{\alpha/2}\del{[0,T];(\s{C}^{1+2\alpha})^*}$, as well.
	This ultimately shows that $\mu\in\s{X}$, and hence $\s{T}$ is well-defined.
	\\\\
	2. \textit{Continuity.}
	To demonstrate continuity of the operator $\s{T}$, start by letting $\{(m_j,\lambda_j)\}_{j=1}^\infty\subseteq\s{X}\times[0,1]$ be a sequence such that $m_j\to m$ in $\s{X}$ and $\lambda_j\to\lambda$ as $j\to\infty$.
	Let $(w_j,\mu_j)$ be the solution to system \eqref{def:L-S_operator} with $m$ replaced by $m_j$ and $\lambda$ replaced by $\lambda_j$, so that $\mu_j:=\s{T}(m_j,\lambda_j)$.
	Via classical estimates (see also Lemma \ref{lem:existence Holder}), the sequences $\{w_j\}_{j=1}^\infty$ and $\{\mu_j\}_{j=1}^\infty$ are uniformly bounded in $\s{C}^{2+\alpha,1+\alpha/2}$, and therefore both $w_j,\mu_j$ and their derivatives are uniformly bounded and equicontinuous.
	By the Arzel\`{a}-Ascoli Theorem and diagonalization, there exists some $(w,\mu)\in\s{C}^{2+\alpha,1+\alpha/2}\times\s{C}^{2+\alpha,1+\alpha/2}$ and a subsequence $\{(w_{j_k},\mu_{j_k})\}_{k=1}^\infty$ such that
	\begin{equation}\label{ConvergenceInC_2_1}
		(w_{j_k},\mu_{j_k})\to(w,\mu)\text{ in }\s{C}^{2,1}\del{[0,R] \times [0,T]}\times\s{C}^{2,1}\del{[0,R] \times [0,T]}\text{ as }k\to\infty, \quad \forall R > 0.
	\end{equation}
	In particular, $w_{j_k},\mu_{j_k}$ and their derivatives converge pointwise and are uniformly bounded.
	Also, because we also have $m_j\to m$ in $L_t^\infty(L_x^1)$, we see that
	\begin{equation}
		G\del{(w_{j_k})_x,m_{j_k};\epsilon}(t) \to G\del{w_x,m;\epsilon}(t) \quad \forall t \in [0,T].
	\end{equation}
	Now let $k \to \infty$ in the system solved by $(w_{j_k},\mu_{j_k})$ to see that $(w,\mu)$ solves System \eqref{def:L-S_operator}, i.e.~$\mu = \s{T}(m,\lambda)$.
	Finally, by the argument that appears in the next step, we deduce that $\mu_{j_k} \to \mu$ in $\s{X}$.
	Hence $\s{T}(m_{j_k},\lambda_{j_k}) \to \s{T}(m,\lambda)$ in $\s{X}$.
	It follows that $\s{T}$ is continuous.
	\\\\
	3. \textit{Compactness.}
	Finally, to demonstrate compactness of the operator $\s{T}$, let $\{(m_j,\lambda_j)\}_{j=1}^\infty$ be any bounded sequence in $\s{X}\times[0,1]$.
	If again we let $(w_j,\mu_j)$ be the solution to system \eqref{def:L-S_operator} with $m$ replaced by $m_j$ and $\lambda$ replaced by $\lambda_j$,
	then as in the previous step we see that the sequences $\{w_j\}_{j=1}^\infty$ and $\{\mu_j\}_{j=1}^\infty$ are uniformly bounded in $\s{C}^{2+\alpha,1+\alpha/2}$,
	and there exists a subsequence $\{(w_{j_k},\mu_{j_k})\}_{k=1}^\infty$ such that \eqref{ConvergenceInC_2_1} holds.	
	It remains to show that $\mu_{j_k}\to \mu$ in $\s{X} = \s{C}^{\alpha,\alpha/2} \cap L_t^\infty(L_x^1)$.
		
	We will appeal to Proposition~\ref{prop:LtInftyLx1_moment_mu}.
	First, note that $\mu_{j_k}-\mu$ ($k=1,2,\dots$) satisfies \eqref{F-P_equation_mu} of Proposition~\ref{prop:LtInftyLx1_moment_mu} with
	\begin{equation}
		b_k :=(\lambda_{j_k}-\lambda)F^{(0)}
	\end{equation}
	and
	\begin{equation}
		\nu_k :=(\lambda_{j_k}-\lambda)\Phi+\frac{1}{2}\Big(\lambda_{j_k}G\big((w_{j_k})_x,m_{j_k};\varepsilon\big)-\lambda G(w_x,m;\varepsilon)-\big(\lambda_{j_k}(w_{j_k})_x-\lambda w\big)\Big)m^{(0)}.
	\end{equation}
	We have that $b_k$ is uniformly bounded in $L^\infty$; we need to show that $\nu_k$ and $x\nu_k$ are uniformly bounded in $L_t^\infty(L_x^1)$.
	By assumption, $\Phi\in L_t^\infty(L_x^1)$ and $x\Phi\in L_t^\infty(L_x^1)$.
	Next, by \eqref{ConvergenceInC_2_1}, there exists a constant $C_1\in(0,\infty)$ such that
	\begin{equation}
		\enVert[1]{\lambda_{j_k}w_{j_k}-\lambda u}_{\infty}+\enVert[1]{(\lambda_{j_k}w_{j_k})_x-\lambda w_x}_{\infty}\leq C_1\,\,\text{ for all }\,\, k\in\bb{N}.
	\end{equation}
	Also, by assumption, $\{m_{j_k}\}_{k=1}^\infty$ is uniformly bounded in $L_t^\infty(L_x^1)$, so there exists some constant $C_2\in(0,\infty)$ such that
	\begin{equation}
		\enVert[1]{\lambda_{j_k}m_{j_k}-\lambda m}_{L_t^\infty(L_x^1)}\leq C_2\,\,\text{ for all }\,\, k\in\bb{N}.
	\end{equation}
	It follows that
	\begin{align}
		\Big|\lambda_{j_k}G\big((w_{j_k})_x,&m_{j_k};\varepsilon\big)-\lambda G(w_x,m;\varepsilon)-\big(\lambda_{j_k}(w_{j_k})_x-\lambda w\big)\Big|
		\\[4pt]
		&=\envert{\beta(\varepsilon\xi)}\envert{\int_0^\infty\big(\lambda_{j_k}(w_{j_k})_x-\lambda w_x\big)m^{(0)}\dif x+\int_0^\infty u_x^{(0)}\big(\lambda_{j_k}m_{j_k}-\lambda m\big)\dif x}\nonumber
		\\[4pt]
		&\leq C_1+\enVert[1]{u_x^{(0)}}_{\infty}C_2\,\,\text{ for all }\,\,k\in\bb{N}.\nonumber
	\end{align}
	Note that to obtain the inequality we also used that $\int_0^\infty m^{(0)}\dif x\leq 1$.
	We now find that
	\begin{equation}
		\enVert{\nu_k}_{L_t^\infty(L_x^1)}\leq 2\enVert{\Phi}_{L_t^\infty(L_x^1)}+\frac{1}{2}\Big(C_1+\enVert[1]{u_x^{(0)}}_{\infty}C_2\Big)\sup_{t\in[0,T]}\int_0^\infty m^{(0)}(x,t)\dif x<M_1
	\end{equation}
	for some constant $M_1\in(0,\infty)$, independent of $k$.
	Also, Proposition~\ref{prop:LtInftyLx1_moment_mu} {\it{(b)}} ensures that $xm^{(0)}\in L_t^\infty(L_x^1)$, and thus
	\begin{equation}
		\enVert{x\nu_k}_{L_t^\infty(L_x^1)}\leq 2\enVert{x\Phi}_{L_t^\infty(L_x^1)}+\frac{1}{2}\Big(C_1+\enVert[1]{u_x^{(0)}}_{\infty}C_2\Big)\sup_{t\in[0,T]}\int_0^\infty xm^{(0)}(x,t)\dif x<M_2
	\end{equation}
	for some constant $M_2\in(0,\infty)$, independent of $k$.
	Therefore, the PDE that $\lambda_{j_k}\mu_{j_k}-\lambda\mu$ satisfies, namely \eqref{F-P_equation_mu}, also satisfies the hypotheses of Proposition~\ref{prop:LtInftyLx1_moment_mu} uniformly for all $k\in\bb{N}$.
	
	We then write
	\begin{align}
		\sup_{t\in[0,T]}\int_0^\infty\big|\mu_{j_k}(x,t)&-\mu(x,t)\big|\dif x
		\\[4pt]
		&=\sup_{t\in[0,T]}\Bigg[\int_0^K\envert{\mu_{j_k}(x,t)-\mu(x,t)}\dif x+\int_K^\infty\envert{\mu_{j_k}(x,t)-\mu(x,t)}\dif x\Bigg],\nonumber
	\end{align}
	By \eqref{ConvergenceInC_2_1}, for each fixed $K>0$,
	\begin{equation}
		\sup_{t\in[0,T]}\int_0^K\envert{\mu_{j_k}(x,t)-\mu(x,t)}\dif x\to0\,\,\text{ as }\,\, k\to\infty.
	\end{equation}
	By Proposition~\ref{prop:LtInftyLx1_moment_mu} {\it{(b)}},
	\begin{align}
		\sup_{t\in[0,T]}\int_K^\infty\envert{\mu_{j_k}(x,t)-\mu(x,t)}\dif x&\leq\sup_{t\in[0,T]}\frac{1}{K}\int_0^Kx\envert{\mu_{j_k}(x,t)-\mu(x,t)}\dif x
		\\[4pt]
		&\leq K^{-1}M_{\sigma,b}\big(\enVert{\nu_k}_{L_t^\infty(L_x^1)}+\enVert{x\nu_k}_{L_t^\infty(L_x^1)}\big)\\
		&\leq K^{-1}M_{\sigma,b}(M_1 + M_2)\to0\,\,\text{ as }\,\, K\to\infty\nonumber
	\end{align}
	for all $k\in\bb{N}$, and we see that $\mu_{j_k}\to \mu$ in $L_t^\infty(L_x^1)$.
	Similarly,  let $\bar \mu_{j_k} = \mu_{j_k} - \mu$.
	For all $t_1 \neq t_2$,  by taking $K = \abs{t_1-t_2}^{-1/2}$ we have, cf.~\eqref{eq:holder t L1x1},
	\begin{multline}
			\int_0^\infty \abs{\bar \mu_{j_k}(x,t_1) - \bar \mu_{j_k}(x,t_2)}\dif x\\
		\leq \enVert{\mu_{j_k} - \mu}_{\s{C}^{2,1}\del{[0,K] \times [0,T]}}\abs{t_1-t_2}^{1/2}
			+ 2\abs{t_1-t_2}^{1/2}\sup_t \int_0^\infty \abs{\mu_{j_k}(x,t) - \mu(x,t)}\dif x.
	\end{multline}
	From this we deduce that $\mu_{j_k} \to \mu$ in $\s{C}^{1/2}([0,T];L^1)$, hence also in $\s{C}^{\alpha/2}\del{[0,T];(\s{C}^{1+2\alpha})^*}$.
	
	Moreover, by \cite[Lemma 2.2]{cesaroni2018concentration}, we have
	\begin{equation}
		\sup_{t \in [0,T]}\sup_{x \geq K} \envert{\mu_{j_k}(x,t)-\mu(x,t)}
		\leq C\del{K^{-1}M_{\sigma,b}(M_1 + M_2)}^{\frac{\alpha}{\alpha+1}},
	\end{equation}
	where $C$ depends only on $\sup_k \enVert{\mu_{j_k} - \mu}_{\s{C}^{\alpha,\alpha/2}}$.
	Combining this estimate with \eqref{ConvergenceInC_2_1}, we deduce that $\mu_{j_k} \to \mu$ uniformly.
	Since $\{\mu_{j_k}\}$ is bounded in $\s{C}^{2+\alpha,1+\alpha/2}$, we deduce that $\mu_{j_k} \to \mu$ in $\s{C}^{\alpha,\alpha/2}$.
	Ultimately, we have that
	\begin{equation}
		\mu_{j_k}\to \mu\,\,\text{ in }\,\, \s{X} \,\,\text{ as }\,\, k\to\infty,
	\end{equation}
	as desired.
	We have thus shown that the operator $\s{T}$ is continuous and compact.
	 This completes the proof.
\end{proof}

With Proposition ~\ref{lem:L-S_hypotheses} established, we may now appeal to the Leray-Schauder fixed point theorem to obtain an existence result for System \eqref{abstract_system}.

\begin{theorem}\label{thm:existence_result_w/mu}
	There exists a unique classical solution $(w,\mu)$ to System \eqref{abstract_system} provided $\varepsilon_0 > 0$ satisfies \eqref{Assumption:epsilon_small} and $\varepsilon\in[0,\varepsilon_0]$.
\end{theorem}

\begin{proof}
	We endeavor to use the Leray-Schauder fixed point theorem, i.e., Theorem \ref{thm:Leray-Schauder}.
	For that purpose, a quick inspection reveals that when $\lambda=0$, we have $\s{T}(\mu,0)=0$ for all $m\in \s{X}$.
	Furthermore, Lemma~\ref{lem:L-S_hypotheses} shows that $\s{T}$ satisfies the compactness requirement of the Leray-Schauder fixed point theorem.
	Now, suppose that there exists $\mu \in X$ such that $\mu=\s{T}(\mu,\lambda)$, so that \eqref{def:L-S_operator} has a solution $(w,\mu)$ with $\mu = m$.
	Then via the a priori estimates established in Section~\ref{sec:a_priori}, specifically the regularity results of Theorem~\ref{thm:HolderRegularity_w/mu} and Proposition~\ref{prop:LtInftyLx1_moment_mu} {\it{(a)}}, we find that indeed,
	\begin{equation}\label{a_priori_bound_LS}
		\enVert{\mu}_{\s{X}}\leq C
	\end{equation}
	for some constant $C\in(0,\infty)$. 
	We conclude that there exists a fixed point of the operator $\s{T}$, corresponding to $\lambda=1$.
	We deduce that there exists a solution to \eqref{abstract_system}, as wanted.
	
	To show uniqueness, let $(w_1,\mu_1)$ and $(w_2,\mu_2)$ be solutions to \eqref{abstract_system}. Consider the difference $(w,\mu):=(w_1-w_2, \mu_1-\mu_2)$. Then because System \eqref{abstract_system} is linear, we find that $(w,\mu)$ satisfies System \eqref{abstract_system} with $\Psi$ and $\Phi$ identically 0. As such, Theorem \ref{thm:HolderRegularity_w/mu} implies
	\begin{equation}
		\enVert{w}_{\s{C}^{2+\alpha,1+\alpha/2}}=\enVert{\mu}_{\s{C}^{2+\alpha,1+\alpha/2}}\leq 0,
	\end{equation}
	and subsequently $w=\mu=0$. Therefore, $w_1=w_2$ and $\mu_1=\mu_2$, hence the solution to System \eqref{abstract_system} is unique. This completes the proof.
\end{proof}

\section{Proof of Theorem \ref{thm:ExistenceUniquenessEtc_un/mn}}\label{sec:Induction}

We now present the proof of Theorem~\ref{thm:ExistenceUniquenessEtc_un/mn}.

\begin{proof}[Proof of Theorem \ref{thm:ExistenceUniquenessEtc_un/mn}]
	We induct on $k$ and break the proof into three separate inductive arguments.\\\\
	\textit{Step 1.}  Let $n\in\bb{N}$ and suppose the following induction hypothesis.
		\begin{equation*}
			\parbox{12cm}{For all $k\in\{0,1,\dots,n-1\}$ assume there exists  a solution $(u^{(k)}, m^{(k)})$ to \eqref{un-mnsystem}, as well as constants $C_k, C_k^\ast\in(0,\infty)$, depending solely on the data and on $(u^{(j)}, m^{(j)})$ for $j < k$, such that $u^{(k)}$ and $m^{(k)}$ satisfy}
		\end{equation*}
		\begin{equation}\label{InductionHypothesis}
			\enVert[1]{u^{(k)}}_{\s{C}^{2+\alpha,1+\alpha/2}}+\enVert[1]{m^{(k)}}_{\s{C}^{2+\alpha,1+\alpha/2}}\leq C_{k} \,\,\text{ and }\,\, \enVert[1]{m^{(k)}}_{\s{C}^{1/2}\del{[0,T];L^1}} + \enVert{xm^{(k)}}_{L^\infty_t(L_x^1)} \leq C_{k}^\ast.
		\end{equation}
		As a consequence of the inductive hypothesis, for all $k\in\bb{N}$ with $0\leq k<n$, we can estimate
		\begin{equation}\label{FkHolderEstimate}
			\enVert[1]{F^{(k)}}_{\s{C}^{1+\alpha,\alpha/2}}\leq C_{F_k}<\infty,
		\end{equation}
		where we define the constant
		\begin{equation}\label{def:C_Fk}
			C_{F_k}:=k!\Big(1+\sum_{i=0}^k\sum_{j=0}^{k-i}C_{j}C_{k-i-j}^\ast\Big)+C_{k}.
		\end{equation}
		In order to apply our abstract results from Section \ref{sec:L-S_Theorem}, we introduce the following notation.
		We will separate out the $n$-th order terms appearing in $J_n$ and $K_n$ by defining $\widetilde{J}_n$ and $\widetilde{K}_n$ implicitly as
		\begin{equation}
			J_n=\widetilde{J}_n+F^{(0)}\Bigg(\beta(\varepsilon\xi)\int_0^\infty \Big(u_x^{(0)}m^{(n)}+u_x^{(n)}m^{(0)}\Big)\dif x-u_x^{(n)}\Bigg)
		\end{equation}
		and
		\begin{equation}
			K_n=\widetilde{K}_n+\frac{1}{2}\Bigg(\beta(\varepsilon\xi)\int_0^\infty \Big(u_x^{(0)}m^{(n)}+u_x^{(n)}m^{(0)}\Big)\dif x-u_x^{(n)}\Bigg)m^{(0)}.
		\end{equation}
		\begin{comment}
		Similarly, we define (with some abuse of notation) $\widetilde{F}^{(n)}$ by
		\begin{equation}
			F^{(n)} =: \widetilde{F}^{(n)}+\beta(\varepsilon\xi)\int_0^\infty \Big(u_x^{(0)}m^{(n)}+u_x^{(n)}m^{(0)}\Big)\dif x-u_x^{(n)}.
		\end{equation}
		\end{comment}
		
		While the exact formulas of $\widetilde{J}_n$ and $\widetilde{K}_n$
		 can be easily gleaned from \eqref{def:Jk}, \eqref{def:Kk}, and \eqref{def:Fk}, it is more important to know that $\widetilde{J}_n$ and $\widetilde{K}_n$ consist entirely of lower order terms, and hence ``known" quantities in light of the induction hypothesis (see \eqref{InductionHypothesis} in Section \ref{sec:Induction}).
		
		\begin{comment}
		We now estimate
		\begin{equation}\label{FnTildeHolderEstimate}
			\enVert[1]{\widetilde{F}^{(n)}}_{\s{C}^{\alpha,\alpha/2}}\leq n!\Big(1+\sum_{i=1}^n\sum_{j=0}^{n}C_{j}C_{n-i-j}^\ast+\sum_{j=1}^{n-1}C_{j}C_{n-j}^\ast\Big)<\infty.
		\end{equation}
		\end{comment}
		Using \eqref{FkHolderEstimate} and recalling the formulas for $J_n$ and $K_n$ in \eqref{def:Jk} and \eqref{def:Kk}, we may estimate
		\begin{equation}\label{JnTildeHolderEstimate}
			\enVert[1]{\widetilde{J}_n}_{\s{C}^{\alpha,\alpha/2}}\leq\sum_{k=1}^{n-1}\binom{n}{k}C_kC_{n-k}+2n!\Big(1+\sum_{i=1}^n\sum_{j=0}^{n}C_{j}C_{n-i-j}^\ast+\sum_{j=1}^{n-1}C_{j}C_{n-j}^\ast\Big)<\infty,
		\end{equation}
		and
		\begin{equation}\label{KnTildeL1Estimate}
			\int_0^\infty(1+x)\big|\widetilde{K}_n(x,t;\varepsilon)\big|\dif x\leq\sum_{k=1}^{n-1}\binom{n}{k} C_{F_k}C_{n-k}^\ast+n!\Big(1+\sum_{i=1}^n\sum_{j=0}^{n}C_{j}C_{n-i-j}^\ast+\sum_{j=1}^{n-1}C_{j}C_{n-j}^\ast\Big)<\infty.
		\end{equation}
		Also, using the product rule, we find that
		\begin{align}
			\big(\widetilde{K}_{n}\big)_x&=\sum_{k=1}^{n-1}\binom{n}{k}\Big(F_x^{(k)}m^{(n-k)}+F^{(k)}m_x^{(n-k)}\Big)+\widetilde{F}_x^{(n)}m^{(0)}+\widetilde{F}^{(n)}m_x^{(0)}
			\\[4pt]
			&=\sum_{k=1}^{n-1}\binom{n}{k}\Big(-u_{xx}^{(k)}m^{(n-k)}+F^{(k)}m_x^{(n-k)}\Big)+\widetilde{F}^{(n)}m_x^{(0)},\nonumber
		\end{align}
		so that
		\begin{equation}\label{KnTildeHolderEstimate}
			\enVert[1]{(\widetilde{K}_n)_x}_{\s{C}^{\alpha,\alpha/2}}\leq\sum_{k=1}^{n-1}\binom{n}{k}C_k^2+C_{F_k}C_k+n!\Big(1+\sum_{i=1}^n\sum_{j=0}^{n}C_{j}C_{n-i-j}^\ast+\sum_{j=1}^{n-1}C_{j}C_{n-j}^\ast\Big)C_0<\infty.
		\end{equation}
		Now, the base case, $k=0$, is established in \cite{graber2021nonlocal} (cf.~\cite{graber2018existence}).
		(Note that $\enVert[1]{m^{(0)}}_{\s{C}^{1/2}([0,T];L^1)}\leq C_{0}^\ast$ can be established, even on an unbounded domain, by the same argument as in the proof of Theorem \ref{thm:HolderRegularity_w/mu}, cf.~Equation \eqref{eq:holder t L1x1}.)
		Assume then that the induction hypothesis holds for $k = n-1$.
		Observe that System \eqref{un-mnsystem} can be written as System \eqref{abstract_system} with $w:=u^{(n)}$, $\mu:=m^{(n)}$, $\Psi:=\widetilde{J}_n$, and $\Phi:=\widetilde{K}_n$. Moreover, \eqref{JnTildeHolderEstimate}, \eqref{KnTildeL1Estimate}, and \eqref{KnTildeHolderEstimate} show that the induction hypothesis guarantees $\Psi\in \s{C}^{\alpha,\alpha/2}$ and $\Phi\in\s{C}^{1+\alpha,\alpha/2}\cap L_t^\infty(L_x^1)$. Therefore, by Theorem \ref{thm:existence_result_w/mu}, there exists a unique solution to System \eqref{un-mnsystem}, which by Theorem \ref{thm:HolderRegularity_w/mu} must satisfy \eqref{InductionHypothesis} for $k=n$.
		
		\begin{remark}
			\label{rem:constants}
			The constants $C_n$ derived from this inductive argument are likely to blow up very quickly.
			Assuming they are greater than 1, they will be an increasing sequence, and an inspection of the argument above reveals that
			\begin{equation}
				C_n \gtrsim 2^n C_{n-1}^2. 
			\end{equation}
			This in turn implies that the ratio $C_n/C_{n-1}$ converges to $+\infty$, which means $\sum C_n x^n$ does not converge on any radius.
		\end{remark}
		\textit{Step 2.} Next, let $n\in\bb{N}$ and suppose the following induction hypothesis.
		\begin{equation*}
			\parbox{12cm}{Assume the induction hypothesis of \textit{Step 1} holds. Aditionally, for all $k\in\{0,1,\dots,n-1\}$, suppose there exist constants $C_k',(C_k^\ast)'\in(0,\infty)$ depending solely on lower order terms and the data such that}
		\end{equation*}
		\begin{equation}\label{InductionHypothesis2}
			\begin{array}{c}
				\enVert[1]{u^{(k)}(\varepsilon+h)-u^{(k)}(\varepsilon)}_{\s{C}^{2+\alpha,1+\alpha/2}}\leq C_k'h,
				\vspace{12pt}\\
				\enVert[1]{m^{(k)}(\varepsilon+h)-m^{(k)}(\varepsilon)}_{\s{C}^{2+\alpha,1+\alpha/2}}\leq C_k'h,
				\vspace{12pt}\\
				\enVert[1]{m^{(k)}(\varepsilon+h)-m^{(k)}(\varepsilon)}_{\s{C}^{1/2}([0,T];L^1)} + \enVert[1]{x\del{m^{(k)}(\varepsilon+h)-m^{(k)}(\varepsilon)}}_{L^\infty_t(L_x^1)}\leq (C_k^*)'h.
			\end{array}
		\end{equation}
		For $k=0,1,\dots, n$ define
		\begin{equation}
			w^{(k)}:=u^{(k)}(\varepsilon+h)-u^{(k)}(\varepsilon)\,\,\text{ and }\,\,\mu^{(k)}:=m^{(k)}(\varepsilon+h)-m^{(k)}(\varepsilon).
		\end{equation}
		In the case $k=n$, it can be shown that $w^{(n)}$ and $\mu^{(n)}$ satisfy System \eqref{abstract_system} with $w:=w^{(n)}$, $\mu:=\mu^{(n)}$,
		\begin{equation}
			\Psi:=\Psi_1+2F^{(0)}(\varepsilon)I_1(\varepsilon), \,\,\text{ and }\,\, \Phi:=\Phi_1+I_1(\varepsilon)m^{(0)}(\varepsilon),
		\end{equation}
		where
		\begin{align}
			I_1(\varepsilon)&:=\frac{1}{2}\big(\xi^n\alpha^{(n)}\big((\varepsilon+h)\xi\big)-\xi^n\alpha^{(n)}(\varepsilon\xi)\big)
			\\[4pt]
			&\quad+\frac{1}{2}\int_0^\infty\sum_{i=1}^n\sum_{j=0}^{n-i}\binom{n}{i}\binom{n-i}{j}\Bigg\{\xi^i\beta^{(i)}(\varepsilon\xi)\Big(u_x^{(j)}(\varepsilon+h)\mu^{(n-i-j)}+w_x^{(j)}m^{(n-i-j)}(\varepsilon)\Big)\nonumber
			\\[4pt]
			&\quad+\xi^i\Big(\beta^{(i)}\big((\varepsilon+h)\xi\big)-\beta^{(i)}(\varepsilon\xi)\Big)u_x^{(j)}(\varepsilon+h)m^{(n-i-j)}(\varepsilon+h)\Bigg\}\dif x\nonumber
			\\[4pt]
			&\quad+\frac{1}{2}\int_0^\infty\sum_{j=1}^{n-1}\binom{n}{j}\beta(\varepsilon\xi)\Big(u_x^{(j)}(\varepsilon+h)\mu^{(n-j)}+w_x^{(j)}m^{(n-j)}(\varepsilon)\Big)\dif x\nonumber
			\\[4pt]
			&\quad+\frac{1}{2}\int_0^\infty\beta(\varepsilon\xi)\Big(w_x^{(0)}m^{(n)}(\varepsilon)+u_x^{(n)}(\varepsilon+h)\mu^{(0)}\Big)+\Big(u_x^{(0)}(\varepsilon+h)-u_x^{(0)}(\varepsilon)\Big)\mu^{(n)}\dif x,\nonumber
		\end{align}
		\begin{align}
			\Psi_1&:=\sum_{k=1}^n\binom{n}{k}\Big(F^{(k)}(\varepsilon+h)+F^{(k)}(\varepsilon)\Big)\Big(F^{(n-k)}(\varepsilon+h)-F^{(n-k)}(\varepsilon)\Big)
			\\[4pt]
			&\quad+\Big(F^{(0)}(\varepsilon+h)-F^{(0)}(\varepsilon)\Big)\Big(F^{(n)}(\varepsilon+h)-F^{(n)}(\varepsilon)\Big),\nonumber
		\end{align}
		and
		\begin{align}
			\Phi_1&:=\sum_{k=1}^{n-1}\binom{n}{k}\Bigg\{F^{(k)}(\varepsilon+h)\mu^{(n-k)}+\big(F^{(k)}(\varepsilon+h)-F^{(k)}(\varepsilon)\big)m^{(n-k)}(\varepsilon)\Bigg\}
			\\[4pt]
			&\quad+\Big(F^{(0)}(\varepsilon+h)-F^{(0)}(\varepsilon)\Big)m^{(n)}(\varepsilon)+\Big(F^{(0)}(\varepsilon+h)-F^{(0)}(\varepsilon)\Big)\mu^{(n)}+F^{(n)}(\varepsilon+h)\mu^{(0)}.\nonumber
		\end{align}
		A quick inspection reveals that the induction hypothesis implies $\enVert{\Psi}_{\s{C}^{\alpha,\alpha/2}}$ and $\enVert{\Phi}_{\s{C}^{1+\alpha,\alpha/2}}$ are $O(h)$. As such, the estimates in Theorem~\ref{thm:HolderRegularity_w/mu} show that $u^{(n)}$ and $m^{(n)}$ are $O(h)$ as well.\\\\
		\textit{Step 3.} Let $n\in\bb{N}$ and suppose the following induction hypothesis.
		\begin{equation*}
			\parbox{12cm}{Assume the induction hypothesis of \textit{Step 2} holds. Additionally, suppose there exist constants $C_k'', (C_k^\ast)''\in(0,\infty)$ depending solely on lower order terms and the data such that}
		\end{equation*}
		\begin{equation}
			\begin{array}{c}
				\enVert[1]{u^{(k)}(\varepsilon+h)-u^{(k)}(\varepsilon)-hu^{(k+1)}(\varepsilon)}_{\s{C}^{2+\alpha,1+\alpha/2}}\leq C_k''h^2,
				\vspace{12pt}\\
				\enVert[1]{m^{(k)}(\varepsilon+h)-m^{(k)}(\varepsilon)-hm^{(k+1)}(\varepsilon)}_{\s{C}^{2+\alpha,1+\alpha/2}}\leq C_k''h^2,
				\vspace{12pt}\\
				\enVert[1]{m^{(k)}(\varepsilon+h)-m^{(k)}(\varepsilon)-hm^{(k+1)}(\varepsilon)}_{\s{C}^{1/2}([0,T];L^1)}\leq (C_k^\ast)''h^2,\vspace{12pt}\\
				\enVert[1]{x\del{m^{(k)}(\varepsilon+h)-m^{(k)}(\varepsilon)-hm^{(k+1)}(\varepsilon)}}_{L^\infty_t(L_x^1)}\leq (C_k^\ast)''h^2
			\end{array}
		\end{equation}
		For $k=0,1,\dots, n$ define
		\begin{equation}
			w^{(k)}:=u^{(k)}(\varepsilon+h)-u^{(k)}(\varepsilon)-hu^{(k+1)}(\varepsilon)\,\,\text{ and }\,\,\mu^{(k)}:=\mu^{(k)}(\varepsilon+h)-\mu^{(k)}(\varepsilon)-h\mu^{(k+1)}(\varepsilon).
		\end{equation}
		Note that $u^{(n+1)}$ and $m^{(n+1)}$ are well-defined objects via \textit{Step 1}. As it happens, when $k=n$ the pair $(w^{(n)}, \mu^{(n)})$ satisfies System \eqref{abstract_system} with $w:=w^{(n)}$, $\mu:=\mu^{(n)}$,
		\begin{equation}
			\Psi:=\Psi_2+2F^{(0)}(\varepsilon)I_2(\varepsilon), \,\,\text{ and }\,\, 	\Phi:=\Phi_2+I_2(\varepsilon)m^{(0)}(\varepsilon),
		\end{equation}
		where
		\begin{align}
			I_2(\varepsilon)&:=\frac{1}{2}\Big(\xi^n\alpha^{(n)}\big((\varepsilon+h)\xi\big)-\xi^n\alpha^{(n)}(\varepsilon\xi)-h\xi^{(n+1)}\alpha^{(n+1)}(\varepsilon\xi)\Big)
			\\[4pt]
			&\quad+\frac{1}{2}\int_0^\infty\sum_{i=1}^n\sum_{j=0}^{n-i}\binom{n}{i}\binom{n-i}{j}\Bigg\lbrace\xi^i\beta^{(i)}(\varepsilon\xi)\Big(u_x^{(j)}(\varepsilon)\mu^{n-i-j}+w_x^{(j)}m^{(n-i-j)}(\varepsilon)\Big)\nonumber
			\\[4pt]
			&\quad+\Big[\xi^i\beta^{(i)}\big((\varepsilon+h)\xi\big)-\xi^i\beta^{(i)}(\varepsilon\xi)-h\xi^{i+1}\beta^{(i+1)}(\varepsilon\xi)\Big]u_x^{(j)}(\varepsilon)m^{(n-i-j)}(\varepsilon)\nonumber
			\\[4pt]
			&\quad+\Big[\xi^i\beta^{(i)}\big((\varepsilon+h)\xi\big)u_x^{(j)}(\varepsilon+h)-\xi^i\beta^{(i)}(\varepsilon\xi)u_x^{(j)}(\varepsilon)\Big]\Big[m^{(n-i-j)}(\varepsilon+h)-m^{(n-i-j)}(\varepsilon)\Big]\nonumber
			\\[4pt]
			&\quad+\Big[\xi^i\beta^{(i)}\big((\varepsilon+h)\xi\big)-\xi^i\beta^{(i)}(\varepsilon\xi)\Big]\Big[u_x^{(j)}(\varepsilon+h)-u_x^{(j)}(\varepsilon)\Big]m^{(n-i-j)}(\varepsilon)\Bigg\rbrace\dif x\nonumber
			\\[4pt]
			&\quad+\frac{1}{2}\int_0^\infty\sum_{j=0}^{n}\binom{n}{j}\Bigg\lbrace\Big[\beta\big((\varepsilon+h)\xi\big)-\beta(\varepsilon\xi)-h\xi\beta^{(1)}(\varepsilon\xi)\Big]u_x^{(j)}(\varepsilon)m^{(n-j)}(\varepsilon)\nonumber
			\\[4pt]
			&\quad+\Big[\beta\big((\varepsilon+h)\xi\big)u_x^{(j)}(\varepsilon+h)-\beta(\varepsilon\xi)u_x^{(j)}(\varepsilon)\Big]\Big[m^{(n-j)}(\varepsilon+h)-m^{(n-j)}(\varepsilon)\Big]\nonumber
			\\[4pt]
			&\quad+\Big[\beta\big((\varepsilon+h)\xi\big)-\beta(\varepsilon\xi)\Big]\Big[u_x^{(j)}(\varepsilon+h)-u_x^{(j)}(\varepsilon)\Big]m^{(n-j)}(\varepsilon)\Bigg\rbrace\dif x\nonumber
			\\[4pt]
			&\quad+\frac{1}{2}\int_0^\infty\sum_{j=1}^{n-1}\binom{n}{j}\beta(\varepsilon\xi)\Big(u_x^{(j)}(\varepsilon)\mu^{n-j}+w_x^{(j)}m^{(n-j)}(\varepsilon)\Big)\dif x\nonumber
			\\[4pt]
			&\quad+\frac{1}{2}\int_0^\infty\beta(\varepsilon\xi)\Big(u_x^{(n)}\mu^{(0)}+w_x^{(0)}m^{(n)}\Big)\dif x\nonumber
		\end{align}
		\begin{align}
			\Psi_2&:=2\sum_{k=0}^{n-1}\binom{n}{k}\Big(F^{(k)}(\varepsilon+h)-F^{(k)}(\varepsilon)-hF^{(k+1)}(\varepsilon)\Big)F^{(n-k)}(\varepsilon)
			\\[4pt]
			&\quad+\sum_{k=0}^n\binom{n}{k}\Big(F^{(k)}(\varepsilon+h)-F^{(k)}(\varepsilon)\Big)\Big(F^{(n-k)}(\varepsilon+h)-F^{(n-k)}(\varepsilon)\Big),\nonumber
		\end{align}
		and
		\begin{align}
			\Phi_2&:=\sum_{k=0}^{n}\binom{n}{k}\Big(F^{(k)}(\varepsilon+h)-F^{(k)}(\varepsilon)-hF^{(k+1)}(\varepsilon)\Big)\Big(m^{(n-k)}(\varepsilon+h)-m^{(n-k)}(\varepsilon)\Big)
			\\[4pt]
			&\quad+h\sum_{k=0}^{n}\binom{n}{k}F^{(k+1)}(\varepsilon)\Big(m^{(n-k)}(\varepsilon+h)-m^{(n-k)}(\varepsilon)\Big)+\sum_{k=1}^n\binom{n}{k}F^{(k)}(\varepsilon)\mu^{(n-k)}\nonumber
			\\[4pt]
			&\quad+\sum_{k=0}^{n-1}\binom{n}{k}\Big(F^{(k)}(\varepsilon+h)-F^{(k)}(\varepsilon)-hF^{(k+1)}(\varepsilon)\Big)m^{(n-k)}(\varepsilon).\nonumber
		\end{align}
		As such, we find that the induction hypothesis together with \textit{Step 2} imply that $\enVert{\Psi}_{\s{C}^{\alpha,\alpha/2}}$ and $\enVert{\Phi}_{\s{C}^{1+\alpha,\alpha/2}}$ are $O(h^2)$. Using the estimates in Theorem \ref{thm:HolderRegularity_w/mu}, we ultimately see that $u^{(n)}$ and $m^{(n)}$ are $O(h^2)$. This yields the desired result concerning differentiability with respect to $\varepsilon$.
\end{proof}

\appendix

\section{Some technical proofs of results from Section \ref{sec:a_priori}} \label{ap:a priori}

\subsection{Proofs for Section \ref{sec:a priori F-P}} \label{ap:a priori F-P}

\begin{proof}[Proof of Proposition \ref{prop:LtInftyLx1_moment_mu}]
	Let $S(x,t)$ denote the heat kernel, given by
	\begin{equation}\label{heat_kernel}
		S(x,t) := (2\sigma^2 \pi t)^{-1/2}\exp\cbr{-\frac{x^2}{2\sigma^2 t}}.
	\end{equation}
	Using Duhamel's principle and integration by parts, we have $\mu = \mu_1 + \mu_2$, where
	\begin{equation}\label{Duhamel+IBP_mu}
		\begin{split}
			\mu_1(x,t)&=\int_0^t\int_0^\infty\big(S(x-y,t-s)-S(x+y,t-s)\big)\big((b\mu)_y(y,s)+\nu_y(y,s)\big)\dif y\dif s\\
			&=\int_0^t\int_0^\infty\big(S_x(x-y,t-s)+S_x(x+y,t-s)\big)\big((b\mu)(y,s)+\nu(y,s)\big)\dif y\dif s
		\end{split}
	\end{equation}
	and 
	\begin{equation}
		\mu_2(x,t) = \int_0^\infty \del{S(x-y,t) - S(x+y,t)}\mu_0(y)\dif y.
	\end{equation}
	Note that
	\begin{equation}\label{Integral_S_x}
		\int_0^\infty\big|S_x(x,t)\big|\dif x=-\int_0^\infty S_x(x,t)\dif x=-\lim_{x\to\infty}\big(S(x,t)-S(0,t)\big)=(2\pi\sigma^2 t)^{-1/2}.
	\end{equation}
	Call
	\begin{equation}\label{c'_constant}
		c':=2(2\pi\sigma^2)^{-1/2},
	\end{equation}
	and let $M>1$ to be determined later.
	Then
	\begin{equation}\label{L1Estimate_m_1}
		\begin{split}
			&e^{-Mt}\int_0^\infty|\mu_1(x,t)|\dif x
			\\[4pt]
			&\quad\leq e^{-Mt}\int_0^t\int_0^\infty\int_0^\infty\big(|S_x(x-y,t-s)|+|S_x(x+y,t-s)|\big)\big(\enVert{b}_{L^\infty}|\mu(y,s)|+|\nu(y,s)|\big)\dif x\dif y\dif s
			\\[4pt]
			&\quad\leq c'e^{-Mt}\int_0^t(t-s)^{-1/2}\int_0^\infty\big(\enVert{b}_{L^\infty}|\mu(y,s)|+|\nu(y,s)|\big)\dif y\dif s
			\\[4pt]
			&\quad=c'\int_0^te^{-M(t-s)}(t-s)^{-1/2}\int_0^\infty e^{-Ms}\big(\enVert{b}_{L^\infty}|\mu(y,s)|+|\nu(y,s)|\big)\dif y\dif s.
		\end{split}
	\end{equation}
	Here, the first inequality follows directly from \eqref{Duhamel+IBP_mu}; the second inequality follows from \eqref{Integral_S_x} and noting that $\exp\big\{-y^2/(2\sigma^2t)\big\}\leq 1$ for all $y\in(0,\infty)$.
	Define
	\begin{equation}\label{def:B}
		B:=\sup_{0\leq \tau\leq T}e^{-M\tau}\int_0^\infty|\mu(x,\tau)|\dif x.
	\end{equation}
	Then \eqref{L1Estimate_m_1} implies
	\begin{equation}\label{L1Estimate_m_2}
		\begin{split}
			e^{-Mt}\int_0^\infty|\mu_1(x,t)|\dif x&\leq c'B\enVert{b}_{L^\infty}\int_0^te^{-M(t-s)}(t-s)^{-1/2}\dif s+c'e^{-Mt}\enVert{\nu}_{L_t^\infty(L_x^1)}\int_0^t(t-s)^{-1/2}\dif s
			\\[4pt]
			&\leq c'B\enVert{b}_{L^\infty}M^{-1/2}\sqrt{\pi}+2c'e^{-Mt}\enVert{\nu}_{L_t^\infty(L_x^1)}\sqrt{t}.
		\end{split}
	\end{equation}
	The second inequality follows from the substitutions and estimate below
	\begin{equation}
		\int_0^te^{-M(t-s)}(t-s)^{-1/2}\dif s=\int_0^te^{-Ms}s^{-1/2}\dif s= M^{-1/2}\int_0^{Mt} e^{-s}s^{-1/2}\dif s\leq M^{-1/2}\sqrt{\pi}.
	\end{equation}
	Now, if $M>1$, then $2e^{-Mt}\sqrt{t}\leq 1$ for all $t>0$. Thus, \eqref{L1Estimate_m_2} implies
	\begin{equation} \label{eq:e-Mtabsmu1}
		e^{-Mt}\int_0^\infty|\mu_1(x,t)|\dif x\leq c'\enVert{b}_{L^\infty}\sqrt{\pi}M^{-1/2}B+c'\enVert{\nu}_{L_t^\infty(L_x^1)}.
	\end{equation}
	On the other hand, since $S(\cdot,t)$ is a probably density, we deduce
	\begin{equation} \label{eq:int mu2}
		\int_0^\infty \abs{\mu_2(x,t)}\dif x \leq \int_0^\infty \abs{\mu_0(x)}\dif x.
	\end{equation}
	Add $e^{-Mt}$ times \eqref{eq:int mu2} to \eqref{eq:e-Mtabsmu1}, take a supremum on both sides over all $t\in[0,T]$, and we find that
	\begin{equation}
		B\leq c'\enVert{b}_{L^\infty}\sqrt{\pi}M^{-1/2}B + \enVert{\mu_0}_{L^1} +c'\enVert{\nu}_{L_t^\infty(L_x^1)}.
	\end{equation}
	Choose $M>1$ large enough so that $M>(2c'\enVert{b}_{L^\infty}\sqrt{\pi})^2$.
	Consequently, we find that
	\begin{equation}
		B\leq 2\enVert{\mu_0}_{L^1} + 2c'\enVert{\nu}_{L_t^\infty(L_x^1)},
	\end{equation}
	which, given the definition of $B$ in \eqref{def:B}, establishes part {\textit{(a)}}.
	
	For part {\textit{(b)}} we argue similarly.
	First, note that
	\begin{equation}\label{Integral_xS_x}
		2\int_0^\infty x|S_x(x,t)|\dif x=-2\int_0^\infty xS_x(x,t)\dif x=2\int_0^\infty S(x,t)\dif t=1.
	\end{equation}
	Knowing this and using \eqref{Integral_S_x}, we can estimate
	\begin{align}\label{MomentEstimateInt_x}
		\begin{split}
			\int_0^\infty x\big(|S_x(x-y,t-s)|&+|S(x+y,t-s)|\big)\dif x
			\\[4pt]
			&=\int_{-y}^\infty(x+y)\envert{S_x(x,t-s)}\dif x+\int_y^\infty(x-y)\envert{S_x(x,t-s)}\dif x
			\\[4pt]
			&\leq\int_{-\infty}^\infty(x+y)\envert{S_x(x,t-s)}\dif x+\int_0^\infty x\envert{S_x(x,t-s)}\dif x
			\\[4pt]
			&=\frac{3}{2}+y\int_{-\infty}^\infty\envert{S_x(x,t-s)}\dif x
			\\[4pt]
			&\leq2+c'y(t-s)^{-1/2},
		\end{split}
	\end{align}
	where $c'$ is as in \eqref{c'_constant}.
	Note that the second equality above follows from \eqref{Integral_xS_x}, and the last inequality from \eqref{Integral_S_x}.
	We now use Fubini's Theorem to estimate
	\begin{equation}\label{MomentEstimate_1}
		\begin{split}
			\int_0^\infty &x|\mu_1(x,t)|\dif x
			\\[4pt]
			&\leq\int_0^t\int_0^\infty\int_0^\infty x\big(|S_x(x-y,t-s)|+|S_x(x+y,t-s)|\big)\big(\enVert{b}_{L^\infty}|\mu(y,s)|+|\nu(y,s)|\big)\dif x\dif y\dif s
			\\[4pt]
			&\leq\int_0^t\int_0^\infty\big(2+c'y(t-s)^{-1/2}\big)\big(\enVert{b}_{L^\infty}|\mu(y,s)|+|\nu(y,s)|\big)\dif y\dif s.
		\end{split}
	\end{equation}
	The second inequality follows directly from \eqref{MomentEstimateInt_x}.
	Let $M>1$ to be determined later and define
	\begin{equation}\label{def:B'}
		B':=\sup_{0\leq t\leq T}e^{-Mt}\int_0^\infty x|\mu(x,t)|\dif x.
	\end{equation}
	We then estimate in a similar manner as in the proof of part {\textit{(a)}}
	\begin{align}\label{MomentEstimate_2}
		\begin{split}
			e^{-Mt}\int_0^\infty x|\mu_2(x,t)|\dif x&\leq e^{-Mt}\int_0^t\int_0^\infty\big(2+c'y(t-s)^{-1/2}\big)\big(\enVert{b}_{L^\infty}|\mu(y,s)|+|\nu(y,s)|\big)\dif y\dif s
			\\[4pt]
			&\leq c'\int_0^te^{-M(t-s)}(t-s)^{-1/2}\int_0^\infty e^{-Ms}y\big(\enVert{b}_{L^\infty}|\mu(y,s)|+|\nu(y,s)|\big)\dif y\dif s
			\\[4pt]
			&\quad+2te^{-Mt}\big(\enVert{b}_{L^\infty}\enVert{\mu}_{L_t^\infty(L_x^1)}+\enVert{\nu}_{L_t^\infty(L_x^1)}\big).
		\end{split}
	\end{align}
	In similar fashion, we estimate
	\begin{equation} \label{MomentEstimate_3}
		\begin{split}
			\int_0^\infty x\abs{\mu_2(x,t)}\dif x
			&\leq \int_0^\infty \int_{-y}^\infty (x+y)S(x,t) \abs{\mu_0(y)}\dif x\dif y 
			+ \int_0^\infty \int_{y}^\infty (x-y)S(x,t) \abs{\mu_0(y)}\dif x\dif y \\
			&\leq 3\int_0^\infty xS(x,t)\dif x \int_0^\infty \abs{\mu_0(y)}\dif y
			+ \int_0^\infty y\abs{\mu_0(y)}\dif y,
		\end{split}
	\end{equation}
	where we have also used that $S(\cdot,t)$ is an even function.
	We calculate (by a change of variables) that $\int_0^\infty xS(x,t)\dif x = \sqrt{\frac{\sigma^2 t}{2\pi}}$.
	Add together \eqref{MomentEstimate_2} and \eqref{MomentEstimate_3} to get
	\begin{equation} \label{MomentEstimate_4}
		e^{-Mt}\int_0^\infty x|\mu(x,t)|\dif x \leq c'\enVert{b}_{L^\infty}\sqrt{\pi}M^{-1/2}B'+ \enVert{x\mu_0}_{L^1} + c'\enVert{x\nu}_{L_t^\infty(L_x^1)}+c_{b,\sigma}\del{\enVert{\mu_0}_{L^1} + \enVert{\nu}_{L_t^\infty(L_x^1)}}.
	\end{equation}
	Here, $c_{b,\sigma}<\infty$ is a constant that depends only on $\enVert{b}_{L^\infty}$ and $\sigma$; its existence is guaranteed by part {\textit{(a)}} of this proposition.
	Now, taking a supremum on both sides of \eqref{MomentEstimate_4} over all $t\in[0,T]$, then choosing $M>1$ large enough so that $M>(2\enVert{b}_{L^\infty}c'\sqrt{\pi})^2$, we derive
	\begin{equation}
		B'\leq 2\del{\enVert{x\mu_0}_{L^1} + c'\enVert{x\nu}_{L_t^\infty(L_x^1)}+c_{b,\sigma}\del{\enVert{\mu_0}_{L^1} + \enVert{\nu}_{L_t^\infty(L_x^1)}}}.
	\end{equation}
	Recalling the definition of $B'$ in \eqref{def:B'}, the above estimate implies part {\textit{(b)}} of the proposition.
\end{proof}

Before proving Proposition \ref{prop:L2timeL1space_mu}, we present an abstract lemma.
\begin{lemma}\label{lem:L2timeL1space:m}
	Let $A\geq1$, $B,\delta > 0$ be given constants.
	Suppose $f,g:\intco{0,\infty} \to \intco{0,\infty}$ are functions that satisfy
	\begin{equation} \label{eq:ineq}
		f(t_1) \leq Af(t_0) + B\int_{t_0}^{t_1}(t_1-s)^{-1/2}\del{f(s) + g(s)}\dif s \quad \forall 0 \leq t_0 \leq t_1 \leq t_0 + \delta
	\end{equation}
	Then for any $\lambda > \frac{1}{\delta}\ln(A)$, we have
	\begin{equation} \label{eq:estimate}
		\del{1 - \frac{2\delta^{1/2}B}{1 - Ae^{-\lambda \delta}}}\int_0^T e^{-\lambda t}f(t)\dif t \leq \frac{A}{\kappa}f(0) 
		+
		\frac{2\delta^{1/2}B}{1 - Ae^{-\lambda \delta}}\int_0^T e^{-\lambda t}g(t) \dif t.
	\end{equation}
\end{lemma}

\begin{proof}
	Set $h(t) = f(t) + g(t)$, so that \eqref{eq:ineq} reads simply
	\begin{equation} \label{eq:ineq1}
		f(t_1) \leq Af(t_0) + B\int_{t_0}^{t_1}(t_1-s)^{-1/2}h(s)\dif s \quad  \text{for all }0 \leq t_0 \leq t_1 \leq t_0 + \delta.
	\end{equation}
	For arbitrary $t > 0$ let $n = \floor{\frac{t}{\delta}}$.
	Use \eqref{eq:ineq1} $n+1$ times to get
	\begin{equation} \label{eq:ineq2}
		f(t) \leq A^{n+1}f(0) + \sum_{j=0}^n A^jB\int_{\del{t-(j+1)\delta}_+}^{t-j\delta}(t-j\delta-s)^{-1/2}h(s)\dif s,
	\end{equation}
	where $s_+ := \max\{s,0\}$.
	Note that
	\begin{equation*}
		t - (j+1)\delta < s \leq t-j\delta \Rightarrow
		j = \floor{\frac{t-s}{\delta}},
	\end{equation*}
	so we define $\phi(s) = \del{s - \floor{\frac{s}{\delta}}\delta}^{-1/2}$.		
	Then \eqref{eq:ineq2} implies
	\begin{align} \label{eq:ineq3}
		f(t) &\leq A^{\frac{t}{\delta}+1}f(0) + \sum_{j=0}^n B \int_{\del{t-(j+1)\delta}_+}^{t-j\delta}
		A^{\frac{t-s}{\delta}}\phi(t-s)h(s)\dif s
		\\[4pt]
		&= A^{\frac{t}{\delta}+1}f(0) + B\int_0^t A^{\frac{t-s}{\delta}}\phi(t-s)h(s)\dif s.\nonumber
	\end{align}
	Let $\lambda > \frac{1}{\delta}\ln(A)$ and set $\kappa = \lambda - \frac{1}{\delta}\ln(A) > 0$.
	Multiply \eqref{eq:ineq3} by $e^{-\lambda t}$, then integrate from 0 to $T$ to get
	\begin{equation} \label{eq:ineq4}
		\begin{split}
			\int_0^T e^{-\lambda t}f(t)\dif t &\leq \frac{A}{\kappa}f(0) 
			+ B\int_0^T\int_0^t e^{-\kappa(t-s)}\phi(t-s)e^{-\lambda s}h(s)\dif s \dif t\\
			&= \frac{A}{\kappa}f(0) 
			+ B\int_0^T\int_0^{T-s} e^{-\kappa t}\phi(t)e^{-\lambda s}h(s)\dif t \dif s.
		\end{split}
	\end{equation}
	We now observe that
	\begin{equation} \label{eq:int phi}
		\begin{split}
			\int_0^\infty e^{-\kappa t}\phi(t)\dif t
			&=
			\sum_{n=0}^\infty \int_{n\delta}^{(n+1)\delta} e^{-\kappa t}(t-n\delta)^{-1/2}\dif t\\
			&= \sum_{n=0}^\infty e^{-n\kappa \delta}\int_{0}^{\delta} e^{-\kappa t}t^{-1/2}\dif t\\
			&\leq \frac{1}{1 - e^{-\kappa \delta}}\int_0^\delta t^{-1/2}\dif t\\
			&= \frac{2\delta^{1/2}}{1 - e^{-\kappa \delta}}
			= \frac{2\delta^{1/2}}{1 - Ae^{-\lambda \delta}}
		\end{split}
	\end{equation}
	Applying \eqref{eq:int phi} to \eqref{eq:ineq4}, we get
	\begin{equation} \label{eq:ineq5}
		\int_0^T e^{-\lambda t}f(t)\dif t \leq \frac{A}{\kappa}f(0) 
		+
		\frac{2\delta^{1/2}B}{1 - Ae^{-\lambda \delta}}\int_0^T e^{-\lambda s}\del{f(s) + g(s)} \dif s,
	\end{equation}
	which implies \eqref{eq:estimate}.
\end{proof}

We now apply Proposition \ref{lem:L2timeL1space:m} to the Fokker-Planck equation.
\begin{proof}[Proof of Proposition \ref{prop:L2timeL1space_mu}]
	First we will treat $\mu$ as a solution of the abstract Fokker-Planck equation \eqref{F-P_equation_mu} by identifying $b:=F^{(0)}$ and $\nu:=\Phi+\frac{1}{2}\big(G(\varepsilon)-w_x\big)m^{(0)}$.
	We start with the following formula: for every $t_1 \geq t_0 \geq 0$,
	\begin{equation} \label{eq:mu duhamel}
		\begin{split}
			&\mu(x,t_1) = \int_0^\infty \del{S(x-y,t_1-t_0) - S(x+y,t_1-t_0)}\mu(y,t_0)\dif y
			\\[4pt]
			&\quad+ \int_{t_0}^{t_1}\int_0^\infty \del{\dpd{S}{x}(x-y,t_1-s) + \dpd{S}{x}(x+y,t_1-s)}\del{b(y,s)\mu(y,s) + \nu(y,s)}\dif y \dif s,
		\end{split}
	\end{equation}
	Using the same calculations as in \eqref{L1Estimate_m_1}, we get
	\begin{equation}
		\int_0^\infty \abs{\mu(x,t_1)}\dif x \leq \int_0^\infty \abs{\mu(y,t_0)}\dif y
		+ c'\int_{t_0}^{t_1}\int_0^\infty (t_1-s)^{-1/2}\del{\enVert{b}_\infty\abs{\mu(y,s)} + \abs{\nu(y,s)}}\dif y \dif s,
	\end{equation}
	where $c'$ is defined in \eqref{c'_constant}.
	Let
	\begin{equation*}
		f(t) := \del{\int_0^\infty \abs{\mu(x,t)}\dif x}^2
		\,\,\text{ and }\,\,
		g(t) := \del{\int_0^\infty \abs{\nu(x,t)}\dif x}^2,
	\end{equation*}
	so we can write
	\begin{equation} \label{eq:int |mu| estimate}
		f(t_1)^{1/2} \leq f(t_0)^{1/2}
		+ c\int_{t_0}^{t_1} (t_1-s)^{-1/2}\del{\enVert{b}_\infty f(s)^{1/2} + g(s)^{1/2}}\dif s.
	\end{equation}
	By H\"older's inequality, we get
	\begin{align}
		f(t_1)^{1/2} &\leq f(t_0)^{1/2}
		+ \enVert{b}_\infty c\del{2(t_1-t_0)^{1/2}}^{1/2}\del{\int_{t_0}^{t_1} (t_1-s)^{-1/2}f(s)\dif s}^{1/2}
		\\[4pt]
		&\quad+ c\del{2(t_1-t_0)^{1/2}}^{1/2}\del{\int_{t_0}^{t_1} (t_1-s)^{-1/2}g(s)\dif s}^{1/2}.\nonumber
	\end{align}
	Next, we square both sides and use the inequality $(a+b)^2 \leq 2(a^2 + b^2)$ to get
	\begin{equation}\label{f(t1)bound1}
		f(t_1) \leq 2f(t_0)
		+ 4c(1+\enVert{b}_{\infty})(t_1-t_0)^{1/2}\int_{t_0}^{t_1} (t_1-s)^{-1/2}\del{f(s)+ g(s)}\dif s.
	\end{equation}

	We now estimate
	\begin{equation}\label{Phi+(G-wx)m0 estimate}
		\int_0^\infty\envert[2]{\Phi+\frac{1}{2}\big(G(\varepsilon)-w_x\big)m^{(0)}}\dif x\leq \enVert[1]{\Phi}_{L_x^1}+2\int_0^\infty \envert[1]{w_xm^{(0)}}\dif x+\int_0^\infty|u_x^{(0)}\mu|\dif x,
	\end{equation}
	so that
	\begin{align}\label{g(t)bound}
		g(t)&\leq 3\enVert[1]{\Phi}_{L_x^1}^2+12\Bigg(\int_0^\infty |w_xm^{(0)}|\dif x\Bigg)^2+3\Bigg(\int_0^\infty|u_x^{(0)}\mu|\dif x\Bigg)^2
		\\[4pt]
		&\leq 3\enVert[1]{\Phi}_{L_x^1}^2+12\int_0^\infty w_x^2m^{(0)}\dif x+3\enVert[1]{u_x^{(0)}}_{\infty}^2f(t).\nonumber
	\end{align}
	The first inequality follows from the fact that $\big(\sum_{k=1}^nx_k\big)^2\leq n\sum_{k=1}^nx_k^2$.
	The second inequality follows from the Cauchy-Schwarz inequality and the fact that $\int_0^\infty m^{(0)}\dif x\leq 1$. 
	 Now, define
	\begin{equation}
		h(t):=\frac{1}{1+3\enVert[1]{u_x^{(0)}}_{\infty}^2}\Big(3\enVert[1]{\Phi}_{L_t^\infty(L_x^1)}+12\int_0^\infty w_x^2m^{(0)}\dif x\Big).
	\end{equation}
	Since $$1 + \enVert{b}_{\infty}=1 + \enVert[1]{F^{(0)}}_{\infty}\leq 2\big(1+\enVert[1]{u_x^{(0)}}_{\infty}\big),$$ 
	we see that \eqref{g(t)bound} and \eqref{f(t1)bound1} yield
	\begin{equation}\label{f(t1)bound2}
		f(t_1)\leq 2f(t_0)+24c\big(1+\enVert[1]{u_x^{(0)}}_{\infty}\big)^3(t_1-t_0)^{1/2}\int_0^\infty(t_1-s)^{1/2}\big(f(s)+h(s)\big),
	\end{equation}
	upon factoring out $\big(1+3\enVert[1]{u_x^{(0)}}_{\infty}^2\big)$, and estimating $\big(1+3\enVert[1]{u_x^{(0)}}_{\infty}^2\big)\leq 3\big(1+\enVert[1]{u_x^{(0)}}_\infty\big)^2$.
	From \eqref{f(t1)bound2}, we see that Proposition \ref{lem:L2timeL1space:m} applies with $A:=2$, and $B:=24c\big(1+\enVert[1]{u_x^{(0)}}_{\infty}\big)^3\delta^{1/2}$.
	
	Now, choose $\delta>0$ small enough such that
	\begin{equation}
		2\delta^{1/2}B<\frac{1}{4}\iff\delta<\frac{1}{192c\big(1+\enVert[1]{u_x^{(0)}}_\infty\big)^3}.
	\end{equation}
	Then choose $\lambda>\frac{1}{\delta}\ln(2)$ large enough such that
	\begin{equation}
		1-Ae^{-\lambda\delta}>\frac{1}{2}\iff\lambda>\frac{2}{\delta}\ln(2)\iff\lambda>C_0.
	\end{equation}
	As a result,
	\begin{equation}
		\frac{2\delta^{1/2}B}{1-Ae^{-\lambda \delta}}<\frac{1}{2},
	\end{equation}
	and by Proposition \ref{lem:L2timeL1space:m} we obtain
	\begin{align}
		\int_0^Te^{-\lambda t}\Bigg(\int_0^\infty|\mu|\dif x\Bigg)^2\dif t&\leq\frac{3}{1+3\enVert[1]{u_x^{(0)}}_\infty^2}\int_0^Te^{-\lambda t}\enVert[1]{\Phi}_{L_x^1}^2\dif t
		\\[4pt]
		&\quad+\frac{12}{1+3\enVert[1]{u_x^{(0)}}_{\infty}^2}\int_0^Te^{-\lambda t}\int_0^\infty w_x^2m^{(0)}\dif x\dif t\nonumber
		\\[4pt]
		&\leq\frac{3\enVert[1]{\Phi}_{L_t^\infty(L_x^1)}^2}{C_0\big(1+3\enVert[1]{u_x^{(0)}}_\infty^2\big)}\nonumber
		\\[4pt]
		&\quad+\frac{12}{1+3\enVert[1]{u_x^{(0)}}_{\infty}^2}\int_0^Te^{-\lambda t}\int_0^\infty w_x^2m^{(0)}\dif x\dif t.\nonumber
	\end{align}
	The second inequality is obtained by the simple estimate
	\begin{equation}
		\int_0^Te^{-\lambda t}\dif t=\frac{1}{\lambda}(1-e^{-\lambda T})<\frac{1}{\lambda}<\frac{1}{C_0}.
	\end{equation}
	Thus, \eqref{L2timeLxLambdaBig} holds.
	
	Now, in the case that $T$ is finite, we can estimate for all $\lambda\in(0,C_0]$,
	\begin{equation}
		e^{-\lambda t}=e^{(2C_0-\lambda)t}e^{-2C_0t}\leq e^{(2C_0-\lambda)T}e^{-2C_0t}.
	\end{equation}
	Consequently, using \eqref{L2timeLxLambdaBig} with $\lambda=2C_0$, we obtain
	\begin{align}
		\int_0^Te^{-\lambda t}\Bigg(\int_0^\infty|\mu|\dif x\Bigg)^2\dif t&\leq e^{(2C_0-\lambda)T}\int_0^Te^{-2C_0 t}\Bigg(\int_0^\infty|\mu|\dif x\Bigg)^2\dif t
		\\[4pt]
		&\leq e^{(2C_0-\lambda)T}\Bigg(\frac{C_1\enVert[1]{\Phi}_{L_t^\infty(L_x^1)}^2}{C_0}+C_2\int_0^T e^{-2C_0 t}\int_0^\infty w_x^2m^{(0)}\dif x\dif t\Bigg)\nonumber
		\\[4pt]
		&\leq\frac{C_1\enVert[1]{\Phi}_{L_t^\infty(L_x^1)}^2e^{2C_0T}}{C_0}+C_2e^{2C_0T}\int_0^T e^{-\lambda t}\int_0^\infty w_x^2m^{(0)}\dif x\dif t,\nonumber
	\end{align}
	whereby the last inequality follows, because $e^{-2C_0t}\leq e^{-\lambda t}$ for all $t\in[0,T]$. This completes the proof.
\end{proof}

\subsection{An abstract H\"older estimate}

\label{sec:abstract holder}

In this subsection we will give an abstract result on H\"older regularity for a parabolic equation with Dirichlet boundary conditions and bounded coefficients.
Our argument is in the spirit of \cite[Lemma 3.2.2.]{cardaliaguet2019master}, but we cover the case of an unbounded domain.
The result stated below is also meant to allow for a possibly infinite time horizon, though in the present work we will not exploit this.
\begin{lemma} \label{lem:uniform t Holder ux}
	Fix $0 < \alpha < 1$.
	Let $u$ be a solution of 
	\begin{equation}
		\label{eq:linear hj}
		\dpd{u}{t} + \lambda u - \frac{\sigma^2}{2}\dpd[2]{u}{x} + V(x,t)\dpd{u}{x} = F, \quad u(0,t) = 0, \quad u(x,0) = u_0(x)
	\end{equation}
	where $\lambda$ is any positive constant, $V$ and $F$ are a bounded continuous functions, and $u_0 \in \s{C}^{1+\alpha}_\diamond(\s{D})$ (i.e.~$u_0 \in \s{C}^{1+\alpha}(\overline{\s{D}})$ and $u_0(0) = 0$).
	Then
	\begin{equation} \label{eq:uniform t Holder ux}
		\enVert{u}_{\s{C}^{\alpha,\alpha/2}\del{\overline{\s{D}}  \times \intcc{0,T}}} + \enVert{\dpd{u}{x}}_{\s{C}^{\alpha,\alpha/2}\del{\overline{\s{D}} \times \intcc{0,T}}} \leq C\del{\enVert{V}_\infty,\alpha,\lambda}\del{\enVert{F}_\infty + \enVert{u_0}_{\s{C}^{1+\alpha}}},
	\end{equation}
	where $C\del{\enVert{V}_\infty,\alpha,\lambda}$ is independent of $T$.
	
	As a corollary, we have
	\begin{equation} \label{eq:Calpha t Calpha x}
		\enVert{u}_{\s{C}^{\alpha/4}\del{[0,T];\s{C}^{1+\alpha/2}(\overline{\s{D}})}} \leq C\del{\enVert{V}_\infty,\alpha,\lambda}\del{\enVert{F}_\infty + \enVert{u_0}_{\s{C}^{1+\alpha}}}.
	\end{equation}
	\begin{comment}
	\begin{equation} 
	\begin{aligned}
	\sup_t \del{\enVert{u(\cdot,t)}_{\s{C}^{\alpha}} +  \enVert{\dpd{u}{x}(\cdot,t)}_{\s{C}^{\alpha}}} &\leq C\del{\enVert{V}_0,\alpha,\lambda}\del{\enVert{F}_0 + \enVert{u_0}_{\s{C}^{1+\alpha}}},\\
	\sup_{t_1\neq t_2} \frac{\enVert{u(\cdot,t_1)-u(\cdot,t_2)}_{\s{C}^{\alpha}} +  \enVert{\dpd{u}{x}(\cdot,t_1)-\dpd{u}{x}(\cdot,t_2)}_{\s{C}^{\alpha}}}{\abs{t_1-t_2}} &\leq C\del{\enVert{V}_0,\alpha,\lambda}\del{\enVert{F}_0 + \enVert{u_0}_{\s{C}^{1+\alpha}}},\\
	\end{aligned}.
	\end{equation}
	\end{comment}
\end{lemma}

\begin{proof}
	We will start by assuming $u_0 = 0$.
	By a standard application of the maximum principle we have that $\enVert{u}_0 \leq \frac{1}{\lambda}\enVert{F}_0$.
	Let $\phi(x)$ be a smooth function, and observe that
	\begin{equation} \label{eq:linear hj1}
		\pd{(\phi u)}{t} - \frac{\sigma^2}{2}\pd[2]{(\phi u)}{x} + \del{V + \frac{\sigma^2\phi'}{\phi}}\pd{(\phi u)}{x} + \lambda(\phi u) = g(x,t),
	\end{equation}
	where
	\begin{equation}
		g(x,t) = \phi(x)F(x,t) + \del{\frac{\sigma^2\del{\phi'(x)}^2}{\phi(x)} - \frac{\sigma^2}{2}\phi''(x) + V(x,t)\phi'(x)}u(x,t).
	\end{equation}
	Fix a time $\tau > 0$, and set $v(x,t) = e^{\lambda(t-\tau)}\phi(x)u(x,t)$.
	Then \eqref{eq:linear hj1} becomes
	\begin{equation} \label{eq:linear hj2}
		\pd{v}{t} - \frac{\sigma^2}{2}\pd[2]{v}{x} + \del{V + \frac{\sigma^2\phi'}{\phi}}\pd{v}{x} = \tilde g,
	\end{equation}
	where
	\begin{equation}
		\tilde g(x,t) = e^{\lambda(t-\tau)}g(x,t).
	\end{equation}
	Fix any $a > 0$ and let $\phi(x) = \del{1+(x-a)^2}^{-1/2}$.
	Note that $\phi$ satisfies the following properties:
	\begin{equation}
		\abs{\frac{\phi'}{\phi}} \leq 1, \ \abs{\frac{\phi''}{\phi}} \leq 2,
		\quad 
		\del{\int_{-\infty}^\infty \abs{\phi(x)}^p\dif x}^{1/p} = \del{\frac{2p}{p-1}}^{1/p} \ \forall p > 1,
	\end{equation}
	and also $\del{\int_0^\tau e^{p\lambda(t-\tau)}\dif t}^{1/p} \leq (\lambda p)^{-1/p}$.
	It follows that
	\begin{equation}
		\enVert{\tilde g}_{L^p(\s{D} \times (0,\tau))} \leq C(p)\lambda ^{-1/p}\del{\enVert{F}_0 + \del{1 +\enVert{V}_\infty}\enVert{u}_0}
		\leq C(p)\lambda ^{-1/p}\enVert{F}_0\del{1 + \enVert{V}_\infty},
	\end{equation}
	where $C(p)$ remains bounded as $p\to\infty$.
	
	Take $p$ arbitrarily large.
	By the potential estimates in \cite[Section IV.3]{ladyzhenskaia1968linear} we have an estimate of the form
	\begin{multline}
		\enVert{\dpd{v}{t}}_{L^p(\s{D} \times (0,\tau))} + \enVert{\dpd[2]{v}{x}}_{L^p(\s{D} \times (0,\tau))} \leq C\enVert{\tilde g - \del{V + \frac{2\phi'}{\phi}}\dpd{v}{x}}_{L^p(\s{D} \times (0,\tau))}\\
		\leq C(p)\lambda ^{-1/p}\enVert{F}_0\del{1 + \enVert{V}_0} + C\del{1+\enVert{V}_0}\enVert{\dpd{v}{x}}_{L^p(\s{D} \times (0,\tau))},
	\end{multline}
	where the constants do not depend on $\tau$.
	On the other hand we have
	\begin{equation}
		\enVert{v}_{L^p(\s{D} \times (0,\tau))}
		\leq C(p)\lambda ^{-1/p}\enVert{u}_0 \leq C(p)\lambda ^{-1-1/p}\enVert{F}_0.
	\end{equation}
	By interpolation (see e.g.~\cite[Lemma II.3.3]{ladyzhenskaia1968linear}) we have
	\begin{equation}
		\enVert{\dpd{v}{x}}_{L^p(\s{D} \times (0,\tau))}
		\leq \delta \del{\enVert{\dpd{v}{t}}_{L^p(\s{D} \times (0,\tau))} + \enVert{\dpd[2]{v}{x}}_{L^p(\s{D} \times (0,\tau))}} + \frac{C}{\delta}\enVert{v}_{L^p(\s{D} \times (0,\tau))}
	\end{equation}
	for some constant $C$, where $\delta > 0$ is sufficiently small.
	Choosing $\delta > 0$ small enough, we deduce
	\begin{multline}
		\enVert{\dpd{v}{t}}_{L^p(\s{D} \times (0,\tau))} + \enVert{\dpd[2]{v}{x}}_{L^p(\s{D} \times (0,\tau))}  
		\leq C(p)\lambda ^{-1/p}\enVert{F}_0\del{1 + \enVert{V}_\infty} + C\del{1+\enVert{V}_\infty^2}\enVert{v}_{L^p(\s{D} \times (0,\tau))}\\
		\leq C(p)\lambda ^{-1/p}\enVert{F}_0\del{1 + \enVert{V}_\infty} + C\del{1+C(p)\lambda ^{-1-1/p}\enVert{F}_0\enVert{V}_\infty^2}
	\end{multline}
	Then by a Sobolev type embedding theorem \cite[Lemma II.3.3]{ladyzhenskaia1968linear} we have
	\begin{equation}
		\enVert{v}_{\s{C}^{\alpha,\alpha/2}} + \enVert{\dpd{v}{x}}_{\s{C}^{\alpha,\alpha/2}} \leq C(p,\lambda )\enVert{F}_0\del{1 + \enVert{V}_\infty^2} + C
	\end{equation}
	for $\alpha = 1-\frac{3}{p}$ (assuming $p > 3$).
	We can rewrite this as
	\begin{equation}
		\enVert{e^{\lambda(\cdot - \tau)}\phi u}_{\s{C}^{\alpha,\alpha/2}} + \enVert{e^{\lambda(\cdot - \tau)}\del{\phi' u + \phi\dpd{u}{x}}}_{\s{C}^{\alpha,\alpha/2}} \leq C(p,\lambda )\enVert{F}_0\del{1 + \enVert{V}_\infty^2} + C.
	\end{equation}
	Since $t \mapsto e^{\lambda t}$ is locally Lipschitz with constant depending on $\lambda$, we can write this as
	\begin{equation}
		\enVert{\phi u}_{\s{C}^{\alpha,\alpha/2}(\overline{\s{D}} \times [\tau-1,\tau])} + \enVert{\phi' u + \phi\dpd{u}{x}}_{\s{C}^{\alpha,\alpha/2}(\overline{\s{D}} \times [\tau-1,\tau])} \leq C(p,\lambda )\enVert{F}_0\del{1 + \enVert{V}_\infty^2} + C(\lambda).
	\end{equation}
	Using the fact that $\frac{\phi'}{\phi}$ is bounded by 1 and is globally Lipschitz with $\Lip\del{\frac{\phi'}{\phi}} = 1$, we also have
	\begin{equation}
		\enVert{\phi\dpd{u}{x}}_{\s{C}^{\alpha,\alpha/2}}
		\leq \enVert{\phi'u+\phi\dpd{u}{x}}_{\s{C}^{\alpha,\alpha/2}}
		+ \enVert{\phi'u}_{\s{C}^{\alpha,\alpha/2}}
		\leq \enVert{\phi'u+\phi\dpd{u}{x}}_{\s{C}^{\alpha,\alpha/2}}
		+ \enVert{\phi u}_{\s{C}^{\alpha}},
	\end{equation}
	hence
	\begin{equation}
		\enVert{\phi u}_{\s{C}^{\alpha,\alpha/2}(\overline{\s{D}} \times [\tau-1,\tau])}
		+ \enVert{\phi\dpd{u}{x}}_{\s{C}^{\alpha,\alpha/2}(\overline{\s{D}} \times [\tau-1,\tau])}
		\leq C(p,\lambda )\enVert{F}_0\del{1 + \enVert{V}_\infty^2} + C(\lambda).
	\end{equation}
	Now, since $\frac{1}{\phi}$ is globally Lipschitz with $\Lip(1/\phi) = 1$, and bounded on $[a-1,a+1]$ with an upper bound of 2, we deduce that
	\begin{equation}
		\enVert{u}_{\s{C}^{\alpha,\alpha/2}\del{[a-1,a+1]  \times [\tau-1,\tau]}} + \enVert{\dpd{u}{x}}_{\s{C}^{\alpha,\alpha/2}\del{[a-1,a+1] \times [\tau-1,\tau]}} \leq C(p,\lambda )\enVert{F}_0\del{1 + \enVert{V}_\infty^2} + C(\lambda).
	\end{equation}
	This estimate is independent of $\tau$ and $a$.
	Letting $\tau$ and $a$ vary through all the positive integers, and since $p$ can be determined through $\alpha$, it follows that
	\begin{equation} \label{eq:uniform t Holder ux1}
		\enVert{u}_{\s{C}^{\alpha,\alpha/2}\del{\overline{\s{D}}  \times \intco{0,\infty}}} + \enVert{\dpd{u}{x}}_{\s{C}^{\alpha,\alpha/2}\del{\overline{\s{D}} \times \intco{0,\infty}}} \leq C(\alpha,\lambda )\enVert{F}_0\del{1 + \enVert{V}_\infty^2} + C(\lambda).
	\end{equation}
	\begin{comment}
	In particular,
	\begin{equation} 
	\begin{aligned}
	\sup_t \del{\enVert{u(\cdot,t)}_{\s{C}^{\alpha}\del{\overline{\s{D}}}} + \enVert{\dpd{u(\cdot,t)}{x}}_{\s{C}^{\alpha}\del{\overline{\s{D}}}}} &\leq C(\alpha,\lambda)\enVert{F}_0\del{1 + \enVert{V}_0^2} + C,\\
	\sup_{t_1 \neq t_2} \frac{\enVert{u(\cdot,t_1)-u(\cdot,t_2)}_{\s{C}^{\alpha}\del{\overline{\s{D}}}} + \enVert{\dpd{u(\cdot,t_1)}{x}-\dpd{u(\cdot,t_2)}{x}}_{\s{C}^{\alpha}\del{\overline{\s{D}}}}}{\abs{t_1-t_2}} &\leq C(\alpha,\lambda)\enVert{F}_0\del{1 + \enVert{V}_0^2} + C
	\end{aligned}
	\end{equation}
	\end{comment}
	
	We now remove the assumption that $u_0 = 0$.
	Let $w$ be the solution of
	\begin{equation} \label{eq:w-global in time holder}
		\pd{w}{t} - \frac{\sigma^2}{2}\pd[2]{w}{x} + \lambda w = 0, \quad w(x,0) = u_0(x).
	\end{equation}
	As $\lambda > 0$, by the maximum principle, we have $\enVert{w}_0 \leq \enVert{u_0}_0$.
	Then $\intcc{w}_{\s{C}^{\alpha,\alpha/2}} \leq C\enVert{u_0}_{\s{C}^\alpha}$ by \cite[Theorem 10.1]{ladyzhenskaia1968linear}; this estimate does not depend on time because of the global in time $L^\infty$ bound.
	This establishes $\enVert{w}_{\s{C}^{\alpha,\alpha/2}} \leq C\enVert{u_0}_{\s{C}^{\alpha}}$.
	We then take the derivative in $x$ of Equation \eqref{eq:w-global in time holder} and apply the same argument as above to $\dpd{w}{x}$ to establish $\enVert{\dpd{w}{x}}_{\s{C}^{\alpha,\alpha/2}} \leq C\enVert{u_0}_{\s{C}^{1+\alpha}}.$
	Then let $\hat u$ be the solution of
	\begin{equation} 
		\pd{\hat u}{t} - \frac{\sigma^2}{2}\pd[2]{\hat u}{x} + V(x,t)\pd{\hat u}{x} + \lambda \hat u = F(x,t) - V(x,t)\pd{w}{x}
	\end{equation}
	with zero initial conditions.
	Then by \eqref{eq:uniform t Holder ux1} we have
	\begin{equation}
		\begin{aligned}
			\enVert{\hat u}_{\s{C}^{\alpha,\alpha/2}\del{\overline{\s{D}}  \times \intco{0,\infty}}} + \enVert{\dpd{\hat u}{x}}_{\s{C}^{\alpha,\alpha/2}\del{\overline{\s{D}} \times \intco{0,\infty}}} &\leq C(\alpha,\lambda,\enVert{V}_\infty)\enVert{F+V\dpd{w}{x}}_0\\
			&\leq C\del{\enVert{V}_\infty,\alpha,\lambda }\del{\enVert{F}_0 + \enVert{u_0}_{\s{C}^{1+\alpha}}}.
		\end{aligned}
	\end{equation}
	As $u = \hat u + w$ is the solution to \eqref{eq:linear hj}, the claim is proved.
	
	Finally, to prove \eqref{eq:Calpha t Calpha x}, note that \eqref{eq:uniform t Holder ux} immediately implies
	\begin{equation} \label{eq:Calpha t Calpha x1}
		\sup_{0\leq t\leq T} \enVert{u(\cdot,t)}_{\s{C}^{1+\alpha/2}(\overline{\s{D}})} \leq \sup_{0\leq t\leq T} \enVert{u(\cdot,t)}_{\s{C}^{1+\alpha}(\overline{\s{D}})}
		\leq C\del{\enVert{V}_\infty,\alpha,\lambda}\del{\enVert{F}_\infty + \enVert{u_0}_{\s{C}^{1+\alpha}}},
	\end{equation}
	and we also have
	\begin{equation} \label{eq:Calpha t Calpha x2}
		\begin{split}
			\sup_{t \neq s} \sup_{x \neq y}&\frac{\abs{\pd{u}{x}(x,t) - \pd{u}{x}(x,s) - \pd{u}{x}(y,t) + \pd{u}{x}(y,s)}}{\abs{t-s}^{\alpha/4}\abs{x-y}^{\alpha/2}}\\
			& \leq \sup_{t \neq s} \sup_{x \neq y}\frac{2\enVert{\pd{u}{x}}_{\s{C}^{\alpha,\alpha/2}}\min\cbr{\abs{t-s}^{\alpha/2},\abs{x-y}^{\alpha}}}{\abs{t-s}^{\alpha/4}\abs{x-y}^{\alpha/2}}
			\\
			&\leq C\del{\enVert{V}_\infty,\alpha,\lambda}\del{\enVert{F}_\infty + \enVert{u_0}_{\s{C}^{1+\alpha}}}.
		\end{split}
	\end{equation}
	Combine \eqref{eq:Calpha t Calpha x1} and \eqref{eq:Calpha t Calpha x2} to get \eqref{eq:Calpha t Calpha x}.
\end{proof}

\section{Nonlocal Existence Lemma}\label{sec:nonlocalExistenceLemmas}

For this lemma, recall the definition $\s{X} = \s{C}^{\alpha,\alpha/2} \cap L_t^\infty(L_x^1) \cap \s{C}^{\alpha/2}\del{[0,T];(\s{C}^{1+2\alpha})^*}$
with norm
\begin{equation}
	\enVert{m}_{\s{X}} :=
	\enVert{m}_{\s{C}^{\alpha,\alpha/2}} +
	\enVert{m}_{L_t^\infty(L_x^1)} +
	 \enVert{m}_{\s{C}^{\alpha/2}\del{[0,T];(\s{C}^{1+2\alpha})^*}}.
\end{equation}
\begin{lemma} \label{lem:existence Holder}
	Let $\alpha \in (0,1/2)$ be such that $u^{(0)},m^{(0)} \in \s{C}^{2+2\alpha,1+\alpha}$.
	Let $m \in \s{X}$, $u_T \in \s{C}^{2+\alpha}(\intco{0,\infty})$, and $\Psi\in\s{C}^{\alpha,\alpha/2}(\intco{0,\infty} \times [0,T])$ be given.
	Consider the backward parabolic equation
\begin{equation}\label{eq:nonlocal}
	\begin{cases}
		\begin{array}{lll}
			(i)&\displaystyle u_t+\frac{\sigma^2}{2}u_{xx}-ru + \lambda\Psi+\lambda F^{(0)}(\varepsilon)\big(G(u_x,m;\varepsilon)-u_x\big)=0, \text{\hspace{.5cm}} &0\leq x< \infty,\,\,0\leq t\leq T\vspace{4pt}\\
			(ii)&\displaystyle u(x,T)=u_T(x), &0\leq x< \infty\vspace{4pt}\\
			(iii)&\displaystyle u(0,t)=0, &0\leq t\leq T,
		\end{array}
	\end{cases}
\end{equation}
where $G$ is defined as in \eqref{def:G}.
We assume the usual compatibility conditions of first order:
\begin{equation}
	u_T(0) = 0, \quad \frac{\sigma^2}{2}u_T''(0) + \lambda \Psi(0,T) - \lambda F^{(0)}(\epsilon)(0,0)u_T'(0) = 0.
\end{equation}
Then there exists a unique classical solution $u \in \s{C}^{2+\alpha,1+\alpha/2}(\intco{0,\infty} \times [0,T])$ to the boundary value problem \eqref{eq:nonlocal}.
Moreover, the following estimate holds:
\begin{equation}
\label{eq:parabolic estimate Holder}
\enVert{u}_{\s{C}^{2+\alpha,1+\alpha/2}} \leq C\del{\enVert{m}_{\s{X}} + \enVert{\Psi}_{\s{C}^{\alpha/2,\alpha}} + \enVert{u_T}_{\s{C}^{2+\alpha}}}
\end{equation}
where $C$ depends only on the data.
\end{lemma}

\begin{proof}
	Let $u \in C^{2+\alpha,1+\alpha/2}(\intco{0,\infty} \times [0,T])$ be given.
	Note that
	\begin{equation}
		\enVert{G(u_x,m;\varepsilon)}_{\s{C}^{\alpha,\alpha/2}} \leq C\del{\enVert{m}_{\s{C}^{\alpha/2}\del{[0,T];(\s{C}^{1+2\alpha})^*}} + \enVert{u}_{\s{C}^{2+\alpha,1+\alpha/2}}},
	\end{equation}
	where $C$ depends on the $\s{C}^{2+2\alpha,1+\alpha}$ norms of $u^{(0)}$ and $m^{(0)}$.
By \cite[Theorem IV.5.2]{ladyzhenskaia1968linear}, there exists a unique solution $w \in C^{2+\alpha,1+\alpha/2}(\intco{0,\infty} \times [0,T])$ the boundary value problem
\begin{equation}\label{eq:nonlocal fixed pt}
	\begin{cases}
		\begin{array}{lll}
			(i)&\displaystyle w_t+\frac{\sigma^2}{2}w_{xx}-rw+\lambda\Psi+\lambda F^{(0)}(\varepsilon)\big(G(u_x,m;\varepsilon)-w_x\big)=0, \text{\hspace{.5cm}} &0\leq x< \infty,\,\,0\leq t\leq T\vspace{4pt}\\
			(ii)&\displaystyle w(x,T)=u_T(x), &0\leq x< \infty\vspace{4pt}\\
			(iii)&\displaystyle w(0,t)=0. &0\leq t\leq T,
		\end{array}
	\end{cases}
\end{equation}
and it satisfies the estimate
\begin{equation}
	\enVert{w}_{\s{C}^{2+\alpha,1+\alpha/2}} \leq C\del{\enVert{\Psi}_{\s{C}^{\alpha,\alpha/2}} + \enVert{G(u_x,m;\varepsilon)}_{\s{C}^{\alpha,\alpha/2}}
	+ \enVert{u_T}_{\s{C}^{2+\alpha}}}.
\end{equation}
We will denote $w = F(u)$.
Our goal it to show that $F$ is a contraction on a suitably defined metric space, and use this to prove \eqref{eq:nonlocal} has a solution.

Let $u_1,u_2 \in \s{C}^{2+\alpha,1+\alpha/2}(\intco{0,\infty} \times [0,T])$, then define $u = u_1 - u_2$ and $w = F(u_1) - F(u_2)$.
Note that $w$ satisfies
\begin{equation}\label{eq:nonlocal difference}
	\begin{cases}
		\begin{array}{lll}
			(i)&\displaystyle w_t+\frac{\sigma^2}{2}w_{xx}-rw+\lambda F^{(0)}(\varepsilon)\big(G(u_x,0;\varepsilon)-w_x\big)=0, \text{\hspace{.5cm}} &0\leq x< \infty,\,\,0\leq t\leq T\vspace{4pt}\\
			(ii)&\displaystyle w(x,T)=0, &0\leq x< \infty\vspace{4pt}\\
			(iii)&\displaystyle w(0,t)=0. &0\leq t\leq T.
		\end{array}
	\end{cases}
\end{equation}
Recalling the definition of $G$ in \eqref{def:G}, we see that there is a constant $C_1$, depending only on the data, such that
\begin{equation}
	\enVert{G(u_x,0;\varepsilon)}_{\s{C}^{\alpha,\alpha/2}} \leq C_1\enVert{u_x}_{\s{C}^{\alpha,\alpha/2}}.
\end{equation}
Combining this with classical estimates from \cite[Theorem IV.5.2]{ladyzhenskaia1968linear}, we see that there is some constant $C_2$, depending only on the data,  such that
\begin{equation}
	\enVert{w}_{\s{C}^{2+\alpha,1+\alpha/2}} \leq C_2\enVert{u_x}_{\s{C}^{\alpha,\alpha/2}}.
\end{equation}
Using interpolation on H\"older spaces, we deduce that there exists a constant $C_3$, depending only on $C_2$, such that
\begin{equation}
	C_2\enVert{v_x}_{\s{C}^{\alpha,\alpha/2}} \leq \frac{1}{4}\enVert{v}_{\s{C}^{2+\alpha,1+\alpha/2}} + C_3\enVert{v}_\infty \quad \forall v \in \s{C}^{2+\alpha,1+\alpha/2}.
\end{equation}
Hence 
\begin{equation}
	\enVert{w}_{\s{C}^{2+\alpha,1+\alpha/2}} \leq \frac{1}{4}\enVert{u}_{\s{C}^{2+\alpha,1+\alpha/2}} + C_3\enVert{u}_\infty.
\end{equation}
Next, we define $\tilde w(x,t) = e^{r(T-t)}w \pm r^{-1}(e^{r(T-t)} - 1)C_1\enVert{u_x}_{\s{C}^{\alpha,\alpha/2}}$, which satisfies
\begin{equation}
	\pm (-\tilde w_t - \frac{\sigma^2}{2} \tilde w_{xx} - \lambda F^{(0)}(\varepsilon)\tilde w_x) \leq 0.
\end{equation}
Apply the maximum principle to $\tilde w$ to deduce that
\begin{equation}
	\abs{w(x,t)} \leq r^{-1}(e^{r(T-t)} - 1)C_1\enVert{u_x}_{\s{C}^{\alpha,\alpha/2}}
	\leq (T-t)e^{rT}C_1\enVert{u}_{\s{C}^{2+\alpha,1+\alpha/2}}.
\end{equation}
We now apply all of these estimates only on the time interval $[T-\tau,T]$ for some $\tau > 0$.
We deduce that
\begin{multline}
	\enVert{w}_{\s{C}^{2+\alpha,1+\alpha/2}(\intco{0,\infty} \times [T-\tau,T])} 
	+ 2C_3\enVert{w}_{L^\infty(\intco{0,\infty} \times [T-\tau,T])}
	\\
	\leq \del{\frac{1}{4} + 2C_3\tau e^{rT}C_1} \enVert{u}_{\s{C}^{2+\alpha,1+\alpha/2}(\intco{0,\infty} \times [T-\tau,T])} + C_3\enVert{u}_{L^\infty(\intco{0,\infty} \times [T-\tau,T])}.
\end{multline}
We now set $\tau = \dfrac{1}{8C_3 e^{rT}C_1}$.
Define $\s{Y}_\tau$ to be the space $\s{C}^{2+\alpha,1+\alpha/2}(\intco{0,\infty} \times [T-\tau,T])$ endowed with the norm
\begin{equation}
	\enVert{w}_{\s{Y}_\tau} := \enVert{w}_{\s{C}^{2+\alpha,1+\alpha/2}(\intco{0,\infty} \times [T-\tau,T])} 
	+ 2C_3\enVert{w}_{L^\infty(\intco{0,\infty} \times [T-\tau,T])}.
\end{equation}
Observe that $\s{Y}_\tau$ is a Banach space.
Moreover, by the above estimates, $F:\s{Y}_\tau \to \s{Y}_\tau$ is a contraction, since $\enVert{F(u_1) - F(u_2)}_{\s{Y}_\tau} \leq \frac{1}{2}\enVert{u_1 - u_2}_{\s{Y}_\tau}$.
Hence $F$ has a unique fixed point $u$, which is a solution to \eqref{eq:nonlocal} and satisfies the estimate \eqref{eq:parabolic estimate Holder} on the time interval $[T-\tau,T]$.
However, $T$ is arbitrary.
We can now partition the interval $[0,T]$ into subintervals that are each at most $\tau$ in length, i.e.~$0 = t_0 < t_1 < \cdots < t_N = T$ where $t_{j+1} - t_j \leq \tau$.
Apply the same argument on each subinterval $[t_{j-1},t_j]$, replacing the final condition $u_T(x)$ with $w(x,t_j)$, for each $j$ starting with $N$ and going down to 1.
(Cf.~the proof of \cite[Proposition 3.11]{cirant2019existence}.)
In this way we obtain a solution $u$ to Equation \eqref{eq:nonlocal}, which indeed satisfies \eqref{eq:parabolic estimate Holder}.
Uniqueness of this solution follows from uniqueness on each subinterval.

\end{proof}


\bibliographystyle{siam}
\bibliography{C:/mybib/mybib}
\end{document}